\documentclass[11pt]{amsart}
\usepackage{amsmath,amssymb,amscd,amsfonts,graphicx,latexsym,epsfig,psfrag,url,color,xr,mathtools,enumitem,transparent,fullpage,hyperref,float,mathrsfs,wrapfig,rotating,array}
\usepackage[all]{xy}
\usepackage[foot]{amsaddr}

\newtheorem{theorem}{Theorem}[section]

\newtheorem{proposition}[theorem]{Proposition}
\newtheorem{lemma}[theorem]{Lemma}

\newtheorem{deflem}[theorem]{Definition-Lemma}

\theoremstyle{definition}
\newtheorem{definition}[theorem]{Definition}
\theoremstyle{remark}
\newtheorem{remark}[theorem]{Remark}

\newtheorem{example}[theorem]{Example}
\theoremstyle{plain}
\newcommand{\thistheoremname}{}
\newtheorem{genericthm}[theorem]{\thistheoremname}

\newtheorem*{genericthm*}{\thistheoremname}
\newenvironment{namedthm*}[1]
  {\renewcommand{\thistheoremname}{#1}%
   \begin{genericthm*}}
  {\end{genericthm*}}
  
\newcommand\cM{\mathcal{M}}

\newcommand{\bR}{\mathbb{R}}

\newcommand{\bZ}{\mathbb{Z}}

\newcommand\ba{\mathbf{a}}
\newcommand\bb{\mathbf{b}}
\newcommand\bc{\mathbf{c}}
\newcommand\bm{\mathbf{m}}
\newcommand\bn{\mathbf{n}}

\newcommand\bzero{\mathbf{0}}

\newcommand{\btB}{{\mathbf{2B}}}
\newcommand{\sB}{\mathscr{B}}
\newcommand{\stB}{2\mathscr{B}}

\newcommand{\on}{\operatorname}

\newcommand\pt{{\on{pt}}}

\newcommand{\Fuk}{\on{Fuk}}
\newcommand{\comp}{C^2}

\renewcommand{\comp}{{\on{comp}}}
\newcommand{\seam}{{\on{seam}}}
\newcommand{\mk}{{\on{mark}}}
\newcommand{\incom}{\on{in}}

\newcommand{\inte}{{\on{int}}}

\renewcommand{\root}{{\on{root}}}

\newcommand{\rk}{{\on{rk}\:}}
\newcommand{\dist}{{\on{dist}}}
\newcommand{\tree}{{\on{tree}}}
\newcommand{\br}{{\on{br}}}
\newcommand{\cl}{\mathrm{cl}}
\newcommand{\out}{{\on{out}}}

\renewcommand{\top}{{\on{top}}}

\newcommand\qu{/\kern-.7ex/} 
\newcommand\lqu{\backslash \kern-.7ex \backslash}

\newcommand{\ol}{\overline}

\newcommand{\sr}{\stackrel}
\newcommand{\wh}{\widehat}
\newcommand{\wt}{\widetilde}

\newcommand{\eps}{\epsilon}

\def\hra{\hookrightarrow}
\def\lra{\longrightarrow}

\newcounter{qcounter}

\newcommand\quotient[2]{
        \mathchoice
            {
                \text{\raise1ex\hbox{$#1$}\Big/\lower1ex\hbox{$#2$}}%
            }
            {
                #1\,/\,#2
            }
            {
                #1\,/\,#2
            }
            {
                #1\,/\,#2
            }
    }

\newcommand\quoti[2]{
                \text{\raise1ex\hbox{$#1$}/\lower1ex\hbox{$\scriptstyle#2$}}
  }

\newcommand\quot[2]{
                \text{\raise1ex\hbox{$#1\!\!$}/\lower1ex\hbox{$\!\scriptstyle#2$}}
  }

\newcommand\quo[2]{
                \text{\raise.8ex\hbox{$\scriptstyle#1\!$}/\lower.8ex\hbox{$\!\scriptstyle#2$}}
  }

\newcommand\qq[2]{
                \text{\raise.8ex\hbox{$#1\!$}/\lower.8ex\hbox{$#2$}}
}

\begin{document}

\title{2-associahedra}
\author{Nathaniel Bottman}
\address{School of Mathematics, Institute for Advanced Study,
1 Einstein Dr, Princeton, NJ 08540}
\email{\href{mailto:nbottman@math.ias.edu}{nbottman@math.ias.edu}}

\maketitle

\begin{abstract}  
For any $r\geq 1$ and $\bn \in \bZ_{\geq0}^r \setminus \{\mathbf0\}$ we construct a poset $W_\bn$ called a \emph{2-associahedron}.
The 2-associahedra arose in symplectic geometry, where they are expected to control maps between Fukaya categories of different symplectic manifolds.
We prove that the completion $\wh{W_\bn}$ is an abstract polytope of dimension $|\bn|+r-3$.
There are forgetful maps $W_\bn \to K_r$, where $K_r$ is the $(r-2)$-dimensional associahedron, and the 2-associahedra specialize to the associahedra (in two ways) and to the multiplihedra.
In an appendix, we work out the 2- and 3-dimensional 2-associahedra in detail.
\end{abstract}

\setcounter{tocdepth}{1}
\tableofcontents

\section{Introduction}
\label{sec:intro}

Ma'u--Wehrheim--Woodward proposed in \cite{mww} that a Lagrangian correspondence $M_1 \sr{L_{12}}{\lra} M_2$ between symplectic manifolds should induce an $A_\infty$-functor $\Fuk(M_1) \sr{\Phi(L_{12})}{\lra} \Fuk(M_2)$ between Fukaya categories, where $\Phi(L_{12})$ is defined by counting pseudoholomorphic quilted disks.
In his thesis \cite{b:thesis}, the author suggested extending this proposal by counting witch balls, which are pseudoholomorphic quilted spheres.
The underlying domain moduli spaces (described in \S\ref{ss:motivation}) are stratified topological spaces, and the posets indexing the strata have interesting combinatorial structure, similarly to how the Stasheff polytopes have strata indexed by
the associahedra.
In this paper we define these underlying posets, which we call \emph{2-associahedra}.

In this introduction, we will first motivate the construction of the 2-associahedra by describing some bubbling phenomena in spaces of witch curves, which are the moduli spaces of domains of witch balls.
After that, we will give a plan for the body of the paper.

\subsection{Motivation for \texorpdfstring{$W_\bn$}{Wn} from witch curves}
\label{ss:motivation}

The author was led to define and study the 2-associahedra
$(W_\bn)$
because they are the posets corresponding to degenerations in the moduli space $\ol{2\cM}_\bn$ of \emph{witch curves}
--- similarly to how the compactified moduli space of $r$ points on the line modulo translations and dilations is stratified by the associahedron $K_r$.
$\ol{2\cM}_\bn$
is defined in \cite{b:realization}, crucially relying on the current paper's construction of $W_\bn$.
Here we sketch the definition of $\ol{2\cM}_\bn$ to motivate the definition of
$W_\bn$.

We begin by defining the uncompactified moduli space:
\begin{align}
2\cM_\bn \coloneqq \left\{
\begin{array}{rcl}
(x_1,\ldots,x_r) &\in& \bR^r \\
(y_{11},\ldots,y_{1n_1}) &\in& \bR^{n_1} \\
&\vdots& \\
(y_{r1},\ldots,y_{rn_r}) &\in& \bR^{n_r}
\end{array}
\:\left|\:
\begin{array}{c}
x_1 < \cdots < x_r \\
y_{11} < \cdots < y_{1n_1} \\
\vdots \\
y_{r1} < \cdots < y_{rn_r}
\end{array}
\right.\right\}\Big/\bR^2 \rtimes \bR_{>0} \eqqcolon X/G.
\end{align}
We view an element of $X$ as describing a configuration of $r$ vertical lines in $\bR^2$ with $x$-positions $x_1, \ldots, x_r$, along with $n_i$ marked points on the $i$-th line with $y$-positions $y_{i1}, \ldots, y_{in_i}$.
(By identifying $\bR^2 \cup \{\infty\} \simeq S^2$, we can also view an element of $X$ as a configuration of marked circles on $S^2$, where all the circles intersect tangentially at the south pole.)
We view $G$ as the group of affine-linear transformations of the plane which consist of a translation and a positive dilation, which we can extend to define an action of $G$ on $X$.

\begin{figure}[H]
\centering
\includegraphics[width=0.4\columnwidth]{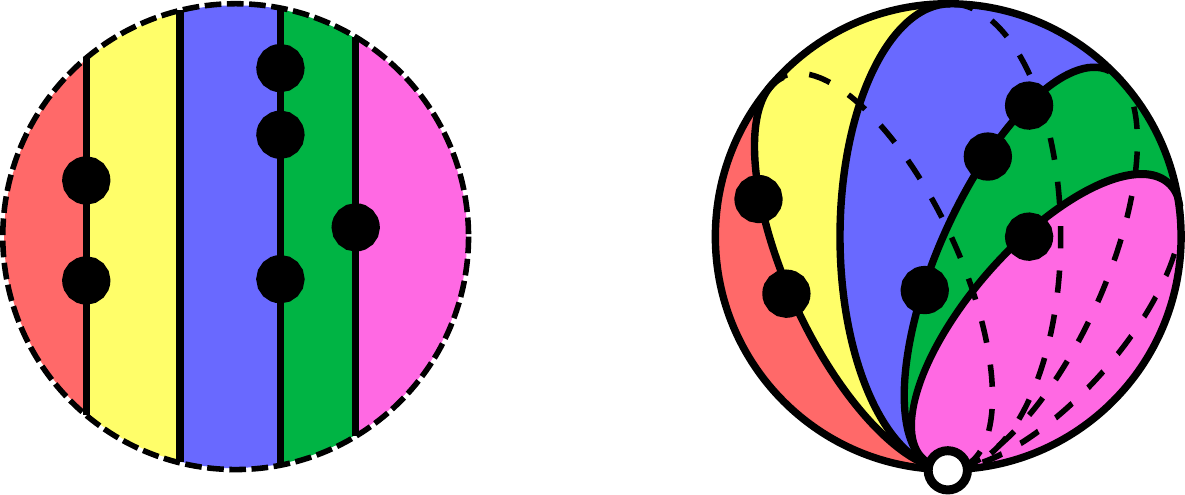}
\caption{Two views of a witch curve in the main stratum $2\cM_\bn \subset \ol{2\cM}_\bn$: on the left, as $\bR^2$ with marked lines; and on the right, as $S^2$ with marked circles.}
\label{fig:witch_ball}
\end{figure}

$2\cM_\bn$ is not compact: points on a single line can collide, or lines can collide, or these two phenomena can take place simultaneously.
We compactify $2\cM_\bn$ to $\ol{2\cM}_\bn$ like so: when a collection of lines collide, then wherever the marked points on these lines are as this collision happens, we bubble off another configuration of lines and points.
To define $\ol{2\cM}_\bn$ precisely, we need to specify the allowed degenerations, and this is where the 2-associahedra come in: the elements of $W_\bn$ correspond to the allowed degenerations in $\ol{2\cM}_\bn$.
We illustrate this in the following figure: on the left is the compactified moduli space $\ol{2\cM}_{200}$, and in the middle and on the right are two presentations of $W_{200}$.

\begin{figure}[H]
\centering
\def\svgwidth{1.0\columnwidth}
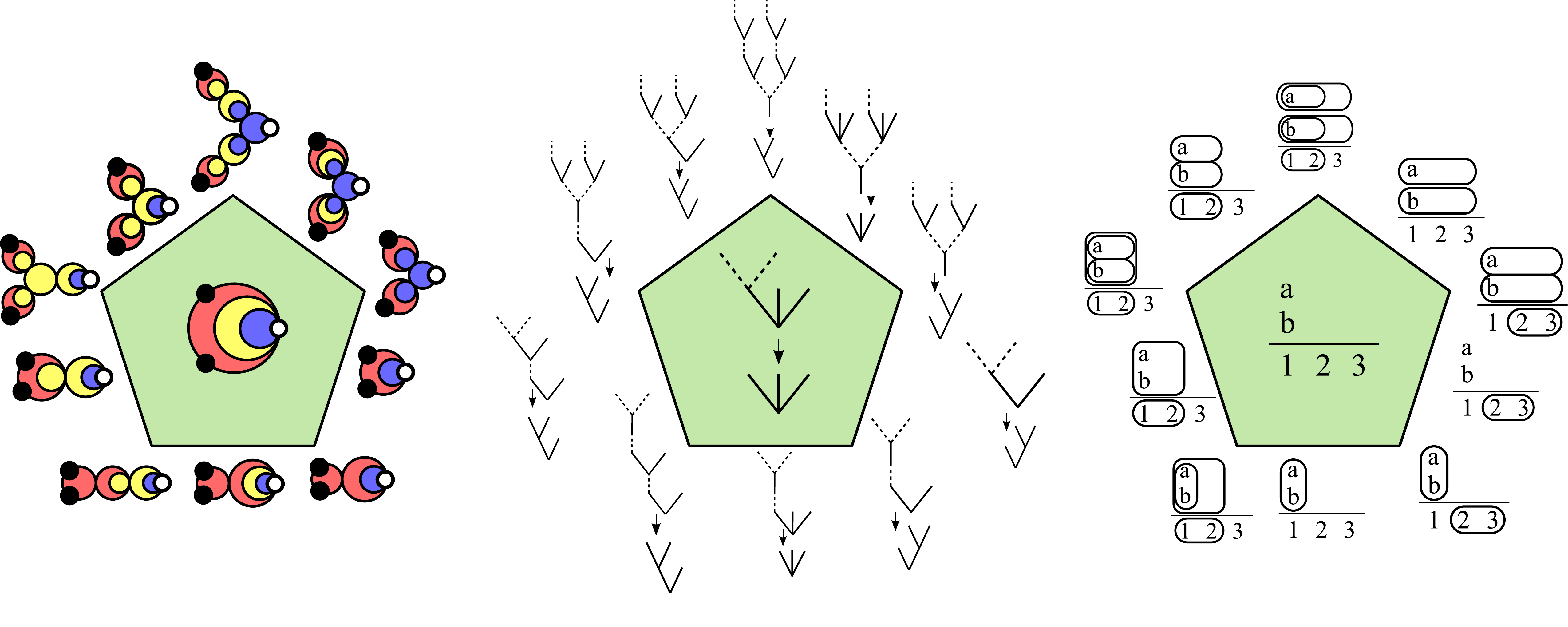
\end{figure}

\noindent
On the left, we are identifying configurations of lines and points in the plane with configurations of vertical lines in the right half-plane, which can in turn be identified with configurations of circles on a disk, all of which intersect tangentially at a point on the boundary.
In this figure, and throughout this paper, we overlay the posets $W_{200}^\tree$ and $W_{200}^\br$ over polytopes.
The set of faces of any polytope has a poset structure, where $F<G$ if the containment $F \subset \ol G$ holds; our depiction of $W_{200}^\tree$ and $W_{200}^\br$ indicates that they are isomorphic to the face poset of the pentagon.



$W_\bn$ is intended to index the possible degenerations that a sequence of points in the moduli space $\ol{2\cM}_\bn$ can undergo.
In
\S\ref{sec:2ass}
we will define two models for $W_\bn$: $W_\bn^\tree$ and $W_\bn^\br$.
To approach and motivate these models, consider the following degenerations in $\ol{2\cM}_{200}$, corresponding to the bottom resp.\
bottom-left
resp.\ upper-right edges in the depiction of $\ol{2\cM}_{200}$ above:

\begin{figure}[H]
\centering
\def\svgwidth{0.4\columnwidth}
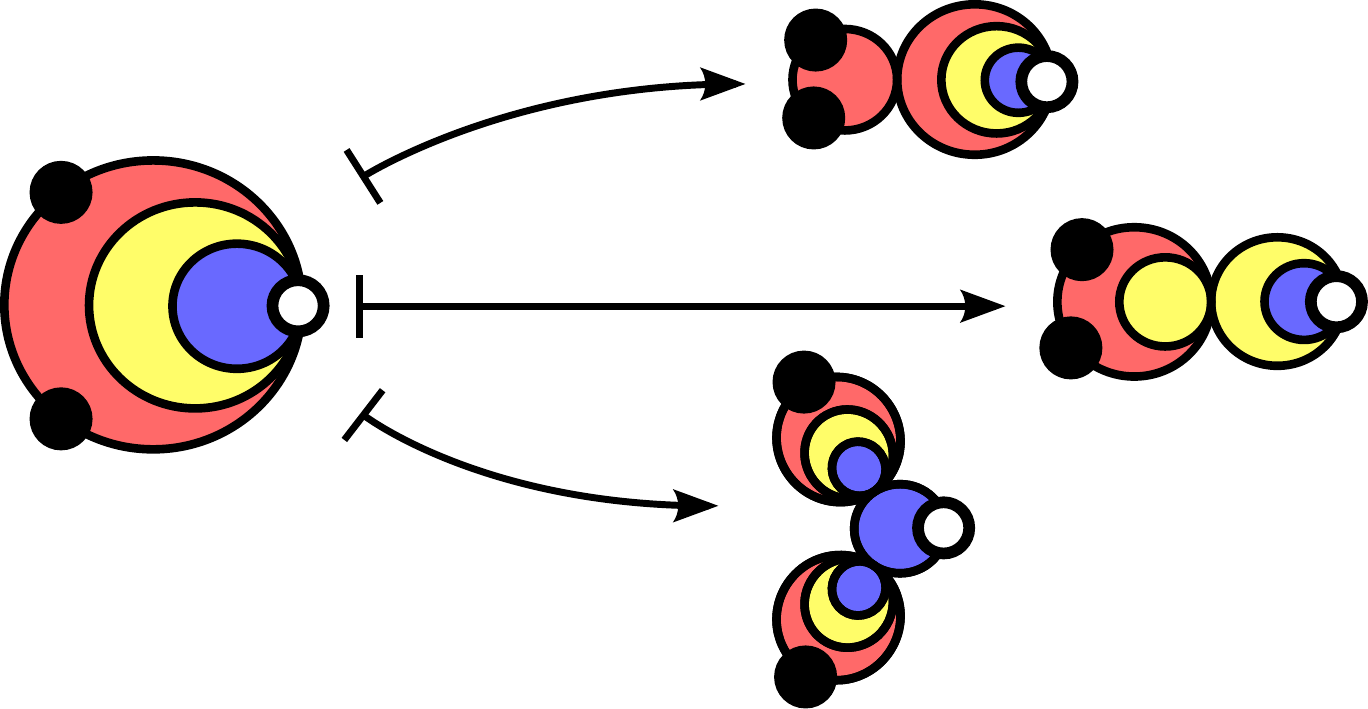
\end{figure}

\noindent
Degeneration 1 occurs when the two black points collide; degeneration 2 occurs when the larger interior circle expands and collides with the boundary circle, while the black points simultaneously collide; and degeneration 3 occurs when the two interior circles simultaneously expand and collide with the boundary circle.
To define the 2-associahedra, we must produce combinatorial data that track these degenerations.
We can do so in two ways:
\begin{itemize}
\item Represent each
disk
as a vertex in a tree, with solid edges corresponding to the seams (i.e., boundary circle or interior circles) appearing on that
disk.
We represent an attachment between two
disks
as a dashed edge; we also represent a marked point by a dashed edge.
This leads to the model $W_\bn^\tree$, and in this model the degenerations pictured above take the following form:

\begin{figure}[H]
\centering
\def\svgwidth{0.3\columnwidth}
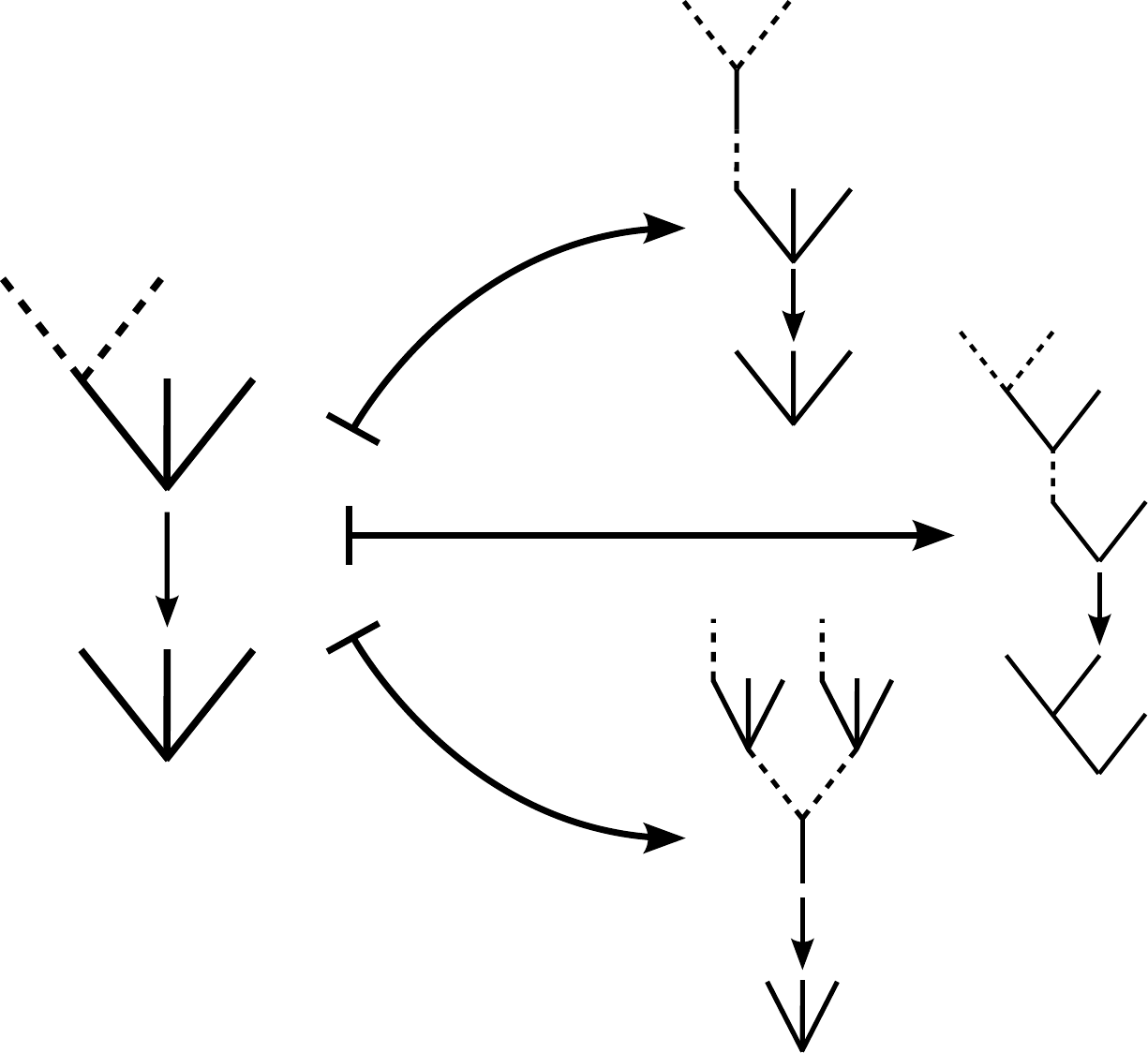
\end{figure}

\noindent The reader will observe that a single datum is not only a tree with solid and dashed edges, but also a smaller, solid, tree which receives a map from the larger tree.
The reason for this is that when
disks
bubble off, the same seam may appear in multiple
disks.
These seams must remain linked, so that the enlargement $\ol{2\cM}_\bn$ is a reasonable compactification of $2\cM_\bn$, and this linking is enforced by the smaller tree and the map it receives.

\item Represent the seams as a horizontal line of numbers; above each number, represent the points that appear on that seam as a vertical line of letters.
For any given
disk
$C$ in the bubble tree, form a subtree consisting of the
disks
that can only be reached from the main component by passing through $C$; the datum corresponding to $C$ is a grouping including those marked points appearing in this subtree.
Such a grouping is called a {\it 2-bracket}, and every 2-bracket comes with a ``width'', which indicates the seams that appear on the corresponding
disk.
This leads to the model $W_\bn^\br$, and in this model the above degenerations take the following form:

\begin{figure}[H]
\centering
\def\svgwidth{0.35\columnwidth}
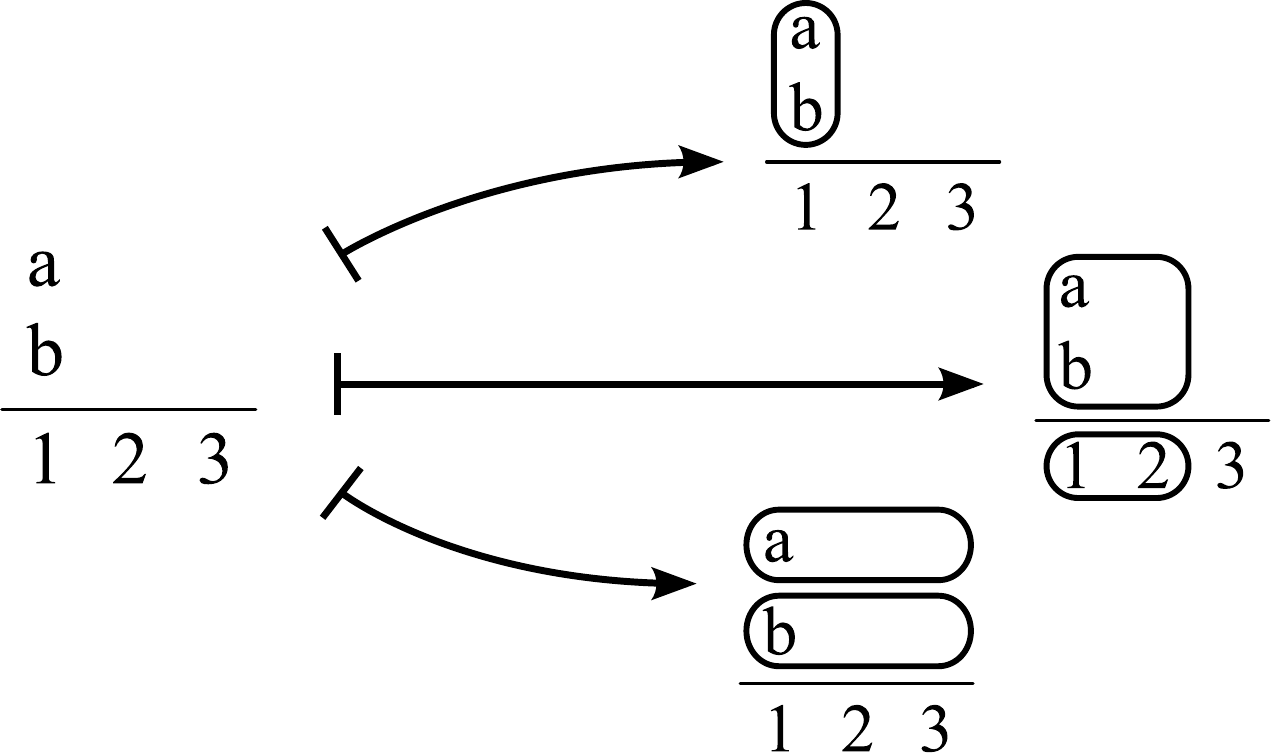
\end{figure}
\end{itemize}

\subsection{Plan}

The constructions in this paper are rather technical, and with the exception of \S\ref{sec:ass}, our definitions and results are completely new.
For this reason, we give a plan of the paper to orient the reader.

\medskip

\noindent {\bf\S\ref{sec:ass}:}
We recall two equivalent constructions,
called $K_r^\tree$ and $K_r^\br$,
of the associahedra $K_r$, along with several basic properties.
This material is not new, but these particular constructions of $K_r$ are needed for the constructions of the 2-associahedron $W_\bn$ in \S\ref{sec:2ass}.
In addition, the constructions of $K_r^\tree$ and $K_r^\br$ and the proofs of Prop.~\ref{prop:Kr_iso} and Prop.~\ref{prop:Kr_main} are analogous to the constructions of $W_\bn^\tree$ and $W_\bn^\br$ and the proofs of Thms.~\ref{thm:iso} and \ref{thm:main}, and so will serve as an introduction to \S\S\ref{sec:2ass}--\ref{sec:Wn_polytope}.
\begin{enumerate}
\item[\bf\S\ref{ss:Kr_construction}:]
In Def.~\ref{def:Krtree_set} and Def.-Lem.~\ref{deflem:Krtree_poset} we define a poset, $K_r^\tree$, consisting of rooted ribbon trees with $r$ leaves.
Then, in Def.~\ref{def:Krbr}, we define the poset $K_r^\br$, consisting of 1-bracketings of $r$ letters.
We prove that the posets $K_r^\tree$ and $K_r^\br$ are isomorphic in Prop.~\ref{prop:Kr_iso}, and define $K_r \coloneqq K_r^\tree = K_r^\br$.

\medskip

\item[\bf\S\ref{ss:Kr_polytope}:]
We establish two important properties of $K_r$, collected in the following result:

\medskip

\begin{samepage}
\noindent
{\bf Proposition \ref{prop:Kr_main}} (Key properties of $K_r$){\bf.}
\label{prop:Kr_main}
The posets $(K_r)$ satisfy the following properties:
\begin{itemize}
\item[] \textsc{(abstract polytope)} For $r \geq 2$, $\wh{K_r} \coloneqq K_r \cup \{F_{-1}\}$ is an abstract polytope of dimension $r-2$.

\item[] \textsc{(recursive)} For any $T \in K_r^\tree$, there is an inclusion of posets
\begin{align}
\gamma_T\colon \prod_{\alpha \in T_\inte} K_{\#\!\incom(\alpha)}^\tree \hra K_r^\tree,
\end{align}
which restricts to a poset isomorphism onto $\cl(T) = (F_{-1},T]$.
\end{itemize}
\end{samepage}

\medskip

\noindent We now give brief explanations of these properties.
\begin{enumerate}
\item[]
\textsc{(abstract polytope)}:
As explained in Def.~\ref{def:abstract_polytope}, an abstract polytope is a poset satisfying some of the characteristic combinatorial properties of a convex polytope.

\smallskip

\item[]
\textsc{(recursive)}:
This property reflects the fact that if $S$ is a stratum in $\ol\cM_\bn$ corresponding to a configuration with several disk-components, then the degenerations that can take place in $\ol S$ correspond to a choice of a degeneration (or lack thereof) in each of the disk-components.
We depict one of the maps $\gamma_T$ in the following figure (using $\ol\cM_\bn$ rather than $K_r^\tree$ or $K_r^\br$ for clarity):

\begin{figure}[H]
\label{fig:operad}
\centering
\def\svgwidth{0.6\columnwidth}
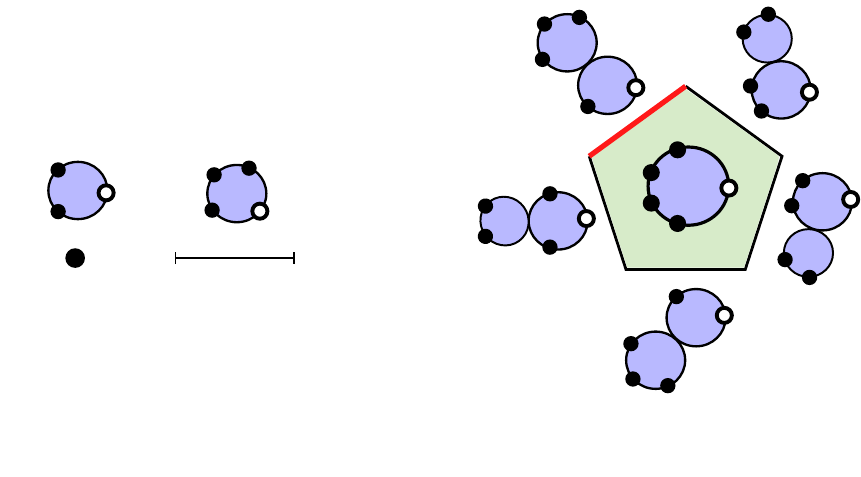
\end{figure}

\noindent
Here the upper-left edge in $K_4$ is decomposed as the product of $K_2$ and $K_3$, corresponding to the two disk-components appearing in the label on the upper-left edge.
These maps give $(K_r)$ the structure of an operad, which is of fundamental importance for applications to symplectic geometry.
\end{enumerate}
\end{enumerate}

\medskip

\noindent {\bf\S\ref{sec:2ass}:}
In this section, we construct the posets $W_\bn^\tree$ and $W_\bn^\br$ and show that they are isomorphic.
\begin{enumerate}
\item[\bf\S\ref{ss:Wntree_construction}, \bf\S\ref{ss:Wnbr_construction}:]
Here we define the posets $W_\bn^\tree$ and $W_\bn^\br$ (see Def.~\ref{def:Wn_tree} and Def.~\ref{def:Wn_br}), which were motivated in \S\ref{ss:motivation}.
We also show in Lemma~\ref{lem:WnKn} that $W_\bn^\tree$ specializes to the associahedra and the multiplihedra, which will be important for future applications to symplectic geometry.

\medskip

\item[\bf\S\ref{ss:Wn_iso}:]
This subsection is devoted to the proof of the following theorem:

\medskip
\noindent
{\bf Theorem \ref{thm:iso}} (Equivalence of the two models for $W_\bn$).
For any $r\geq 1$ and $\bn \in \bZ_{\geq0}^r\setminus\{\bzero\}$, $W_\bn^\tree$ and $W_\bn^\br$ are isomorphic posets.
\medskip

\noindent
With this theorem in hand, we define $W_\bn \coloneqq W_\bn^\tree = W_\bn^\br$.
We also define a forgetful map $\pi\colon W_\bn \to K_r$, which has a simple definition in either model for $W_\bn$: $\pi^\tree\colon W_\bn^\tree \to K_r^\tree$ sends a tree-pair $T_b \to T_s$ to the seam tree $T_s$, and $\pi^\br\colon W_\bn^\br \to K_r^\br$ sends a 2-bracketing to the underlying 1-bracketing of $1\:2\:\cdots\: r$.
In the following figure we depict $\pi^\br\colon W_{200}^\br \to K_3^\br$:

\begin{figure}[H]
\centering
\def\svgwidth{0.65\columnwidth}
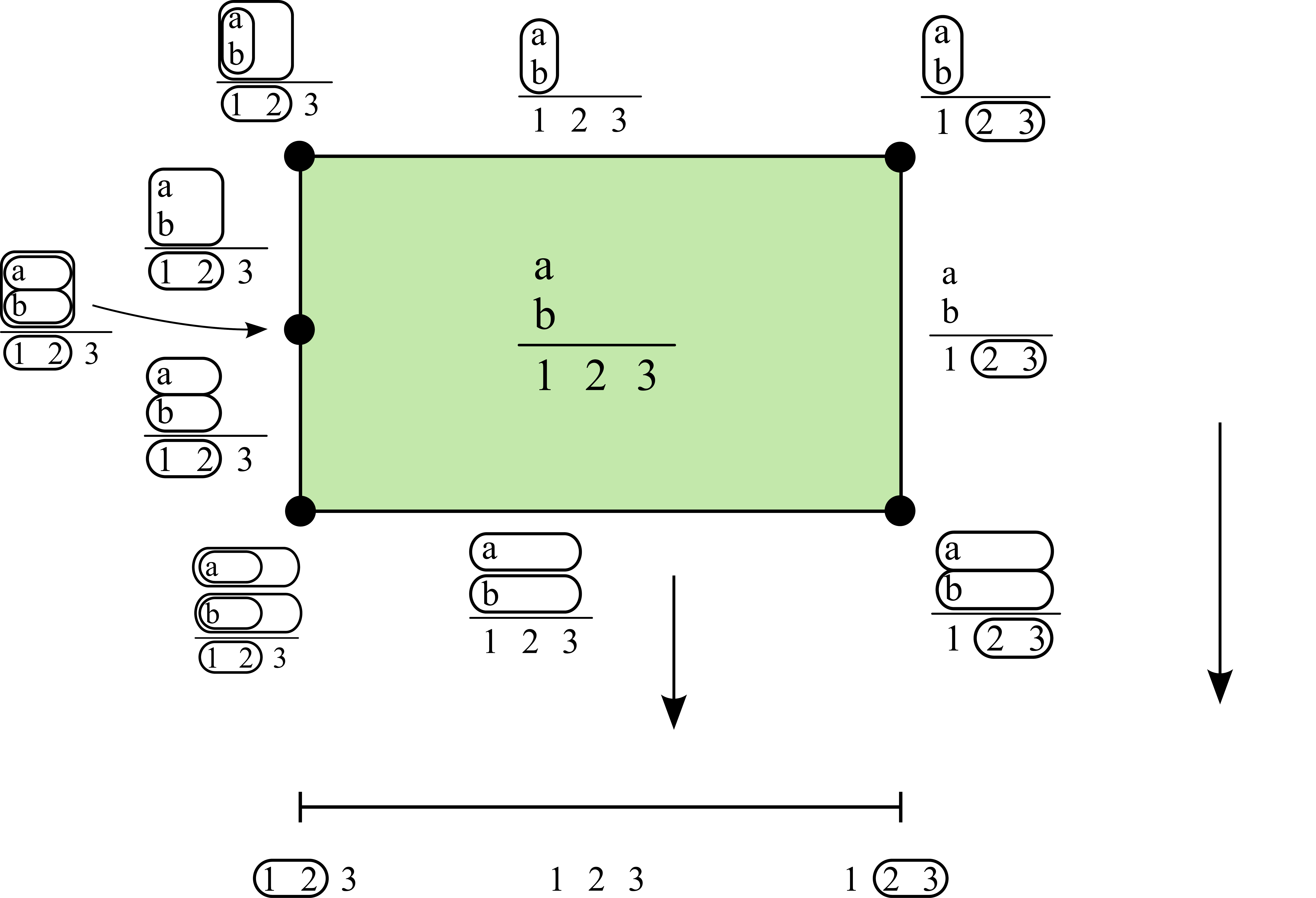
\end{figure}

\noindent
The forgetful map provides an important connection between the 2-associahedra and the associahedra.
Together with the \textsc{(recursive)} property described below, the forgetful map endows $(W_\bn)$ with the structure of a {\textit{relative 2-operad}}, a notion which the author plans to describe in a forthcoming paper.

\medskip

\noindent{\bf\S\ref{sec:Wn_polytope}:}
This section is devoted to the proof of several properties of $W_\bn$ which we collect in this paper's main theorem:

\medskip

\noindent
{\bf Theorem \ref{thm:main}} (Key properties of $W_\bn$).
For any $r \geq 1$ and $\bn \in \bZ^r_{\geq0}\setminus\{\bzero\}$, the 2-associahedron $W_\bn$ is a poset, the collection of which satisfies the following properties:
\begin{itemize}
\item[] \textsc{(abstract polytope)} For $\bn \neq (1)$, $\wh{W_\bn} \coloneqq W_\bn \cup \{F_{-1}\}$ is an abstract polytope of dimension $|\bn| + r - 3$.

\item[] \textsc{(forgetful)} $W_\bn$ is equipped with forgetful maps $\pi\colon W_\bn \to K_r$, which are surjective maps of posets.

\item[] \textsc{(recursive)} For any stable tree-pair $2T = T_b \sr{f}{\to} T_s \in W_\bn^\tree$, there is an inclusion of posets
\begin{align}
\Gamma_{2T} \colon \prod_{
{\alpha \in V_\comp^1(T_b),}
\atop
{\incom(\alpha)=(\beta)}
} W_{\#\!\incom(\beta)}^\tree
\times
\prod_{\rho \in V_\inte(T_s)} \prod^{K_{\#\!\incom(\rho)}}_{
{\alpha\in V_\comp^{\geq2}(T_b)\cap f^{-1}\{\rho\},}
\atop
{\incom(\alpha)=(\beta_1,\ldots,\beta_{\#\!\incom(\rho)})}
}
\hspace{-0.25in} W^\tree_{\#\!\incom(\beta_1),\ldots,\#\!\incom(\beta_{\#\!\incom(\alpha)})}
\hra W_\bn^\tree,
\end{align}
where the superscript on one of the product symbols indicates that it is a fiber product with respect to the maps described in \textsc{(forgetful)}.
This inclusion is a poset isomorphism onto $\cl(2T) = (F_{-1},2T]$.
\end{itemize}

\medskip

\noindent We now make remarks about two of these properties.

\begin{enumerate}
\item[] \textsc{(abstract polytope):}
It seems likely that $\ol{2\cM}_\bn$ can be realized as a convex polytope in a way that identifies its face lattice with $W_\bn$, but this is not important for the author's purposes.

\smallskip

\item[]
\textsc{(recursive)}:
This property is similar to the \textsc{(recursive)} property of $K_r$, but differs in that the closed strata of $W_\bn$ are \emph{fiber} products of lower-dimensional 2-associahedra.
We depict one of the maps $\Gamma_{2T}$ in the following figure (using $\ol{2\cM}_\bn$ rather than $W_\bn^\tree$ or $W_\bn^\br$ for clarity):

\begin{figure}[H]
\centering
\def\svgwidth{0.85\columnwidth}
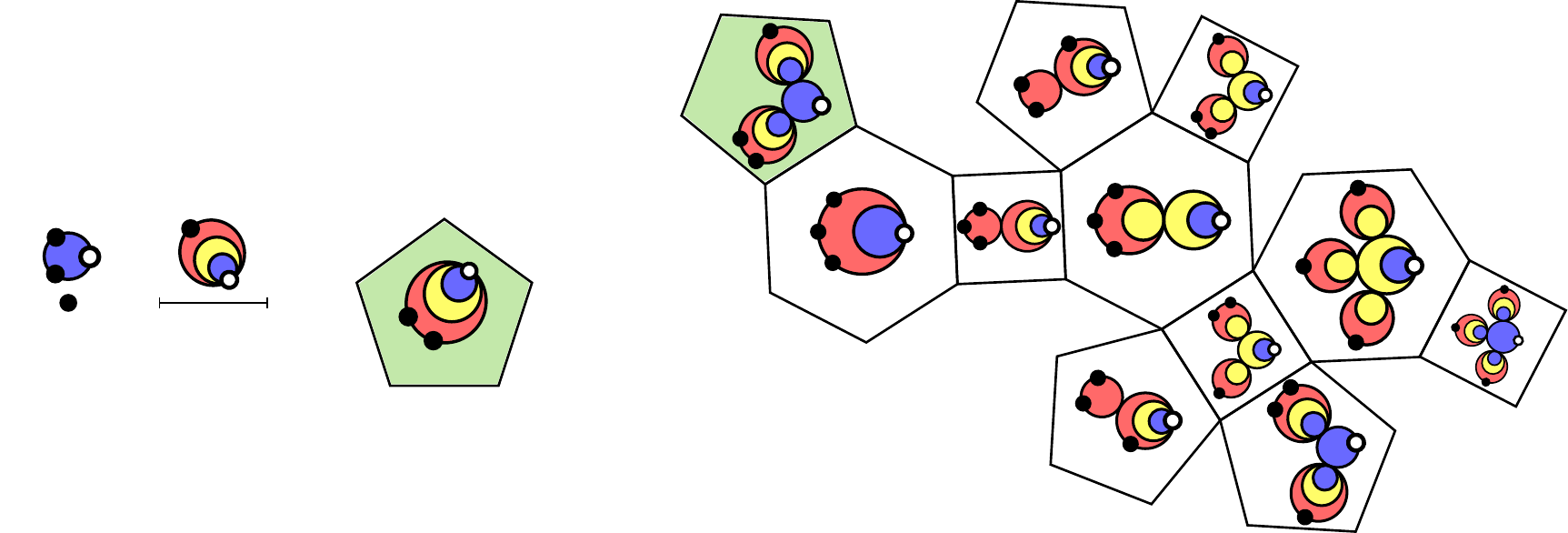
\end{figure}

\noindent
Here the fiber product $W_2 \times W_{100} \times_{K_3} W_{200}$ is included in $W_{300}$ as the green pentagon.
($W_{300}$ is a polyhedron; here we depict its net.)
$K_3$ is the 1-dimensional associahedron, which is an interval, and the maps from $W_{100}$ and $W_{200}$ to $K_3$ measure the width of the yellow strip.
\end{enumerate}

\medskip

\noindent{\bf\S\ref{app:ex}:}
In the appendix, we record all 2- and 3-dimensional 2-associahedra (except those which are isomorphic to associahedra).
One can immediately see from these examples that the 2-associahedra are not trivial extensions of the associahedra (for instance, products of associahedra).
Moreover, these examples are evidence that all 2-associahedra can be realized as convex polytopes.
\end{enumerate}

\subsection{Future directions}

The construction of the 2-associahedra suggests several potential future directions:
\begin{itemize}
\item In
\cite{b:realization}, the author
shows
that the 2-associahedra have modular realizations in terms of witch curves; these realizations are stratified topological spaces.
This fits into a larger project of constructing invariants of collections of Lagrangian correspondences, defined by counting pseudoholomorphic quilts whose domains are witch curves.
More progress toward this goal can be found in \cite{bw:compactness}, where a version of Gromov compactness for these quilts is proven using the analysis in \cite{b:sing}.

\item The associahedra have by now several realizations as convex polytopes, including a realization as the secondary polytopes of certain planar polygons.
It is natural to wonder whether the 2-associahedra can also be realized as convex polytopes, in particular as secondary or fiber polytopes (a possibility suggested by Gabriel Kerr).
Realizability as convex polytopes is not clearly relevant for applications to symplectic geometry, but a realization as secondary or fiber polytopes could suggest additional structure relevant to the study of Fukaya categories.

\item It is natural to ask whether there is a notion of ``$m$-associahedra'' for all $m \geq 1$.
The author believes that this concept should be relatively straightforward to define, but he does not have a need for this generalization and therefore has no plans to investigate this.

\item The author has conjectured that a cellular model for the little 2-disks operad can be built by gluing together copies of $W_{1\cdots 1}$.
If this is true, it suggests a way of defining a homotopy Gerstenhaber structure on symplectic cohomology involving only finitely many operations of any given arity.

\item Analogously to how an $A_\infty$-category is the same thing as a category over the operad of associahedra, the author plans to define a notion of {\it $A_\infty$-2-category} as a 2-category over the relative 2-operad of 2-associahedra.
\end{itemize}

\subsection{Glossary of notation, and conventions}
\label{ss:notation_and_conventions}

\begin{center}
\begin{tabular}{l|l|l}
notation & interpretation & page first defined \\
\hline
$[\alpha,\beta]$ & path from $\alpha$ to $\beta$ in a tree & p.\ \pageref{p:path} \\
$\alpha(\beta,\gamma,\delta)$ & single vertex in $[\beta,\gamma]\cap[\gamma,\delta]\cap[\delta,\beta]$ & p.\ \pageref{p:alpha_triple} \\
$\alpha(B)$ & vertex corresponding to 1-bracket $B$ & p.\ \pageref{p:alpha_B} \\
$\alpha(\btB)$ & vertex corresponding to 2-bracket $\btB$ & p.\ \pageref{p:alpha_2B} \\
$B$ & 1-bracket & p.\ \pageref{p:B} \\
$B(\alpha)$ & 1-bracket corresponding to vertex $\alpha$ & p.\ \pageref{p:Balpha} \\
$\sB$ & 1-bracketing & p.\ \pageref{p:sB} \\
$\btB = (B,(2B_i))$ & 2-bracket & p.\ \pageref{p:btB} \\
$\btB(\alpha)$ & 2-bracket corresponding to vertex $\alpha$ & p.\ \pageref{p:btBalpha} \\
$(\sB,\stB)$ & 2-bracketing & p.\ \pageref{p:sBstB} \\
$K_r$ & associahedron, $K_r=K_r^\br=K_r^\tree$ & p.\ \pageref{p:Kr} \\
$K_r^\br$ & poset of 1-bracketings of $r$ & p.\ \pageref{p:Krbr} \\
$K_r^\tree$ & poset of rooted ribbon trees with $r$ leaves & p.\ \pageref{deflem:Krtree_poset} \\
$\nu$ & isomorphism $K_r^\tree \to K_r^\br$ & p.\ \pageref{p:nu} \\
$2\nu$ & isomorphism $W_\bn^\tree \to W_\bn^\br$ & p.\ \pageref{p:2nu} \\
$\pi$ & forgetful map $W_\bn \to K_r$ & p.\ \pageref{p:forgetful} \\
$T_{\alpha\beta}$ & those vertices $\gamma$ with $[\alpha,\gamma]\ni\beta$ & p.\ \pageref{p:Talphabeta} \\
$2T = T_b \stackrel{f}{\to} T_s$ & stable tree-pair (with bubble tree $T_b$ and seam tree $T_s$) & p.\ \pageref{p:2T} \\
$W_\bn$ & 2-associahedron, $W_\bn = W_\bn^\br = W_\bn^\tree$ & p.\ \pageref{p:Wn} \\
$W_\bn^\br$ & poset of 2-bracketings of $\bn$ & p.\ \pageref{p:Wnbr} \\
$W_\bn^\tree$ & poset of stable tree-pairs of type $\bn$ & p.\ \pageref{p:Wntree}
\end{tabular}
\end{center}

\smallskip

The following conventions apply throughout this paper:

\begin{center}
\fbox{\parbox{1\columnwidth}{
Unless otherwise specified, $r$ will denote a positive integer, and $\bn$ will denote an element of $\bZ_{\geq0}^r\setminus\{\bzero\}$.
By $|\bn|$ we will denote the sum $\sum_{i=1}^r n_i$.
For $X_1,\ldots,X_\ell$ posets, each equipped with a map $f_i\colon X_i \to Y$ to another poset, we define and denote the \emph{fiber product of $X_1,\ldots, X_\ell$ over $Y$} like so:
\begin{align}
\prod_{1\leq i\leq\ell}^Y X_i
\coloneqq
\left\{\left.(x_1,\ldots,x_\ell,y) \in \prod_{1\leq i\leq\ell} X_i \times Y \:\right|\: \forall\: i: f_i(x_i) = y\right\}.
\end{align}
For any poset $X$ and $F, G \in X$, $[F,G]$ will denote the closed interval $\{H \in X \:|\: F \leq H \leq G\}$.
We will similarly denote half-open and open intervals.
If $X$ is a poset with a maximum resp.\ minimum element, we will denote these elements by $F_\top^X$ resp.\ $F_{-1}^X$ and omit the superscripts when the poset is obvious.}}
\end{center}

\noindent Note that when $\ell=0$, the ``empty fiber product over $Y$'' is $Y$ itself, and that when $\ell=1$ and $f_1\colon X_1 \to Y$ is surjective, this fiber product can be identified with $X_1$.

\subsection{Acknowledgments}

The ideas presented in this paper evolved over the course of several years, and the author is grateful to a number of people for their insight and support.
This project began in 2014, while the author was a graduate student at MIT under the supervision of Katrin Wehrheim; the author is grateful for Prof.\ Wehrheim's support throughout this project.
Conversations with Satyan Devadoss greatly helped in the development of the definition of $W_\bn$.
Stefan Forcey pointed out the notion of abstract polytope.
David Feldman, Nick Sheridan, and the anonymous referees
made suggestions that improved the exposition.
The author thanks Mohammed Abouzaid, Helmut Hofer, Jacob Lurie, Paul Seidel, and James Stasheff for encouragement.

The ideas in this paper were developed while the author was a graduate student at MIT, then a postdoctoral researcher at Northeastern University, and finally a member at the Institute for Advanced Study and a postdoctoral researcher at Princeton University.
The author was supported by an NSF Graduate Research Fellowship and an NSF Mathematical Sciences Postdoctoral Research Fellowship.

\section{The definition of \texorpdfstring{$K_r$}{Kr} and some basic properties}
\label{sec:ass}

In this section we define the associahedra $K_r$ and prove the analogues of Thms.~\ref{thm:iso} and \ref{thm:main}.
We will use $K_r$ later in this paper; besides, the ideas in this section will shed light on the techniques we will use to prove Thms.~\ref{thm:iso} and \ref{thm:main}. 

\subsection{Two constructions of \texorpdfstring{$K_r$}{Kr}: in terms of rooted ribbon trees, and in terms of 1-bracketings}
\label{ss:Kr_construction}

In this subsection we will define two posets $K_r^\tree$ and $K_r^\br$, then show that they are isomorphic.
We begin by recalling the definition of a tree.

\begin{definition}
A \emph{tree} is a finite set $T$ and a relation $E \subset T \times T$ satisfying these axioms:

\begin{itemize}
\item[] ({\sc Symmetry}) If $\alpha E \beta$, then $\beta E \alpha$.

\item[] ({\sc Antireflexivity}) If $\alpha E \beta$, then $\alpha \neq \beta$.

\item[] ({\sc Connectedness}) If $\alpha, \beta$ are distinct vertices, then there exist $\gamma_1, \ldots, \gamma_k \in T$ with $\gamma_1 = \alpha$, $\gamma_k = \beta$, and $\gamma_i E \gamma_{i+1}$ for every $i$.

\item[] ({\sc No cycles}) If $\gamma_1, \ldots, \gamma_k$ are vertices with $\gamma_i E \gamma_{i+1}$ and $\gamma_i \neq \gamma_{i+2}$ for all $i$, then $\gamma_1 \neq \gamma_k$.
\end{itemize}
\null\hfill$\triangle$
\end{definition}

\noindent We can now define the model $K^\tree_r$.
\begin{definition}
\label{def:Krtree_set}
A \emph{rooted ribbon tree} (RRT) is a tree $T$ with a choice of a root $\alpha_\root \in T$ and a cyclic ordering of the edges incident to each vertex; we orient such a tree toward the root.
We say that a vertex $\alpha$ of an RRT $T$ is \emph{interior} if the set $\incom(\alpha)$ of its incoming neighbors is nonempty, and we denote the set of interior vertices of $T$ by $T_\inte$.
An RRT $T$ is \emph{stable} if every interior vertex has at least 2 incoming edges.
We define $K^\tree_r$ to be the set of all isomorphism classes of stable rooted ribbon trees with $r$ leaves.

We denote the $i$-th leaf of an RRT $T$ by $\lambda_i^T$.
For any $\alpha, \beta \in T$, $T_{\alpha\beta}$\label{p:Talphabeta} denotes those vertices $\gamma$ such that the path $[\alpha,\gamma]$\label{p:path} from $\alpha$ to $\gamma$ passes through $\beta$.
We denote $T_\alpha \coloneqq T_{\alpha_\root\alpha}$.
\null\hfill$\triangle$
\end{definition}


Here is an illustration of the notation we have just introduced, in the case of a single RRT $T$:

\begin{figure}[H]
\centering
\def\svgwidth{0.8\columnwidth}
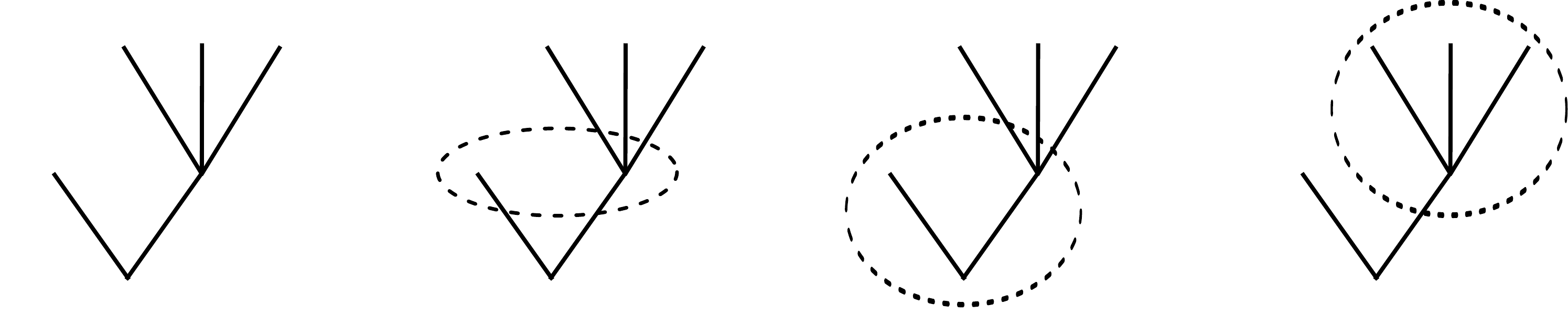
\label{fig:RRT_example}
\end{figure}


The following lemma provides a useful alternate characterization of RRTs.

\begin{lemma}
\label{lem:RRT_in}
An RRT is equivalent to the following data:
\begin{itemize}
\item a finite set $V$ of vertices with a distinguished element $\alpha_\root$;

\item for every $\alpha \in V$, a sequence $\incom(\alpha) \subset V$ such that:
\begin{itemize}
\item[(1)] for every $\alpha \in V$, $\incom(\alpha) \not\ni \alpha_\root$;

\item[(2)] for every $\alpha \neq \alpha_\root$ there exists a unique vertex $\beta$ with $\incom(\beta) \ni \alpha$; and

\item[(3)] if $\alpha_1, \ldots, \alpha_\ell$ is a sequence in $V$ with $\ell \geq 2$ and $\alpha_j \in \incom(\alpha_{j+1})$ for every $j$, then $\alpha_1 \neq \alpha_\ell$.
\end{itemize}
\end{itemize}
Moreover, the RRT is stable if and only if for every $\alpha \in V$, $\#\!\incom(\alpha) \neq 1$.
\end{lemma}

\begin{proof}
\noindent Throughout this proof, we will abbreviate $\alpha \in \incom(\beta)$ by $\alpha \prec \beta$.

\medskip

\noindent {\it{Step 1: Given an RRT, we show that its vertices together with their incoming neighbors satisfy (1--3).}}

\medskip

\noindent Fix an RRT $T$.
Its root $\alpha_\root$ is not an incoming neighbor of any vertex, so (1) holds.
It is also clear that any $\alpha \neq \alpha_\root$ is an incoming neighbor of exactly one vertex, since otherwise the {(\sc{no cycles})} property would not hold; therefore (2) holds.
Finally, if $\alpha_1 \prec \ldots \prec \alpha_\ell$ is a sequence of vertices with $\ell \geq 2$, then $\dist(\alpha_j,\alpha_\root) = \dist(\alpha_{j+1},\alpha_\root) + 1$ for every $j$, so $\alpha_1 \neq \alpha_\ell$.

\medskip

\noindent {\it {Step 2: Given a finite set $V \ni \alpha_\root$ and sequences $\incom(\alpha)$ for $\alpha \in V$ that satisfy (1--3), we produce an RRT having $V$ as its vertices and $\incom(\alpha)$ as the incoming neighbors of $\alpha$, ordered according to the order of $\incom(\alpha)$.}}

\medskip

\noindent Given this data, define a tree $T$ by
\begin{align}
V(T) \coloneqq V, \quad \alpha E \beta \iff \alpha \prec \beta \text{ or } \beta \prec \alpha.
\end{align}
This relation is clearly symmetric, and its antireflexivity follows from the $\ell=2$ case of (3).

To prove {\sc{(connectedness)}}, we will show that every vertex is connected to $\alpha_\root$.
Fix $\alpha \in V \setminus \{\alpha_\root\}$, and define a path like so: set $\alpha_1 \coloneqq \alpha$, and for $j \geq 1$ with $\alpha_j \neq \alpha_\root$, set $\alpha_{j+1}$ to be the unique vertex with $\alpha_j \prec \alpha_{j+1}$.
By (3), this path is nonoverlapping, so since $V$ is finite, this path will eventually terminate at $\alpha_\root$.

To prove {\sc (no cycles)}, consider a path $\alpha_1, \ldots, \alpha_\ell$ with $\ell \geq 3$ and $\alpha_j \neq \alpha_{j+2}$ for every $j$; we must show $\alpha_1 \neq \alpha_\ell$.
If there exists $j$ with $\alpha_j \succ \alpha_{j+1} \prec \alpha_{j+2}$, then (2) implies $\alpha_j = \alpha_{j+2}$, in contradiction to our assumption.
Therefore we must either have (a) $\alpha_1 \prec \cdots \prec \alpha_\ell$, (b) $\alpha_1 \succ \cdots \succ \alpha_\ell$, or (c) $\alpha_1 \prec \cdots \prec \alpha_j \prec \alpha_{j+1} \succ \alpha_{j+2} \succ \cdots \succ \alpha_\ell$.
In cases (a) and (b), (3) implies $\alpha_1 \neq \alpha_\ell$.
In case (c), suppose $\alpha_1 = \alpha_\ell$.
(2) implies that the paths $(\beta_j)$ resp.\ $(\gamma_j)$ defined by $\beta_1 \coloneqq \alpha_1$ and $\beta_{j+1} \succ \beta_j$ resp.\ $\gamma_1 \coloneqq \alpha_\ell$ and $\gamma_{j+1} \succ \gamma_j$ coincide, hence $\alpha_j = \alpha_{j+2}$, a contradiction.
In all cases we have shown $\alpha_1 \neq \alpha_\ell$, hence {\sc (no cycles)} holds.

Finally, we upgrade $V(T)$ to an RRT.
Define its root to be $\alpha_\root \in V$.
With this choice of root, the incoming neighbors of $\alpha$ are exactly the elements of $\in(\alpha)$; order these vertices according to the order on $\in(\alpha)$.

\medskip

\noindent Clearly Steps 1 and 2 are inverse to one another.
The stability criterion is also obvious.
\end{proof}

Now we will define a ``dimension'' function $d$ on $K_r^\tree$.
As described in \S\ref{sec:intro}, $K_r^\tree$ indexes the strata of a topological space $\ol\cM_r$; $d$ assigns to an element of $K_r^\tree$ the dimension of the corresponding stratum of $\ol\cM_r$.

\begin{definition}
\label{def:RRT_dim}
For $T$ a stable RRT in $K_r^\tree$, we define its \emph{dimension} $d(T) \in \bZ_{\geq0}$ like so:
\begin{align}
\label{eq:RRT_dim}
d(T) \coloneqq r - \#\!T_\inte - 1.
\end{align}
\null\hfill$\triangle$
\end{definition}

\begin{remark}
Note that $K_1^\tree$ has a single element, the RRT $\bullet$ with a single vertex and no edges; its dimension is zero.
\null\hfill$\triangle$
\end{remark}

\begin{lemma}
\label{lem:Kr_dim_props}
Fix an RRT $T \in K_r^\tree$.
\begin{itemize}
\item[(a)] The dimension can be re-expressed using this formula:
\begin{equation} \label{eq:T_dim_reform}
	d(T) = \sum_{\alpha \in T_\inte} (\#\!\incom(\alpha)-2).
\end{equation}

\item[(b)] If $r \geq 2$, the dimension satisfies the inequality $0 \leq d(T) \leq r - 2$.
\end{itemize}
\end{lemma}

\begin{proof}
\begin{itemize}
	\item[(a)] \eqref{eq:T_dim_reform} is the result of substituting into \eqref{eq:RRT_dim} the following identity:
	\begin{align}
	\label{eq:T_valence_sum}
	\sum_{\alpha \in T_\inte} \#\!\incom(\alpha) = \#\!T_\inte + r - 1.
	\end{align}
	This follows by noting that $\sum_{\alpha \in T_\inte} \#\!\incom(\alpha)$ counts the vertices in $T$ that are an incoming neighbor of another vertex in $T$.  This set is the complement of the root of $T$, hence has cardinality $\#\!T_\inte + r - 1$.
	
	\item[(b)]	
	The inequality $d(T) \geq 0$ follows from \eqref{eq:T_dim_reform} and the stability hypothesis on $T$; the inequality $d(T) \leq r-2$ follows immediately from \eqref{eq:RRT_dim}.
\end{itemize}
\end{proof}

We now define moves that can be performed on stable RRTs.
As we will show, each move decreases the dimension $d$ by one
--- and in fact, the moves that can be performed on $T$ correspond to the codimension-1 degenerations that can occur in $\ol\cM_r$ starting from the stratum corresponding to $T$.
Given a stable RRT $T$, here is the general description of a legal move that can be performed on $T$: choose $\alpha \in T_\inte$ and a consecutive subset $(\gamma_{p+1},\ldots,\gamma_{p+l})\subset(\gamma_1,\ldots,\gamma_k)=\incom(\alpha)$ where $l$ satisfies $2 \leq l < k$ (necessary to preserve stability).
The corresponding move consists of modifying the incoming edges of $\alpha$ like so:

\begin{figure}[H]
\centering
\def\svgwidth{0.4\columnwidth}
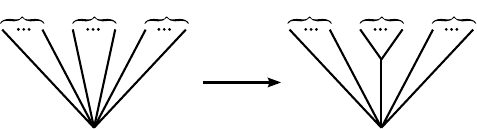
\label{fig:T_moves}
\end{figure}


\begin{example}
In the following figure, we illustrate the notion of a move on an RRT.
On the left resp.\ right we show all the RRTs with three resp.\ four leaves, and indicate all moves amongst these RRTs by arrows, each one corresponding to a single move.
As we will shortly see, these moves equip the set of RRTs with a fixed number of leaves with the structure of a poset, which in fact is an abstract polytope; for this reason we suggestively overlay the RRTs over polytopes.

\begin{figure}[H]
\centering
\def\svgwidth{0.65\columnwidth}
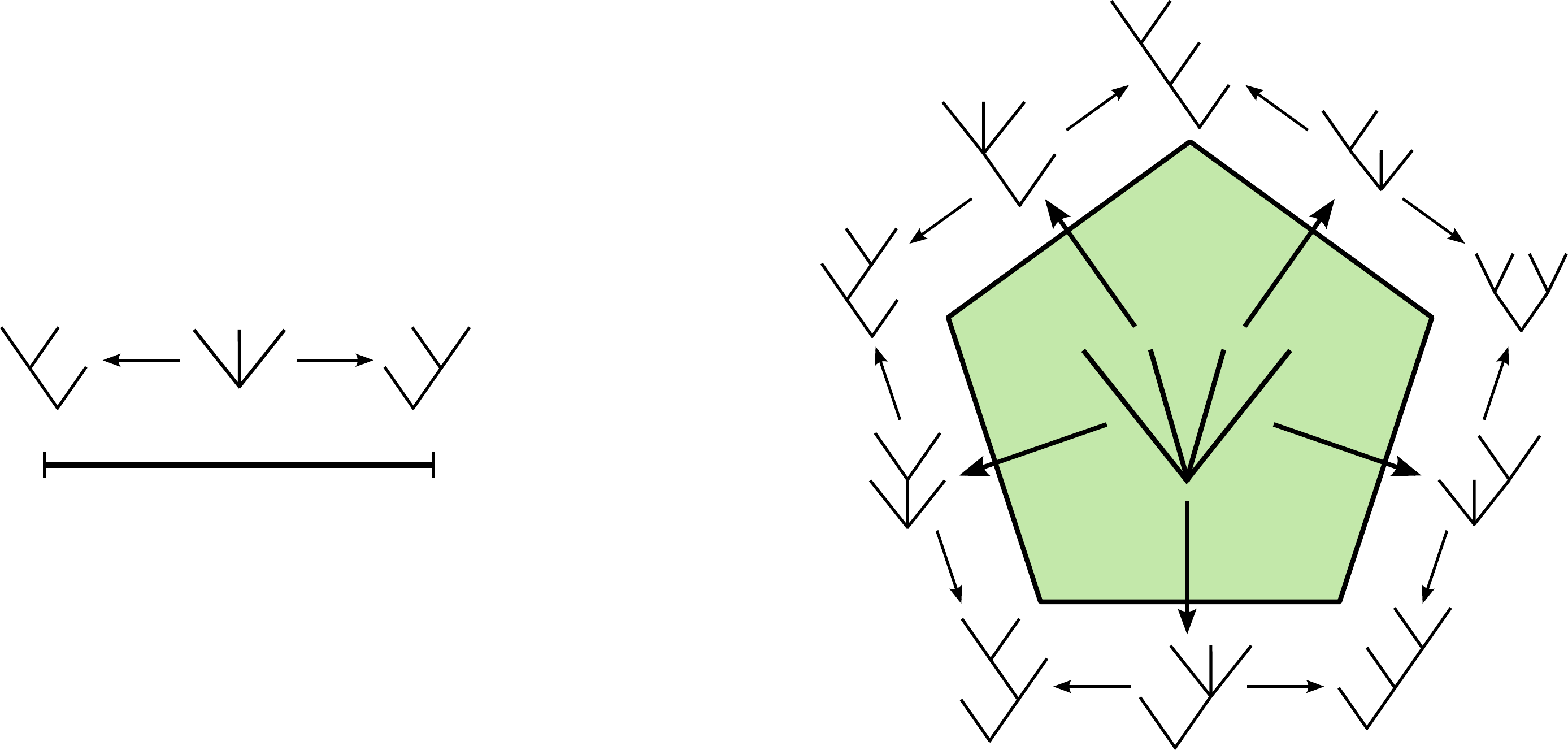
\label{fig:RRT_moves}
\end{figure}
\null\hfill$\triangle$
\end{example}


\begin{lemma}
\label{lem:Kr_move_dim}
If $T$ is a stable RRT and $T'$ is the result of performing a move on $T$, then $d(T') = d(T)-1$.
\end{lemma}

\begin{proof}
If $T$ has $r$ leaves, then we have $d(T) = r - \#\!T_\inte - 1$.
When we perform a move on $T$, the number of leaves remains the same and one new interior node is created, so $d$ decreases by one.
\end{proof}

\begin{deflem}
\label{deflem:Krtree_poset}
Define $K_r^\tree$ as a poset by declaring $T' < T$ if there is a finite sequence of moves that transforms $T$ into $T'$.
\end{deflem}

\begin{proof}
We must check that this defines a partial order on $K_r$.
The reflexivity property follows from Lemma~\ref{lem:Kr_move_dim}.
Transitivity is immediate.
\end{proof}

Recall that a \textbf{tree homomorphism} is a map $f\colon T \to \wt T$ such that $f^{-1}\{\wt\alpha\}$ is a tree for every $\wt\alpha \in \wt T$, and $\alpha E\beta$ implies that either $f(\alpha)=f(\beta)$ or $	f(\alpha)\wt Ef(\beta)$.
An \textbf{RRT homomorphism} is a tree homomorphism $f\colon T \to \wt T$ that sends leaves to leaves, root to root, interior vertices to interior vertices, and respects the cyclic orderings of edges and the orientation in
the following ways:
\begin{itemize}
\item Suppose that $\wt\beta_1, \wt\beta_2$ lie in $\incom(\wt\alpha)$ and satisfy $\wt\beta_1 < \wt\beta_2$.
Suppose that $\beta_1, \beta_2$ satisfy $f(\beta_1) = \wt\beta_1$, $f(\beta_2) = \wt\beta_2$.
Choose $\gamma \in T$ to be the first intersection of the path from $\beta_1$ to $\alpha_\root^T$ and the path from $\beta_2$ to $\alpha_\root^T$, and define $\delta_1, \delta_2$ to be the incoming neighbors of $\gamma$ that these two paths pass through.
Then the inequality $\delta_1 < \delta_2$ holds.

\item Suppose $\alpha$ lies in $T$ and $\beta$ lies in $\incom(\alpha)$.
Then either $f(\beta) = f(\alpha)$ or $f(\beta) \in \incom(\alpha)$.
\end{itemize}
If $\wt T$ is the result of performing a sequence of moves on a stable RRT $T$, then there is a surjective RRT homomorphism $\wt T \to T$ gotten by contracting the edges added to $T$ to form $\wt T$.
As the next lemma shows, all surjective homomorphisms of stable RRTs can be obtained in this fashion.

\begin{lemma}
\label{lem:RRT_ll}
If $f\colon \wt T \to T$ is a surjective homomorphism of stable RRTs, then $\wt T$ can be obtained from $T$ by applying a finite sequence of moves, and $f$ is the map that contracts the new edges that were added to $T$ to form $\wt T$.
\end{lemma}

\begin{proof}
We begin by showing that $T$ and $\wt T$ have the same number of leaves, and that $f$ satisfies $f(\lambda_i^T) = \lambda_i^{\wt T}$ for all $i$.
For this, it suffices to show that $f$ is injective on leaves.
Suppose that for some $i \neq j$, $f(\lambda_i^T) = f(\lambda_j^T)$.
The preimage of every vertex in $\wt T$ is connected, so the outgoing neighbor of $\lambda_i^T$ must also be sent to $f(\lambda_i^T)$.
This contradicts the hypothesis that $f$ sends interior vertices to interior vertices.

We prove the lemma by induction on $\#\!\wt T - \#\!T$.
If $\wt T$ and $T$ have the same number of vertices, then they are isomorphic and the claim is trivially true.
Next, suppose we have proven the claim as long as the inequality $\#\!\wt T - \#\!T \leq k$ holds, and suppose that $\wt T, T$ satisfy $\#\!\wt T = \#\!T + k+1$.
Choose an edge $\wt\alpha \wt E\wt\beta$ with $f(\wt\alpha)=f(\wt\beta)$, and assume that $\wt\beta$ is further from the root than $\wt\alpha$.
Define $T'$ to be the stable RRT gotten by contracting $\wt\alpha\wt E\wt\beta$ in $\wt T$.
Then $\wt T$ can be obtained from $T'$ by making a single move.
Moreover, $f\colon \wt T \to T$ can be factored as $\wt T \to T' \to T$, where $\wt T \to T'$ is the map that contracts $\wt\alpha\wt E\wt\beta$, and $g\colon T' \to T$ is defined by
\begin{align}
g(\wt\gamma) \coloneqq \begin{cases}
f(\wt\alpha), & \wt\gamma \in \{\wt\alpha,\wt\beta\}, \\
f(\wt\gamma), & \text{otherwise}.	
 \end{cases}
\end{align}
By induction, $g\colon T' \to T$ is the result of applying finitely many moves to $T$, so we have proven the claim by induction.
\end{proof}

Another way to characterize a stable RRT is as a \emph{1-bracketing} of $\{1,\ldots,r\}$.
Each 1-bracket corresponds to the RRT's local structure at a particular interior vertex.

\begin{definition}
\label{def:Krbr}
A \emph{1-bracket of $r$} is a nonempty consecutive subset $B \subset \{1,\ldots,r\}$. \label{p:B}
A \emph{1-bracketing of $r$} is a collection $\sB$\label{p:sB} of 1-brackets of $r$ satisfying these properties:
\begin{itemize}
\item[] {\sc (Bracketing)} If $B, B' \in \sB$ have $B \cap B' \neq \emptyset$, then either $B \subset B'$ or $B' \subset B$.

\item[] {\sc (Root and leaves)} $\sB$ contains $\{1,\ldots,r\}$ and $\{i\}$ for every $i$.
\end{itemize}
We denote the set of all 1-bracketings of $r$ by $K_r^\br$,\label{p:Krbr} and define a partial order by defining $\sB' < \sB$ if $\sB$ is a proper subcollection of $\sB'$.
\null\hfill$\triangle$
\end{definition}

\begin{definition}
\label{def:T_bracket}
Fix a stable tree $T \in K_r^\tree$.
A \emph{$T$-bracket} is a 1-bracket $B$ of $r$ with the property that for some $\alpha \in T$, $B$ is the set of indices $i$ for which $T_\alpha$ contains $\lambda_i$.
\label{p:Balpha}
We denote this bracket by $B(\alpha) \coloneqq B$.
\null\hfill$\triangle$
\end{definition}

\noindent Note that since $T$ is stable, the $T$-brackets are in bijective correspondence with the vertices of $T$.

\begin{proposition}
\label{prop:Kr_iso}
\label{p:nu}
The function $\nu\colon K_r^\tree \to K_r^\br$ that sends a stable RRT $T$ to the set of $T$-brackets is an isomorphism of posets.
\end{proposition}

\begin{proof}
\noindent\emph{Step 1: If $T$ is a stable RRT with $r$ leaves, then $\nu(T)$ is a 1-bracketing of $r$.}

\medskip

\noindent We have $B({\alpha_\root}) = \{1,\ldots,r\}$ and $B(\lambda_i^T) = \{i\}$, so $\nu(T)$ contains $\{1,\ldots,r\}$ and $\{i\}$ for every $i$.
Clearly every $B(\alpha) \in \nu(T)$ is nonempty and consecutive.
Finally, if $\alpha, \beta \in T$ have neither $\alpha \in T_\beta$ nor $\beta \in T_\alpha$, then $B(\alpha) \cap B(\beta) = \emptyset$: indeed, if $i \in B(\alpha) \cap B(\beta)$, then the path from $\lambda_i$ to $\alpha_\root$ passes through both $\alpha$ and $\beta$.
On the other hand, if $\beta \in T_\alpha$, then $B(\beta) \subset B(\alpha)$.

\medskip

\noindent\emph{Step 2: We define a putative inverse $\tau\colon K_r^\br \to K_r^\tree$.}

\medskip

\noindent Given $\sB \in K_r^\br$, we will define $\tau(\sB)$ via Lemma~\ref{lem:RRT_in}.
Define
\begin{gather}
V \coloneqq \sB,
\qquad
\{1,\ldots,r\} \eqqcolon \alpha_\root \in V,
\\
B' \in \incom(B) \iff \bigl(B' \subsetneq B \:\text{ and }\: \not\!\exists\: B'' \in \sB: B' \subsetneq B'' \subsetneq B\bigr),
\nonumber
\end{gather}
and denote the vertex corresponding to $B \in \sB$ by $\alpha(B)$. \label{p:alpha_B}
Any distinct $\alpha(B'), \alpha(B'') \in \incom(\alpha(B))$ must have $B' \cap B'' = \emptyset$, since otherwise one of $B'$ and $B''$ would be properly contained in the other; we may therefore order $\incom(\alpha(B))$ by declaring that $\alpha(B'), \alpha(B'') \in \incom(\alpha(B))$ have $\alpha(B') <_{\alpha(B)} \alpha(B'')$ if and only if $i' < i''$ for all $i' \in B'$, $i'' \in B''$.

We now verify conditions (1--3) from Lemma~\ref{lem:RRT_in}.
To prove (1), note that for any $\alpha(B) \in V$ and $\alpha(B') \in \incom(\alpha(B))$, $B'$ is a proper subset of $B$; therefore $\incom(\alpha(B))$ does not contain $\alpha(\{1,\ldots,r\}) = \alpha_\root$.
For (2), fix $\alpha(B) \in V \setminus \{\alpha_\root\}$ and define $\Sigma \coloneqq \{B' \in \sB \:|\: B' \supsetneq B\}$.
Then $\Sigma$ contains $\{1,\ldots,r\}$, hence is nonempty.
Moreover, if $B', B'' \in \Sigma$ are distinct and minimal with respect to inclusion, then $B' \cap B'' \supset B \neq \emptyset$; therefore one of $B', B''$ must contain the other, a contradiction to minimality.
This shows that $\Sigma$ contains a unique minimal element, which is the unique $B'$ with $\incom(\alpha(B')) \ni \alpha(B)$; this establishes (2).
Finally, if $\alpha(B^1), \ldots, \alpha(B^\ell) \in V$ is a sequence with $\ell \geq 2$ and $\alpha(B^j) \in \incom(\alpha(B^{j+1}))$ for every $j$, then $B^1 \subsetneq B^\ell$, hence $\alpha(B^1) \neq \alpha(B^\ell)$.

It is clear that $\tau(\sB)$ is stable and that $\{i\}$ is the $i$-th leaf of $\tau(\sB)$.

\medskip

\noindent{\it Step 3: We show that $\nu$ and $\tau$ are inverse bijections.}

\medskip

\noindent First, fix $T \in K_r^\tree$; we claim $T \simeq \tau(\nu(T))$.
There is an obvious identification of vertices, which identifies root with root.
Next, we must show that the edge relations on the vertices $T$ are the same, which is to say that $\beta \in \incom(\alpha)$ is equivalent to $B(\beta) \subsetneq B(\alpha)$ and the nonexistence of $\gamma \in T$ with $B(\beta) \subsetneq B(\gamma) \subsetneq B(\alpha)$.

Fix $\beta \in \incom(\alpha)$.
Certainly $B(\beta) \subset B(\alpha)$, and this containment is proper by the stability of $T$.
(Indeed, define an outgoing path by setting $\delta_1 \coloneqq \alpha$, choosing $\delta_2$ to be an incoming neighbor of $\alpha$ other than $\beta$, and inductively defining $\delta_{i+1}$ to be an incoming neighbor of $\delta_i$ as long as $\incom(\delta_i)$ is not empty.
This path will terminate, by condition (3) in Lemma~\ref{lem:RRT_in}.
Moreover, this path does not include $\beta$: if it did, it would by (2) have $\delta_i = \alpha$ for some $i \geq 2$, which is impossible by (3).
If $\lambda_j$ is the leaf at which this path terminates, then $j \in B(\alpha)\setminus B(\beta)$.)
Suppose
for a contradiction
that there exists $\gamma$ with $B(\beta) \subsetneq B(\gamma) \subsetneq B(\alpha)$, and choose $i \in B(\beta)$.
Since $B(\alpha), B(\beta), B(\gamma)$ all contain $i$, the path $[\lambda_i,\alpha_\root]$ must contain $\alpha, \beta, \gamma$.
The containments $B(\beta)\subsetneq B(\gamma) \subsetneq B(\alpha)$ now imply that the path $\beta = \delta_1, \delta_2, \ldots, \delta_\ell = \alpha$ from $\beta$ to $\alpha$ is oriented toward the root and has $\delta_i = \gamma$ for some $i \in [2,\ell-1]$.
Without loss of generality we may assume $i = 2$.
Then $\beta \in \incom(\gamma)$, so $\beta$ cannot lie in $\incom(\alpha)$, a contradiction.

Conversely, suppose $B(\beta) \subsetneq B(\alpha)$ and that there does not exist $\gamma$ with $B(\beta) \subsetneq B(\gamma) \subsetneq B(\alpha)$.
An argument similar to the one in the previous paragraph yields $\beta \in \incom(\alpha)$.

\medskip

\noindent Second, fix $\sB \in K_r^\br$; we claim $\sB = \nu(\tau(\sB))$.
To prove this, it suffices to show that $\alpha(\{i\}) \in T_{\alpha(B)}$ if and only if $i \in B$.
Suppose $\alpha(\{i\}) \in T_{\alpha(B)}$.
This means that there is a sequence of 1-brackets $\{i\} = B_1, B_2, \ldots, B_\ell = B$ in $\sB$ such that for every $i$, either (a) $B_i \subsetneq B_{i+1}$ and there exists no $B' \in \sB$ with $B_i \subsetneq B' \subsetneq B_{i+1}$, or (b) the same holds but with $B_i$ and $B_{i+1}$ interchanged.
In fact, an argument similar to the one made in the proof of {\sc (no cycles)} in Lemma~\ref{lem:RRT_in} implies that for every $i$ it is (a) that holds.
Therefore $i$ lies in $B$.
Conversely, suppose $i \in B$.
Define a sequence in $\sB$ by setting $B_1 \coloneqq B$ and, as long as $B_i$ is not equal to $\{i\}$, defining $B_{i+1}$ to be the largest element of $\sB$ satisfying $B_i \supsetneq B_{i+1} \supsetneq \{i\}$.
This defines a non-self-intersecting path in $\tau(\sB)$ that begins at $\alpha(B)$ and terminates at the $i$-th leaf, which proves the backwards direction of the assertion that $\alpha(\{i\}) \in T_{\alpha(B)}$ is equivalent to $i \in B$.

\medskip

\noindent{\it Step 4: We show that $\nu$ and $\tau$ respect the partial orders on $K_r^\tree$ and $K_r^\br$.}

\medskip

\noindent First, we show that if $T' < T$, then $\nu(T') < \nu(T)$.
We may assume without loss of generality that $T'$ is the result of performing a single move on $T$.
Denote by $\alpha \in T$ the vertex at which the move is performed, so that $T'$ is produced from $T$ by introducing a new incoming neighbor of $\alpha$.
We may therefore regard $V(T)$ as a subset of $V(T')$.
If $\beta$ is a vertex of $T$, and $B(\beta)$ resp.\ $B'(\beta)$ denote the indices of the leaves in $T_\beta$ resp.\ in $T'_\beta$, then $B(\beta) = B'(\beta)$.
Therefore $\nu(T') < \nu(T)$.

Second, we show that if $\sB' < \sB$, then $\tau(\sB') < \tau(\sB)$.
Define a map $f\colon \tau(\sB') \to \tau(\sB)$ like so:
\begin{align}
f(\alpha(B')) \coloneqq \alpha\bigl(\min\{B \in \sB \:|\: B \supset B'\}\bigr),	
\end{align}
where the minimum is taken with respect to inclusion.
By {\sc (root and leaves)}, the set over which we are taking the minimum contains $\{1,\ldots,r\}$, hence is nonempty; therefore $f$ is well-defined.
I claim that $f$ is a surjective homomorphism of stable RRTs.
Again by {\sc (root and leaves)}, $f$ sends
$\lambda_i^{\tau(\sB')}$ to $\lambda_i^{\tau(\sB)}$
and $\alpha_\root^{\tau(\sB')}$ to $\alpha_\root^{\tau(\sB)}$;
since $\sB$ is a subcollection of $\sB'$, $f$ is surjective.
It remains to show that the preimage under $f$ of each vertex in
$\tau(\sB)$
is connected.
Fix $B \in \sB$; it suffices to show that for any $B' \in \sB'$ with $f(\alpha(B')) = \alpha(B)$, the path from $\alpha(B')$ to $\alpha(B)$ in $\tau(\sB')$ is contained in $f^{-1}\{\alpha(B)\}$.
This is apparent from the definition of $f$, so we may conclude that $f$ is a surjective morphism of stable RRTs, hence $\tau(\sB') < \tau(\sB)$ by Lemma~\ref{lem:RRT_ll}.
\end{proof}

\noindent By this lemma, we may define $K_r \coloneqq K_r^\tree = K_r^\br$. \label{p:Kr}

\subsection{\texorpdfstring{$\wh K_r$}{Kr} is an abstract polytope of dimension $r-2$}
\label{ss:Kr_polytope}

In this subsection we prove the following proposition.

\begin{proposition}[Key properties of $K_r$]
\label{prop:Kr_main}
The posets $(K_r)$ satisfy the following properties:
\begin{itemize}
\item[] \textsc{(abstract polytope)} For $r \geq 2$, $\wh{K_r} \coloneqq K_r \cup \{F_{-1}\}$ is an abstract polytope of dimension $r-2$.

\item[] \textsc{(recursive)} For any $T \in K_r^\tree$, there is an inclusion of posets
\begin{align}
\gamma_T\colon \prod_{\alpha \in T_\inte} K_{\#\!\incom(\alpha)}^\tree \hra K_r^\tree,
\end{align}
which restricts to a poset isomorphism onto $\cl(T) = (F_{-1},T]$.
\end{itemize}
\end{proposition}

\begin{proof}
These two properties are proven in Def.-Lem.~\ref{deflem:gammaT} resp.\ Prop.~\ref{prop:Kr_polytope}.
\end{proof}

\noindent We begin by establishing {\sc(recursive)}.
After the proof of Def.-Lem.~\ref{deflem:gammaT}, we will illustrate the definition of $\gamma_T$ in an example.

\begin{deflem}
\label{deflem:gammaT}
Fix $r \geq 2$ and $T \in K_r^\tree$.
Define a map
\begin{align}
\gamma_T\colon \prod_{\alpha \in T_\inte} K_{\#\!\incom(\alpha)}^\tree \to K_r^\tree
\end{align}
by sending $(T^\beta)_{\beta \in \inte(\alpha)}$ to the RRT gotten by replacing each $\beta$ and its incoming neighbors and edges by $T^\beta$.
Then $\gamma_T$ restricts to a poset isomorphism from its domain to $\cl(T)$.
\end{deflem}

\begin{proof}
We will define an inverse $\sigma_T\colon \cl(T) \to \prod_{\alpha\in T_\inte} K_{\#\!\incom(\alpha)}^\tree$ to the restriction of $\gamma_T$.
Fix $\wh T \in \cl(T)$; then there is a (unique) surjective homomorphism $f\colon \wh T \to T$ of ribbon trees.
For any $\alpha \in T$, define $\wh\alpha \in \wh T$ to be the element of $f^{-1}\{\alpha\}$ that is closest to the root.
For any $\alpha \in T_\inte$, define
\begin{align}
\wh T^\alpha \coloneqq f^{-1}\{\alpha\} \cup \bigl\{\wh\beta \:|\: \beta \in \incom(\alpha)\}, \quad \alpha_\root^{\wh T^\alpha} \coloneqq \wh \alpha.
\end{align}
Then $\wh T^\alpha$ is a subtree of $\wh T$: since $f$ is a tree homomorphism, $f^{-1}\{\alpha\}$ is a subtree of $\wh T$.
For any $\beta \in \incom(\alpha)$, we must either have $f(\out(\wh\beta)) = f(\wh\beta) = \beta$ or $f(\out(\wh\beta)) = \out(f(\wh\beta)) = \alpha$; by the definition of $\wh\beta$, the latter equality must hold, so $\wh T^\alpha$ is indeed a subtree of $\wh T$.
Furthermore, the ribbon tree structure of $\wh T$ induces a ribbon tree structure on $\wh T^\alpha$, so $\wh T^\alpha$ is an RRT.
I claim that $\wh T^\alpha$ is stable and has leaves in bijection with $\incom(\alpha)$.
For any $\beta \in \incom(\alpha)$, it follows immediately from the definition of $\wh\beta$ that $\wh\beta \in \wh T^\alpha$ is a leaf.
Next, fix $\beta \in f^{-1}\{\alpha\}$; we must show $\beta \in \wh T^\alpha$ has at least 2 incoming neighbors.
In fact, its incoming neighbors are in correspondence with the incoming neighbors of $\beta$ in $\wh T$.
We may conclude that $\wh T^\alpha$ is a stable RRT with leaves in correspondence with $\incom(\alpha)$.
We may now define $\sigma_T$:
\begin{align}
\sigma_T \colon \cl(T) \to \prod_{\alpha\in T_\inte} K_{\#\!\incom(\alpha)}^\tree, \quad \sigma_T(\wh T) \coloneqq \bigl(\wh T^\alpha\bigr)_{\alpha \in T_\inte}.	
\end{align}
It is clear from the definition of $\sigma_T$ that it is an inverse to the restriction $\gamma_T \colon \prod_{\alpha \in T_\inte} K_{\#\!\incom(\alpha)}^\tree \to \cl(T)$.
It is also clear that $\gamma_T$ and $\sigma_T$ respect the partial orders on $\prod_{\alpha \in T_\inte} K_{\#\!\incom(\alpha)}^\tree$ and $K_r^\tree$.
\end{proof}


\begin{example}
In the following figure, we illustrate the definition of $\gamma_T$ in a simple example.
On the left is $T$, which has three interior vertices, the incoming edges of which are colored red, blue, or green respectively.
$\gamma_T$ acts by replacing the red, blue, and green corollas by RRTs in $K_3^\tree$, $K_4^\tree$, and $K_3^\tree$, respectively.
\begin{figure}[H]
\centering
\def\svgwidth{0.9\columnwidth}
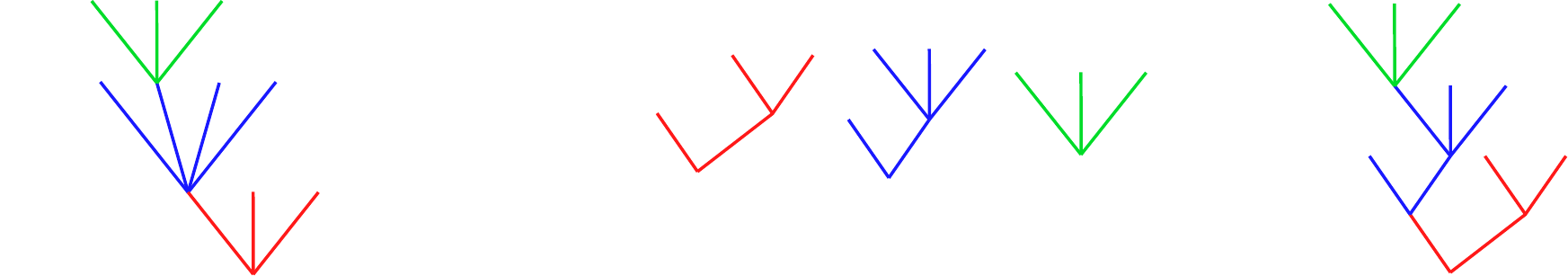
\label{fig:gamma_example}
\end{figure}
\null\hfill$\triangle$
\end{example}


We now turn to the proof of the {\sc(abstract polytope)} property from Prop.~\ref{prop:Kr_main}.
Define $\wh{K_r} \coloneqq K_r \cup \{F_{-1}\}$, where $F_{-1}$ is a formal minimal element with $d(F_{-1}) \coloneqq -1$.
We first recall the notion of abstract polytope:
\begin{definition}
\label{def:abstract_polytope}
An {\bf{abstract polytope of rank $n \in \bZ_{\geq -1}$}} is a partially ordered set $P$ (whose elements are called {\bf{faces}}) satisfying the properties {\sc (extremal)}, {\sc (flag-length)}, {\sc (strongly connected)}, and {\sc (diamond)}, defined below.
\begin{itemize}
\item[] {\sc (extremal)} $P$ has a least and a greatest face, denoted $F_{-1}$ resp.\ $F_\top$.

\item[] {\sc (flag-length)} Every flag (i.e.\ maximal chain) of $P$ has length $n+1$, i.e.\ contains $n+2$ faces.
\end{itemize}
For $F, G \in P$ with $F \leq G$, recall that the closed interval $[F,G]$ is defined by
\begin{align}
[F,G] \coloneqq \{H \in P \:|\: F \leq H \leq G\}.	
\end{align}
It follows from {\sc (extremal)} and {\sc (flag-length)} that we can endow $P$ with a rank function, where $\rk F$ is defined to be the rank of the poset $[F_{-1},F]$.
\begin{itemize}
\item[] {\sc (strongly connected)} For every $F < G$ with $\rk G - \rk F \geq 3$ and for every $H, H' \in (F,G)$, there is a sequence $(H = H_1, H_2, \ldots, H_k = H')$ in $(F,G)$ such that $H_i$ and $H_{i+1}$ are adjacent for every $i$ (i.e., either $H_i \prec H_{i+1}$ or $H_{i+1} \prec H_i$, where we write $F_1 \prec F_2$ if $F_1 < F_2$ and there exists no $F_3$ with $F_1 < F_3 < F_2$).

\item[] {\sc (diamond)} For every $F < G$ with $\rk G - \rk F = 2$, the open interval $(F,G)$ contains exactly 2 elements.
\end{itemize}
\null\hfill$\triangle$
\end{definition}

\begin{proposition}
\label{prop:Kr_polytope}
For any $r \geq 2$, $\wh{K_r}$ is an abstract polytope of dimension $r-2$.	
\end{proposition}

\begin{proof}
\begin{itemize}
\item[] {\sc (extremal)}
The least face is the face $F_{-1}$ we have added to $K_r$ to form $\wh{K_r}$.
The greatest face (in $\wh{K_r^\br}$) is the 1-bracketing $\bigl\{\{1,\ldots,r\},\{1\},\ldots,\{r\}\bigl\}$.

\medskip

\item[] {\sc (flag-length)}
We must show that if $T^0 < \cdots < T^\ell$ is a maximal chain in $K_r^\tree$, then $\ell = r-2$.
By Lemmata~\ref{lem:Kr_dim_props} and \ref{lem:Kr_move_dim}, we have $0 \leq d(T^0) < \cdots < d(T^\ell) \leq r-2$.
To prove the claim, we must show that every dimension between 0 and $r-2$ is represented.
For any $T^i, T^{i+1}$, we must have $d(T^i) = d(T^{i+1}) - 1$: otherwise, there exists $T' \in K_r^\tree$ which can be obtained by performing a single move to $T^{i+1}$, and which satisfies $d(T^i) < d(T') < d(T^{i+1})$; this contradicts the maximality of the chain.
Again by maximality, we must have $T^\ell = F_\top$.
It remains to show $d(T^0) = 0$.
Suppose for a contradiction that $d(T^0)$ is positive; then by Lemma~\ref{lem:Kr_dim_props}, there exists $\alpha \in T^0_\inte$ with $\#\!\incom(\alpha) \geq 3$.
It follows that we may perform a move to $T^0$, so there exists $T' \in K_r^\tree$ with $T' < T^0$, contradicting the maximality of this chain.

\medskip

\item[] {\sc (strongly connected)}
In this step we may assume $r \geq 4$, since otherwise {\sc(strongly connected)} is vacuous.
First, we show that $\wh{K_r}$ is connected.
It suffices to show that for any $a, b \geq 2$ with $a + b - 1 = r$ and $i$ with $1 \leq i \leq a$, there exists a path in $\wh{K_r} \setminus F_\top$ between these two codimension-1 faces:

\begin{figure}[H]
\centering
\def\svgwidth{0.38\columnwidth}
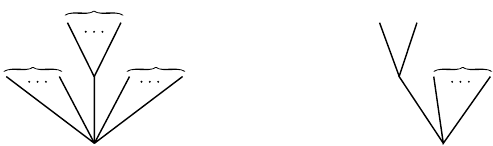
\end{figure}

\noindent We produce such a path in the following four, exhaustive, cases:

\begin{figure}[H]
\centering
\def\svgwidth{0.7\columnwidth}
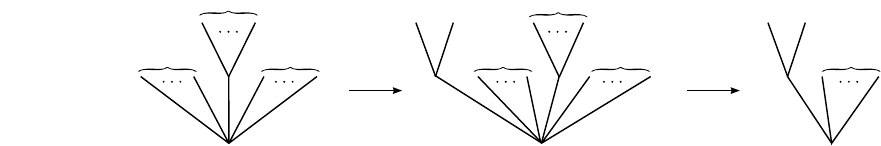
\end{figure}

\begin{figure}[H]
\centering
\def\svgwidth{0.65\columnwidth}
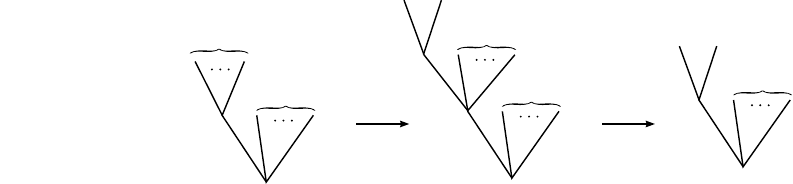
\end{figure}

\begin{figure}[H]
\centering
\def\svgwidth{0.9\columnwidth}
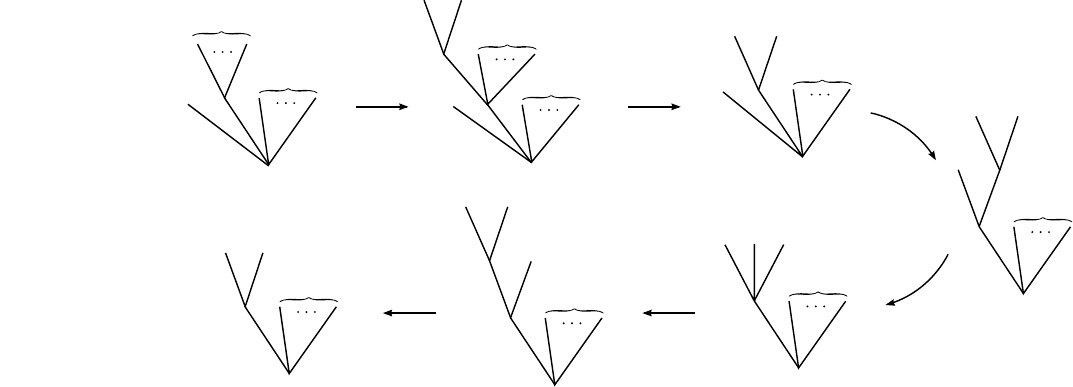
\end{figure}

\begin{figure}[H]
\centering
\def\svgwidth{0.95\columnwidth}
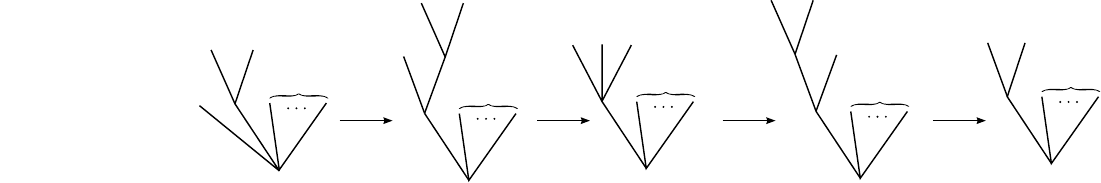
\end{figure}

Next, we show that for any $T \in K_r^\tree$ with $d(T) \geq 2$, the interval $[F_{-1},T]$ is connected.
By Def.-Lem.~\ref{deflem:gammaT}, $[F_{-1},T]$ is isomorphic to $\{F_{-1}\} \cup \prod_{\alpha \in T_\inte} K_{\#\!\incom(\alpha)}^\tree$.
The connectedness of $\wh{K^\tree_s}$
for $s \geq 4$, the fact that $K^\tree_3$ is isomorphic to the face poset of an interval, and the inequality $\sum_{\alpha \in T_\inte} (\#\incom(\alpha)-2) = d(T) \geq 2$ 
imply that $\{F_{-1}\} \cup \prod_{\alpha \in T_\inte} K_{\#\!\incom(\alpha)}^\tree$ is connected.

Finally, we show that for any $\sB,\sB'$ with $\sB' < \sB$ and $d(\sB') \leq d(\sB)-3$, the interval $[\sB',\sB]$ is connected.
Extend the dimension function $d$ to $K_r^\br$ via the identification $K_r^\br \simeq K_r^\tree$.
Then
$d(\sB) = 2r - \#\!\sB - 1$.
To prove that $[\sB',\sB]$ is connected, it suffices to show that for any distinct $\wt\sB^1, \wt\sB^2 \in (\sB',\sB)$ with $d(\wt\sB^1) = d(\wt\sB^2) = d(\sB)-1$, there is a path from $\wt\sB^1$ to $\wt\sB^2$ in $(\sB',\sB)$.
The inequality $\wt\sB^j \geq \sB'$ for $j \in \{1,2\}$ implies that $\wh\sB \coloneqq \wt\sB^1 \cup \wt\sB^2$ is a 1-bracketing.
Moreover, it satisfies $\wh\sB \in [\sB',\sB)$, and by the formula for the dimension of a 1-bracketing given above, it satisfies $d(\wh\sB) = d(\sB) - 2$.
By the hypothesis $d(\sB') \leq d(\sB) - 3$, $\wh\sB$ must therefore lie in $(\sB',\sB)$, so $(\wt\sB^1,\wh\sB,\wt\sB^2)$ is a path in $(\sB',\sB)$.

\medskip

\item[] {\sc(diamond)}
First, fix $T \in K_r^\tree$ with $d(T) = 1$; we must show that the open interval $(F_{-1},T)$ contains exactly two elements.
Lemma~\ref{lem:Kr_dim_props} implies that every vertex in $T_\inte$ has two incoming neighbors except for a single $\alpha$ with $\#\!\incom(\alpha) = 3$.
Denote the incoming neighbors of $\alpha$ by $(\beta_1,\beta_2,\beta_3)$.
There are two possible moves that can be made at $\alpha$, by either splitting off $(\alpha_1,\alpha_2)$ or $(\alpha_2,\alpha_3)$.
In fact, these are the only two moves that can be performed on $T$.
Since $d(T) = 1$, $(F_{-1},T)$ contains two elements.

Next, fix $\sB, \sB' \in K_r^\br$ with $d(\sB') = d(\sB) - 2$.
It follows from the definition of $d$ and Lemma~\ref{prop:Kr_iso} that $\sB'$ is obtained from $\sB$ by adding two 1-brackets.
From this it is clear that the open interval $(\sB',\sB)$ contains two elements.
\end{itemize}
\end{proof}

\section{Construction of the 2-associahedra \texorpdfstring{$W_\bn$}{Wn}}
\label{sec:2ass}


In this section, we define the posets $W_\bn^\tree$ (\S\ref{ss:Wntree_construction}) and $W_\bn^\br$ (\S\ref{ss:Wnbr_construction}), then show that they are isomorphic (\S\ref{ss:Wn_iso}).
This allows us to define the 2-associahedron by $W_\bn \coloneqq W_\bn^\tree = W_\bn^\br$.

\subsection{The poset \texorpdfstring{$W_\bn^\tree$}{Wntree} of stable tree-pairs}
\label{ss:Wntree_construction}

We begin with the definition of $W_\bn^\tree$.
It is rather technical, and we advise the reader to refer to Ex.~\ref{ex:tree-pair_examples} while looking at this definition for the first time.

\begin{definition}
\label{p:2T}
A \emph{stable tree-pair of type $\bn$} is a datum $2T = T_b \sr{f}{\to} T_s$, with $T_b, T_s, f$ described below:
\begin{itemize}
\item The \emph{bubble tree} $T_b$ is an RRT whose edges are either solid or dashed, which must satisfy these properties:
	\begin{itemize}
		\item The vertices of $T_b$ are partitioned as $V(T_b) = V_\comp \sqcup V_\seam \sqcup V_\mk$, where:
			\begin{itemize}
				\item every $\alpha \in V_\comp$ has $\geq 1$ solid incoming edge, no dashed incoming edges, and either a dashed or no outgoing edge;
				\item every $\alpha \in V_\seam$ has $\geq 0$ dashed incoming edges, no solid incoming edges, and a solid outgoing edge; and
				\item every $\alpha \in V_\mk$ has no incoming edges and either a dashed or no outgoing edge.
			\end{itemize}
			We partition $V_\comp \eqqcolon V_\comp^1 \sqcup V_\comp^{\geq2}$ according to the number of incoming edges of a given vertex.
		\item ({\sc stability}) If $\alpha$ is a vertex in $V_\comp^1$ and $\beta$ is its incoming neighbor, then $\#\!\incom(\beta) \geq 2$; if $\alpha$ is a vertex in $V_\comp^{\geq2}$ and $\beta_1,\ldots,\beta_\ell$ are its incoming neighbors, then there exists $j$ with $\#\!\incom(\beta_j) \geq 1$.
	\end{itemize}
	\item The \emph{seam tree} $T_s$ is an element of $K_r^\tree$.
	\item The \emph{coherence map} is a map $f\colon T_b \to T_s$ of sets having these properties:
		\begin{itemize}
			\item $f$ sends root to root, and if $\beta \in \incom(\alpha)$ in $T_b$, then either $f(\beta) \in \incom(f(\alpha))$ or $f(\alpha) = f(\beta)$.
			\item $f$ contracts all dashed edges, and every solid edge whose terminal vertex is in $V_\comp^1$.
			\item For any $\alpha \in V_\comp^{\geq2}$, $f$ maps the incoming edges of $\alpha$ bijectively onto the incoming edges of $f(\alpha)$, compatibly with $<_\alpha$ and $<_{f(\alpha)}$.
			\item $f$ sends every element of $V_\mk$ to a leaf of $T_s$, and if $\lambda_i^{T_s}$ is the $i$-th leaf of $T_s$, then $f^{-1}\{\lambda_i^{T_s}\}$ contains $n_i$ elements of $V_\mk$, which we denote by $\mu_{i1}^{T_b},\ldots,\mu_{in_i}^{T_b}$.
		\end{itemize}
\end{itemize}
We denote by $W_\bn^\tree$\label{p:Wntree} the set of isomorphism classes of stable tree-pairs of type $\bn$.
Here an isomorphism from $T_b \sr{f}{\to} T_s$ to $T_b' \sr{f'}{\to} T_s'$ is a pair of maps $\varphi_b\colon T_b \to T_b'$ and $\varphi_s\colon T_s \to T_s'$ that fit into a commutative square in the obvious way and that respect all the structure of the bubble trees and seam trees.
\null\hfill$\triangle$
\end{definition}


\begin{example}
\label{ex:tree-pair_examples}
In the following figure we illustrate some of the notation that we have just introduced.
We picture the same tree-pair (with $r = 5$, $\bn = (1,1,4,1,0)$) three times, in each case indicating different data.
In each case, $T_b$ is above and $T_s$ is below.
On the left, we label the roots of $T_b$ and $T_s$, the leaves of $T_s$, and the elements of $V_\mk(T_b)$.
In the middle, we indicate the coherence map $f\colon T_b \to T_s$: we color the edges of $T_s$, and use those same colors to show which edges in $T_b$ are identified with the various edges of $T_s$.
Some edges in $T_b$ are contracted by $f$, which we indicate by using black.
On the right, we show how the vertices of $T_b$ are partitioned into $V_\mk$, $V_\seam$, and $V_\comp$.

\begin{figure}[H]
\centering
\def\svgwidth{1.0\columnwidth}
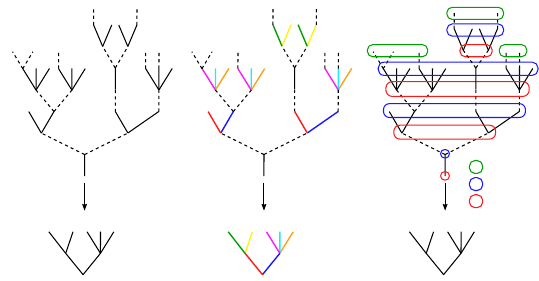
\label{fig:tree-pair_examples}
\end{figure}
\null\hfill$\triangle$
\end{example}


\begin{definition}
\label{def:tree-pair_dim}
For $2T$ a stable tree-pair, we define the \emph{dimension} $d(2T) \in \bZ$ like so:
\begin{align} \label{eq:tree-pair_dim}
d(2T) \coloneqq |\bn| + r - \#\!V^1_\comp(T_b) - \#\!(T_s)_\inte - 2.
\end{align}
\null\hfill$\triangle$
\end{definition}

\begin{remark}
Note that $W_1^\tree$ has a single element, the stable tree-pair $\bullet \to \bullet$; its dimension is zero.
\null\hfill$\triangle$
\end{remark}

\begin{lemma}
\label{lem:Wn_dim}
Fix a stable tree-pair $2T \in W_\bn$.
\begin{itemize}
\item[(a)] The dimension can be re-expressed using this formula:
\begin{equation} \label{eq:2T_dim_reform}
	d(2T) = \sum_{{\alpha \in V^1_\comp(T_b)}\atop{\incom(\alpha) = (\beta)}} \bigl(\#\!\incom(\beta) - 2\bigr)
	+ \sum_{\alpha \in V^{\geq 2}_\comp(T_b)} \left(\Bigl(\sum_{\beta \in \incom(\alpha)} \#\!\incom(\beta)\Bigr)-1\right)
	+ \sum_{\rho \in (T_s)_\inte} \bigl(\#\!\incom(\rho) - 2\bigr).
\end{equation}

\item[(b)] If $\bn \neq (1)$, the dimension satisfies the inequality $0 \leq d(2T) \leq |\bn| + r - 3$.
\end{itemize}
\end{lemma}

\begin{proof}
\begin{itemize}
	\item[(a)] \eqref{eq:2T_dim_reform} is the result of substituting into \eqref{eq:tree-pair_dim} \eqref{eq:T_valence_sum} and the following identity:
	\begin{align}
	\label{eq:2T_valence_sum}
	\sum_{\alpha \in V_\seam(T_b)} \#\!\incom(\alpha) = \#\!V_\comp(T_b) + |\bn| - 1.
	\end{align}
	This follows by noting that $\sum_{\alpha \in V_\seam(T_b)} \#\!\incom(\alpha)$ counts the elements of $V_\comp(T_b)$ that are the incoming neighbor of some element of $V_\seam(T_b)$.
	This set is the complement of the root of $T_b$, hence has cardinality $\#\!V_\comp(T_b) + |\bn| - 1$.
	
	\item[(b)]
	The inequality $d(2T) \geq 0$ follows from \eqref{eq:2T_dim_reform} and the ({\sc stability}) axiom in the definition of stable tree-pairs; the inequality $d(2T) \leq |\bn| + r - 3$ follows immediately from \eqref{eq:tree-pair_dim}.
\end{itemize}
\end{proof}

Next we define three types of ``moves'' that can be performed on stable tree-pairs.
Examples of these moves are
shown in Ex.~\ref{ex:tree-pair_moves}.
As we will show, each move decreases the dimension $d$ by one.

Fix a stable tree-pair $T_b \sr{f}{\to} T_s$.
Here are the moves that may be applied:

\begin{itemize}
\item[] {\emph{Type-1 moves:}} Fix $\alpha \in V_\comp(T_b)$ and $\beta \in \incom(\alpha)$.
Choose a consecutive subset $(\gamma_{p+1},\ldots,\gamma_{p+l}) \subset (\gamma_1,\ldots,\gamma_k) = \incom(\beta)$ where $l$ satisfies the following condition (necessary to preserve stability):
	\begin{itemize}
	\item If $\#\!\incom(\alpha) = 1$, then we require $2 \leq l < k$.

	\item If $\#\!\incom(\alpha) \geq 2$, then we require $2 \leq l \leq k$.
	\end{itemize}
The corresponding \emph{type-1 move} consists of modifying the incoming edges of $\beta$ as shown here,

\begin{figure}[H]
\centering
\def\svgwidth{0.4\columnwidth}
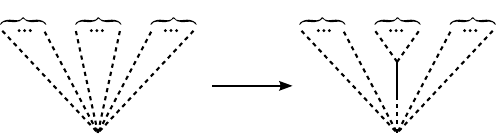
\label{fig:2T_move_1}
\end{figure}

\noindent leaving $T_s$ unchanged, and modifying $f$ in the obvious way.

\medskip

\item[] {\emph{Type-2 moves:}} Fix $\alpha \in V_\comp(T_s)$.
Choose a consecutive subset $(\gamma_{p+1},\ldots,\gamma_{p+l})\subset(\gamma_1,\ldots,\gamma_k)=\incom(\alpha)$ where $l$ satisfies $2 \leq l < k$.
For every $\wt\alpha \in V_\comp^{\geq 2}(T_b) \cap f^{-1}\{\alpha\}$ with $\incom\bigl(\wt\alpha\bigr) \eqqcolon (\beta_1,\ldots,\beta_k)$ and for every $i$ with $1 \leq i \leq l$, choose $q \geq 0$ and partition $\ba \coloneqq \bigl(\#\!\incom(\beta_{p+i})\bigr)_{i=1}^\ell$ as $\ba = \sum_{j=1}^q \bb^j$, for $\bb^1,\ldots,\bb^q \in \bZ^\ell\setminus\{\bzero\}$.
The corresponding \emph{type-2 move} consists of (a) \label{type2a} modifying the incoming edges of $\alpha$ as shown here,

\begin{figure}[H]
\centering
\def\svgwidth{0.35\columnwidth}
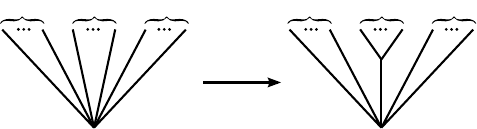
\label{fig:2T_move_2s}
\end{figure}

\noindent and (b) modifying the incoming edges of each $\wt \alpha$ as shown here,

\begin{figure}[H]
\centering
\def\svgwidth{0.7\columnwidth}
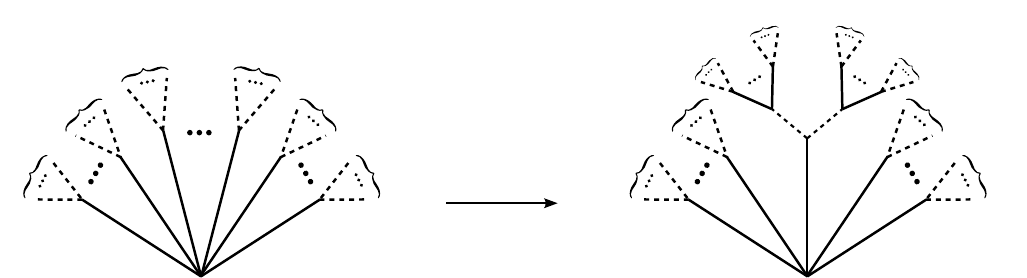
\label{fig:2T_move_2b}
\end{figure}

\noindent and modifying $f$ in the obvious way.

\medskip

\item[] {\emph{Type-3 moves:}} Choose $\alpha \in V^{\geq2}_\comp(T_b)$, and write $\incom(\alpha) \eqqcolon \{\beta_1,\ldots,\beta_k\}$.
Choose $q \geq 2$, and partition $\ba \coloneqq \bigl(\#\!\incom(\beta_i)\bigr)_{i=1}^k$ as $\ba = \sum_{j=1}^q \bb^j$ for $\bb^1,\ldots,\bb^q \in \bZ^k\setminus\{\bzero\}$.
The corresponding {\emph{type-3 move}} consists of modifying the incoming edges of $\alpha$ as shown here,

\begin{figure}[H]
\centering
\def\svgwidth{0.5\columnwidth}
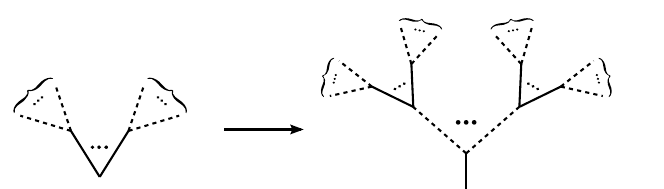
\label{fig:2T_move_3}
\end{figure}

\noindent leaving $T_s$ unchanged, and modifying $f$ in the obvious way.
\end{itemize}


\begin{example}
\label{ex:tree-pair_moves}
In the following figure, we illustrate the moves just introduced.
\begin{figure}[H]
\centering
\def\svgwidth{0.9\columnwidth}
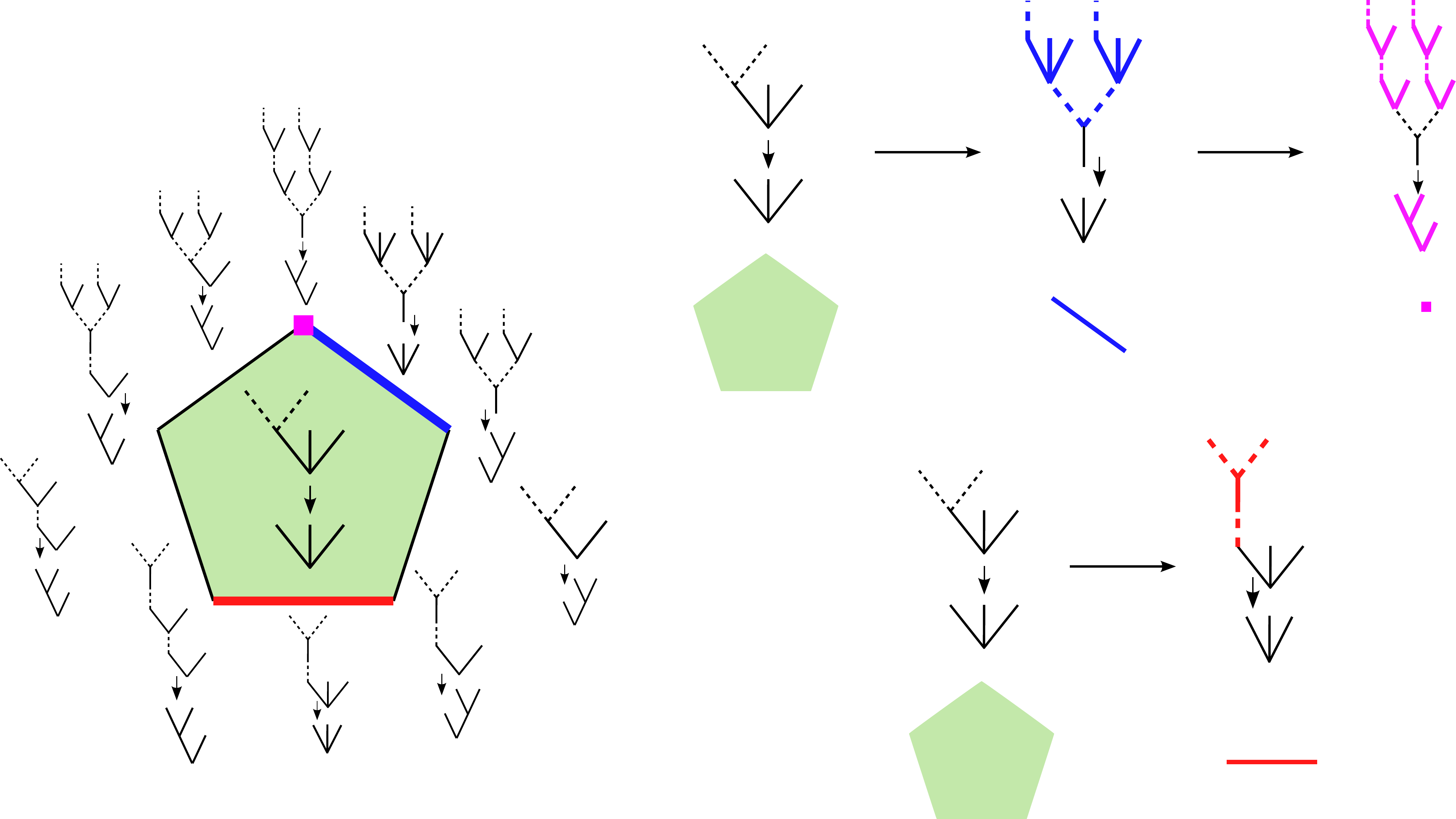
\label{fig:tree_move_examples}
\end{figure}
\noindent
On the left are the eleven tree-pairs comprising $W_{200}^\tree$ --- suggestively overlaid on a pentagon, which encodes the poset structure which is defined just below.
The the three moves which we illustrate here can be thought of as going from the interior of the pentagon to the bottom red edge, resp.\ the interior to the upper-right blue edge, resp.\ the upper-right blue edge to the top mauve vertex.
On the right we illustrate these moves, which are of type 1 resp.\ 3 resp.\ 2 (these numbers label the arrows corresponding to the moves).
We color the portions of the tree-pairs which have been altered by the move.
\null\hfill$\triangle$
\end{example}


\begin{lemma}
\label{lem:Wn_moves}
If $2T'$ is the result of making a move of type 1, 2, or 3 to $2T$, then $d(2T') = d(2T)-1$.
\end{lemma}

\begin{proof}
Recall the definition of $d$:
\begin{align}
|\bn| + r - \#\!V_\comp^1(T_b) - \#\!(T_s)_\inte - 2.
\end{align}
\begin{itemize}
\item In a type-1 move, one 1-seam component forms (where the points collide) and $T_s$ does not change.

\item In a type-2 move, no new 1-seam components form; one new interior vertex in $T_s$ forms.

\item In a type-3 move, one 1-seam component forms (the former $k$-seam component) and the seam tree does not change.
\end{itemize}
\end{proof}

\begin{definition}
\label{def:Wn_tree}
Define $W_\bn^\tree$ as a poset by declaring $2T' < 2T$ if there is a finite sequence of moves that transforms $2T$ into $2T'$.
The poset structure is well-defined by the same argument as in Def.-Lem.~\ref{deflem:Krtree_poset}.
\null\hfill$\triangle$
\end{definition}

\begin{lemma}
\label{lem:WnKn}
For any $n \geq 1$, there is a poset isomorphism $W_n^\tree \simeq K_n^\tree$.
There are also isomorphisms $W_{n,0}^\tree \simeq J_n \simeq W_{0,n}^\tree$, where $J_n$ is the $(n-1)$-dimensional multiplihedron.
\end{lemma}

\begin{proof}
Define a map $W_n^\tree \to K_n^\tree$ like so: given a stable tree-pair $T_b \sr{f}{\to} T_s$, send it to the result of collapsing all the solid edges in $T_b$, then converting all the dashed edges to solid ones.
A straightforward check shows that this map is well-defined and respects the partial order.
An inverse $K_n^\tree \to W_n^\tree$ is given like so: given a stable RRT $T$, convert its edges to dashed ones, then insert a solid edge at every interior vertex.
Here is an illustration of this correspondence, where the stable tree-pair on the left lies in $W_6^\tree$ and the RRT on the right lies in $K_6^\tree$:
\begin{figure}[H]
\centering
\def\svgwidth{0.4\columnwidth}
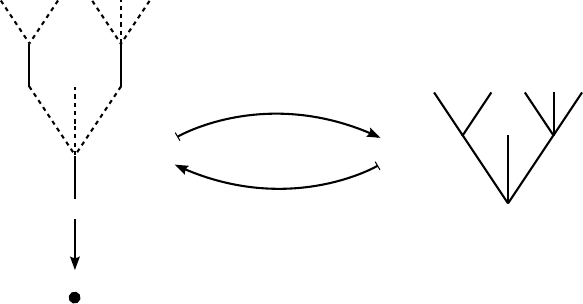
\end{figure}

The identifications $W_{n,0}^\tree \simeq J_n \simeq W_{0,n}^\tree$ are explained in Rmk.~4.4 of version 2 of \cite{bw:compactness}.
\end{proof}

\begin{remark}
There are also poset isomorphisms $W^\tree_{\tiny \underbrace{0,\ldots,0,1,0,\ldots,0}_n} \simeq K^\tree_n$; we leave the details of these isomorphisms to the reader.
\null\hfill$\triangle$
\end{remark}

\subsection{The poset \texorpdfstring{$W_\bn^\br$}{Wnbr} of 2-bracketings}
\label{ss:Wnbr_construction}

We now define the notion of a 2-bracketing, which will allow us to define the model $W_\bn^\br$.
These definitions are somewhat opaque, so we give some motivation after the definitions.

\begin{definition}
\label{def:2bracket}
A \emph{2-bracket of $\bn$} is a pair $\btB = (B, (2B_i))$\label{p:btB} consisting of a 1-bracket $B \subset \{1,\ldots,r\}$ and a consecutive subset $2B_i \subset \{1,\ldots,n_i\}$ for every $i \in B$ such that at least one $2B_i$ is nonempty.
We write $\btB' \subset \btB$ if $B' \subset B$ and $2B_i' \subset 2B_i$ for every $i \in B'$, and we define $\pi(B,(2B_i)) \coloneqq B$.
\null\hfill$\triangle$
\end{definition}

\begin{definition}
\label{def:Wn_br}
A \emph{2-bracketing of $\bn$} is a pair $(\sB, \stB)$,\label{p:sBstB} where $\sB$ is a 1-bracketing of $r$ and $\stB$ is a collection of 2-brackets of $\bn$ that satisfies these properties:
\begin{itemize}
\item[] {\sc (1-bracketing)} For every $\btB \in \stB$, $\pi(\btB)$ is contained in $\sB$.

\medskip

\item[] {\sc (2-bracketing)}
Suppose that $\btB, \btB'$ are elements of $\stB$, and that for some $i_0 \in \pi(\btB) \cap \pi(\btB')$, the intersection $2B_{i_0} \cap 2B_{i_0}'$ is nonempty.
Then either $\btB \subset \btB'$ or $\btB' \subset \btB$.

\medskip

\item[] {\sc (root and marked points)} $\stB$ contains $(\{1,\ldots,r\},(\{1,\ldots,n_1\},\ldots,\{1,\ldots,n_r\}))$ and every 2-bracket of $\bn$ of the form $(\{i\},(\{j\}))$.
\end{itemize}

\noindent For any $B_0 \in \sB$, write $\stB_{B_0} \coloneqq \{(B,(2B_i)) \in \stB \:|\: B = B_0\}$.

\begin{itemize}
\item[] {\sc (marked seams are unfused)}
\begin{itemize}
\item For any $B_0 \in \sB$ and for any $i \in B_0$, we have $\bigcup_{\btB \in \stB_{B_0}} 2B_i = \{1,\ldots,n_i\}$.

\item For every $\btB \in \stB_{B_0}$ for which there exists $\btB' \in \stB_{B_0}$ with $\btB' \subsetneq \btB$, and for every $i \in B_0$ and $j \in 2B_i$, there exists $\btB'' \in \stB_{B_0}$ with $\btB'' \subsetneq \btB$ and $2B''_i \ni j$.
\end{itemize}

\medskip

\item[] {\sc (partial order)} For every $B_0 \in \sB$, $\stB_{B_0}$ is endowed with a partial order with the following properties:
\begin{itemize}
\item $\btB, \btB' \in \stB_{B_0}$ are comparable if and only if $2B_i \cap 2B_i' = \emptyset$ for every $i \in B_0$.

\item For any $i$ and $j<j'$, we have $(\{i\},(\{j\})) < (\{i\},(\{j'\}))$.

\item For any 2-brackets $\btB^j \in \stB_{B_0}, \wt{\btB}^j \in \stB_{\wt B_0}, j \in \{1,2\}$ with $\wt{\btB^j}\subset \btB^j$, we have the implication
\begin{align}
\btB^1 < \btB^2 \implies \wt{\btB}^1 < \wt{\btB}^2.
\end{align}
\end{itemize}
\end{itemize}
We define $W_\bn^\br$ to be the set of 2-bracketings of $\bn$, with the poset structure defined by declaring $(\sB',\stB') < (\sB,\stB)$ if the containments
$\sB' \supset \sB$,
$\stB' \supset \stB$
hold and at least one of these containments is proper. \label{p:Wnbr}
\null\hfill$\triangle$
\end{definition}


\begin{example}
\label{ex:2bracketing_examples}
Define a 2-bracketing $(\sB,\stB) \in W_{11410}^\br$ like so:
\begin{align}
\sB \coloneqq &\{(1),(2),(3),(4),(5),(1,2),(3,4,5),(1,2,3,4,5)\},
\\
\stB \coloneqq &\Bigl\{\bigl((1),((a))\bigr),
\bigl((2),((b))\bigr),
\bigl((3),((c))\bigr),
\bigl((3),((d))\bigr),
\bigl((3),((e))\bigr),
\bigl((3),((f))\bigr),
\bigl((4),((g))\bigr),
\nonumber\\
&\quad \bigl((1,2),((a),())\bigr),
\bigl((1,2),((),(b))\bigr),
\bigl((1,2),((a),(b))\bigr),
\nonumber\\
&\quad \bigl((3,4,5),((c),(g),())\bigr),
\bigl((3,4,5),((d),(),())\bigr),
\bigl((3,4,5),((f,e),(),())\bigr),
\nonumber\\
&\quad \bigl((1,2,3,4,5),((a),(b),(c),(g),())\bigr),
\bigl((1,2,3,4,5),((),(),(f,e,d),(),())\bigr),
\nonumber\\
&\hspace{2.75in}
\bigl((1,2,3,4,5),((a),(b),(f,e,d,c),(g),())\bigr),
\nonumber
\end{align}
subject to the partial orders defined by the following relations:
\begin{gather}
\bigl((3),((f))\bigr) < \bigl((3),((e))\bigr) < \bigl((3),((d))\bigr) < \bigl((3),((c))\bigr),
\\
\bigl((1,2),((a),())\bigr) < \bigl((1,2),((),(b))\bigr),
\nonumber\\
\bigl((3,4,5),((f,e),(),())\bigr) < \bigl((3,4,5),((d),(),())\bigr) < \bigl((3,4,5),((c),(g),())\bigr),
\nonumber\\
\bigl((1,2,3,4,5),((),(),(f,e,d),(),())\bigr) < \bigl((1,2,3,4,5),((a),(b),(c),(g),())\bigr).
\nonumber
\end{gather}
Here we have denoted each 2-bracket $(B,(2B_i))$ to have $2B_i$ a subsequence of $(a)$, $(b)$, $(f,e,d,c)$, or $(g)$, for $i = 1,2,3,4,5$; this alternate notation is easier to parse in examples.

The presentation of $(\sB,\stB)$ just given is obviously cumbersome.
It is more convenient to depict 2-bracketings in the following pictorial format:
\begin{figure}[H]
\centering
\def\svgwidth{0.15\columnwidth}
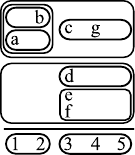
\label{fig:2-bracketing_example}
\end{figure}
\noindent Here the 1-brackets in $\sB$ are shown on the bottom row.
The 2-brackets are shown above the dividing line.
It is important to note that the 2-brackets come with a width, which indicates the 1-brackets they map to under $\pi$, and this is incorporated in the picture.
Moreover, the partial orders are reflected like so: for $\btB_1, \btB_2 \in \stB_{B_0}$, the inequality $\btB_1>\btB_2$ holds if and only if $\btB_1$ appears above $\btB_2$.
We have not depicted the 1-bracket $(1,2,3,4,5)$, the 2-bracket $\bigl((1,2,3,4,5),((a),(b),(f,e,d,c),(g),())\bigr)$, or the 1-brackets resp.\ 2-brackets of the form $(i)$ resp.\ $\bigl((i),((j))\bigr)$: these must be included in $(\sB,\stB)$ according to the $\textsc{(roots and leaves)}$ axiom, so it would not add any information to include them in the picture.
\null\hfill$\triangle$
\end{example}

\begin{example}
In this example, we discuss several invalid variants of the (valid) 2-bracketing from the last example.
Consider the five supposed 2-bracketings in the following figure:
\begin{figure}[H]
\centering
\def\svgwidth{0.875\columnwidth}
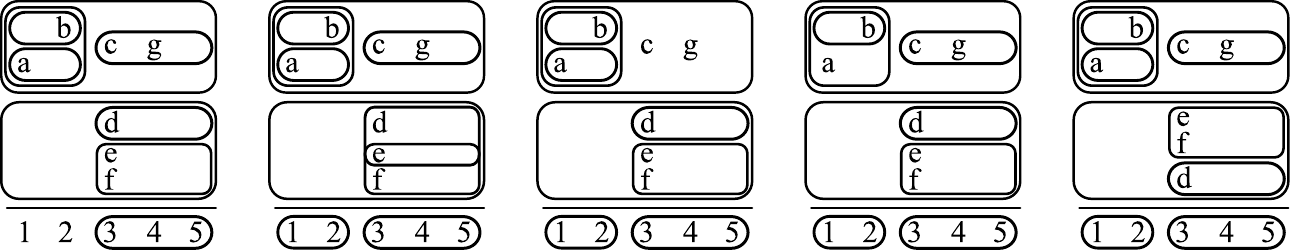
\label{fig:2-bracketing_nonexamples}
\end{figure}
Here is why these are invalid 2-bracketings, from left to right:
\begin{itemize}
\item We have deleted $(1,2)$ from $\sB$.
As a result, \textsc{(1-bracketing)} is not satisfied.

\item We have modified $\stB$ by replacing $\bigl((3,4,5),((d),(),())\bigr)$ with $\bigl((3,4,5),((e,d),(),())\bigr)$.
As a result, \textsc{(2-bracketing)} is not satisfied.

\item Here we have removed the 2-bracket $\bigl((3,4,5),((c),(g),())\bigr)$ from $\stB$.
As a result, we have
\begin{align}
\bigcup_{\btB \in \stB_{(3,4,5)}} 2B_3 = (f,e,d) \subsetneq (f,e,d,c),
\qquad
\bigcup_{\btB \in \stB_{(3,4,5)}} 2B_4 = () \subsetneq (g),
\end{align}
so the first part of \textsc{(marked seams are unfused)} is violated.

\item We have removed $\bigl((1,2),((a),())\bigr)$ from $\stB$.
This violates the second part of \textsc{(marked seams are unfused)}: in the notation of that condition, set $\btB \coloneqq \bigl((1,2),((a),(b))\bigr)$, $\btB' \coloneqq \bigl((1,2),((),(b))\bigr)$, $i \coloneqq 1$, and $j \coloneqq a$.

\item In the fifth non-example, we have modified the partial order on $\stB_{(1,2,3,4,5)}$ by declaring
\begin{align}
\bigl((3,4,5),((d),(),())\bigr) < \bigl((3,4,5),((f,e),(),())\bigr) < \bigl((3,4,5),((c),(g),())\bigr).
\end{align}
This, together with the inequality $\bigl((3),((e))\bigr) < \bigl((3),((d))\bigr)$, contradicts the third part of \textsc{(partial order)}.
\end{itemize}
\null\hfill$\triangle$
\end{example}


\begin{remark}[Motivation for the definition of 2-bracketings]
Recall from \S\ref{ss:motivation} that a 2-bracketing is intended to indicate the bubbling structure of a nodal witch curve.
That is, the 1-bracketing $\sB$ indicates how the seams have collided; each 2-bracket $\btB \in \stB$ corresponds to a bubble in the nodal witch curve, with $\pi(\btB)$ indicating the seams present on the bubble and the fashion in which these seams have collided, and $2B_i$ indicating the marked points which appear on the $i$-th seam and which are either on the present bubble or appear above this bubble (i.e., further from the main component).
The properties {\sc (1-bracketing)}, {\sc (2-bracketing)}, and {\sc (root and marked points)} are straightforward enough: {\sc (1-bracketing)} says that the collisions of seams on the various bubbles in the tree are controlled by a single 1-bracketing; {\sc (2-bracketing)} says that if a single marked point appears above two different bubbles, then one of these bubbles must be above the other; and {\sc (root and marked points)} is a consequence of the fact that the root corresponds to the 2-bracket $(\{1,\ldots,r\},(\{1,\ldots,n_1\},\ldots,\{1,\ldots,n_r\}))$ and that the $j$-th marked point on the $i$-th seam corresponds to the 2-bracket $(\{i\},(\{j\}))$.
{\sc (marked seams are unfused)} guarantees that marked points may only appear on unfused seams, which is a result of the fact that in our putative compactification $\ol{2\cM}_\bn$, when a collection of seams collide, wherever there is a marked point at the instant of this collision, a bubble must form.
Finally, {\sc (partial order)} reflects the fact that in a bubble tree, on each seam of each bubble there is an ordering of the marked and nodal points.
\null\hfill$\triangle$
\end{remark}

\subsection{We prove that \texorpdfstring{$W_\bn^\tree$}{Wntree}, \texorpdfstring{$W_\bn^\br$}{Wnbr} coincide}
\label{ss:Wn_iso}


In this subsection we finally define $W_\bn$, as well as the forgetful map $W_\bn \to K_r$.

\begin{deflem}
\label{def:Wn}
\label{p:forgetful}
Define $W_\bn \coloneqq W_\bn^\tree = W_\bn^\br$. \label{p:Wn}
The forgetful map $\pi\colon W_\bn \to K_r$ is defined in the two models like so: $\pi^\tree\colon W_\bn^\tree \to K_r^\tree$ sends $T_b \sr{f}{\to} T_s$ to $T_s$, and $\pi^\br\colon W_\bn^\br \to K_r^\br$ sends $(\sB,\stB)$ to $\sB$.
\end{deflem}

\begin{proof}
We need to show that $W_\bn^\tree$ and $W_\bn^\br$ are isomorphic posets, and that this isomorphism intertwines $\pi^\tree$ and $\pi^\br$.
The isomorphism $W_\bn^\tree \simeq W_\bn^\br$ is exactly the content of Thm.~\ref{thm:iso} below, and it is evident from the definition of this isomorphism that the two forgetful maps are intertwined.
\end{proof}


We now turn to the proof of the main theorem of this section.

\begin{theorem}[Equivalence of the two models for $W_\bn$]
\label{thm:iso}
For any $r\geq 1$ and $\bn \in \bZ_{\geq0}^r\setminus\{\bzero\}$, $W_\bn^\tree$ and $W_\bn^\br$ are isomorphic posets.
\end{theorem}

\begin{proof}
We construct a bijection $2\nu\colon W_\bn^\tree \to W_\bn^\br$ in Def.-Lem.~\ref{deflem:Wn_models_iso} and show that it respects the partial orders in Lemma~\ref{lem:Wn_models_orders}.
\end{proof}

The notion of a $2T$-bracket will be central in the construction of $2\nu\colon W_\bn^\tree \to W_\bn^\br$ in Def.-Lem.~\ref{deflem:Wn_models_iso}.


\begin{definition}
\label{def:2T_bracket}
Fix a stable tree-pair $2T = T_b \sr{f}{\to} T_s \in W_\bn^\tree$.
A \emph{$2T$-bracket} is a 2-bracket $\btB = (B,(2B_i))$ of $\bn$ with the property that for some $\alpha \in V_\comp(T_b) \sqcup V_\mk(T_b)$, we have
\begin{align}
B = B(f(\alpha)), \quad 2B_i = \bigl\{j \:|\: \mu_{ij}^{T_b} \in (T_s)_\alpha\bigr\}.
\end{align}
We denote this bracket by $\btB(\alpha) = (B(f(\alpha)),(2B_i(\alpha)))$.\label{p:btBalpha}
\null\hfill$\triangle$
\end{definition}

\noindent Note that {\sc(stability)} implies that $\alpha \in V_\comp(T_b) \sqcup V_\mk(T_b)$ is uniquely determined by $\btB(\alpha)$.
We denote the vertex corresponding to $\btB$ by $\alpha(\btB)$. \label{p:alpha_2B}

With this preparation, we are now ready to define the bijection $2\nu\colon W_\bn^\tree \to W_\bn^\br$.
This definition is rather technical, so we advise the reader to consult Ex.~\ref{ex:2nu_example} while reading this definition.

\begin{deflem}
\label{deflem:Wn_models_iso}
\label{p:2nu}
Define a map $2\nu\colon W_\bn^\tree \to W_\bn^\br$ to send a stable tree-pair $2T = T_b \sr{f}{\to} T_s$ to $(\sB(2T),\stB(2T))$, where $\sB(2T) \coloneqq \nu(T_s)$ and where $\stB(2T)$ is the set of $2T$-brackets, with the partial order on $\stB(2T)_{B(\beta)}, \beta \in T_s$ defined like so: fix $\wt\beta_1, \wt\beta_2 \in (V_\comp(T_b) \sqcup V_\mk(T_b)) \cap f^{-1}\{\beta\}$.
If $\btB(\wt\beta_1), \btB(\wt\beta_2)$ have $2B_i(\wt\beta_1) \cap 2B_i(\wt\beta_2) \neq \emptyset$ for some $i \in B(\beta)$, then define $\btB(\wt\beta_1), \btB(\wt\beta_2)$ to be incomparable.
Otherwise, define \label{p:alpha_triple}
\begin{align}
\gamma \coloneqq \alpha(\alpha_\root^{T_b},\wt\beta_1,\wt\beta_2) \coloneqq [\alpha_\root^{T_b},\wt\beta_1] \cap [\wt\beta_1,\wt\beta_2] \cap [\wt\beta_2,\alpha_\root^{T_b}] \in V_\seam(T_b)
\end{align}
as in \S D.2, \cite{ms:jh}, and for $j \in \{1,2\}$, define $\delta_j$ to be the element of $\incom(\gamma)$ with $\wt\beta_j \in (T_b)_{\gamma\delta_j}$.
Now define the order on $\btB(\wt\beta_1), \btB(\wt\beta_2) \in \stB(2T)_{B(\beta)}$ like so:
\begin{itemize}
\item if $\delta_1 <_\gamma \delta_2$, then $\btB(\wt\beta_1) < \btB(\wt\beta_2)$;

\item if $\delta_2 <_\gamma \delta_1$, then $\btB(\wt\beta_2) < \btB(\wt\beta_1)$.
\end{itemize}
Then $2\nu$ is bijective.
\end{deflem}

\begin{proof}
Throughout this proof we assume $\bn \neq (1)$, since in this case the bijectivity of $2\nu$ holds trivially.

\medskip

\noindent\emph{Step 1: If $T$ is an RRT and $\beta_1,\beta_2$ are any distinct non-root vertices, then
\begin{align}
\gamma \coloneqq \alpha(\alpha_\root,\beta_1,\beta_2) \coloneqq [\alpha_\root,\beta_1] \cap [\beta_1,\beta_2] \cap [\beta_2,\alpha_\root]
\end{align}
is the node furthest from the root satisfying $\beta_1,\beta_2 \in T_\gamma$.}

\medskip

\noindent Define $\Sigma$ to consist of those vertices $\delta$ of $T$ satisfying $\beta_1,\beta_2 \in T_\delta$.
The inclusion $\gamma \in [\alpha_\root,\beta_j]$ for $j \in \{1,2\}$ implies $\beta_j \in T_\gamma$, so $\gamma$ is an element of $\Sigma$.
Any two elements $\delta_1,\delta_2 \in \Sigma$ have the property that either $\delta_1 \in [\delta_2,\alpha_\root]$ or $\delta_2 \in [\delta_1,\alpha_\root]$, since both $\delta_1$ and $\delta_2$ lie in $[\alpha_\root,\beta_1]$.
Suppose that $\delta_1, \delta_2$ are distinct elements of $\Sigma$ and that $\delta_1$ lies in $[\delta_2,\alpha_\root]$.
Since $\beta_1,\beta_2$ lie in $T_{\delta_1}$ and $\delta_1$ lies in $[\delta_2,\alpha_\root]$, any path from $\beta_1$ to $\beta_2$ that passes through $\delta_1$ must pass through $\delta_2$ more than once; therefore $\delta_1$ does not lie in $[\beta_1,\beta_2]$, which implies $\gamma \neq \delta_1$.
It follows that $\gamma$ is the (unique) element of $\Sigma$ that is furthest from the root.

\medskip

\noindent\emph{Step 2: If $2T$ is a stable tree-pair of type $\bn$, then $2\nu(2T)$ is a 2-bracketing of $\bn$.}

\medskip

As discussed in the proof of Prop.~\ref{prop:Kr_iso}, $\sB(2T) = \nu(T_s)$ is a 1-bracketing of $r$.
For any $(B,(2B_i)) \in \stB(2T)$, it is clear that for every $i$, $2B_i \subset \{1,\ldots,n_i\}$ is consecutive.
The {\sc(stability)} axiom implies that there exists $i$ for which $2B_i$ is nonempty, so every element of $\stB(2T)$ is a 2-bracket.

Next, we justify two implicit assertions in the definition of $(\sB(2T), \stB(2T))$.
Specifically, we implicitly asserted (1) that for $\beta \in V(T_s)$ and $\wt\beta_1,\wt\beta_2 \in (V_\comp(T_b) \sqcup V_\mk(T_b)) \cap f^{-1}\{\beta\}$ with the property that $2B_i(\wt\beta_1) \cap 2B_i(\wt\beta_2) = \emptyset$ for every $i \in B(\beta)$, $\gamma \coloneqq \alpha(\alpha_\root^{T_b},\wt\beta_1,\wt\beta_2)$ lies in $V_\seam(T_b)$; and (2) that we have defined a partial order on every $\stB(2T)_{B(\beta)}$.
\begin{itemize}
\item[(1)] Suppose for a contradiction that $\gamma \in V_\comp(T_b) \sqcup V_\mk(T_b)$, and recall from Step 1 that $\gamma$ can be interpreted as the node furthest from the root with $\wt\beta_1,\wt\beta_2 \in (T_b)_\gamma$.
The assumption on $\wt\beta_1,\wt\beta_2$ implies that they are not the same vertex, hence $\gamma$ cannot lie in $V_\mk(T_b)$; therefore $\gamma \in V_\comp(T_b)$.
For $j \in \{1,2\}$, define $\eps_j$ to be the element of $\incom(\gamma)$ with $\wt\beta_j \in (T_b)_{\gamma\eps_j}$.
By the definition of $\gamma$, $\eps_1$ and $\eps_2$ cannot coincide.
In particular, $\#\!\incom(\gamma) \geq 2$, so $f$ must map the incoming edges of $\gamma$ bijectively onto the incoming edges of $f(\gamma)$.
Therefore $f(\eps_1)$ and $f(\eps_2)$ are distinct elements of $\incom(f(\gamma))$.
A given edge in $T_b$ is either contracted by $f$ or mapped to an edge in an orientation-preserving fashion, so we must have $f(\wt\beta_1) \neq f(\wt\beta_2)$, in contradiction to the assumption $f(\wt\beta_1) = \beta = f(\wt\beta_2)$.
Therefore our implicit assertion $\gamma \in V_\seam(T_b)$ was justified.

\item[(2)] If $2B_i(\wt\beta_1) \cap 2B_i(\wt\beta_2) = \emptyset$ for every $i \in B(\beta)$, then $\wt\beta_1$ and $\wt\beta_2$ must be distinct; therefore our relation is antireflexive.
Next, fix $\wt\beta_1,\wt\beta_2,\wt\beta_3$ with $\btB(\wt\beta_1) < \btB(\wt\beta_2)$ and $\btB(\wt\beta_2) < \btB(\wt\beta_3)$.
Define \begin{align}
\gamma_{ij} \coloneqq \alpha(\alpha_\root^{T_b},\wt\beta_i,\wt\beta_j), \quad i,j \in \{1,2,3\}.
\end{align}
Then we have either (a) $\gamma_{13} = \gamma_{12}$ and $\gamma_{23} \in (T_b)_{\gamma_{13}}$ or (b) $\gamma_{13} = \gamma_{23}$ and $\gamma_{12} \in (T_b)_{\gamma_{13}}$.
(Indeed, $\gamma_{12}$ and $\gamma_{23}$ are both in $[\wt\beta_2,\alpha_\root^{T_b}]$, so either $\gamma_{12} \in (T_b)_{\gamma_{23}}$ or $\gamma_{23} \in (T_b)_{\gamma_{12}}$.
Suppose $\gamma_{12} \in (T_b)_{\gamma_{23}}$.
Then $\wt\beta_1,\wt\beta_3 \in (T_b)_{\gamma_{23}}$.
Moreover, $\gamma_{23}$ is the furthest vertex from the root with this property: if there exists $\zeta \in \incom(\gamma_{23})$ such that $(T_b)_\zeta$ contains $\wt\beta_1$ and $\wt\beta_3$, then $(T_b)_\zeta$ also contains $\gamma_{12}$ and therefore $\wt\beta_2$, contradicting the fact that $\gamma_{23}$ is the furthest vertex from the root with $\wt\beta_2, \wt\beta_3 \in (T_b)_{\gamma_{23}}$.)
Suppose (a) holds.
If $\gamma_{12} = \gamma_{13} = \gamma_{23}$, then there are $\delta_1,\delta_2,\delta_3 \in \incom(\gamma_{13})$ with $\wt\beta_j \in (T_b)_{\delta_j}$ for all $j$.
By hypothesis, $\delta_1 <_{\gamma_{13}} \delta_2$ and $\delta_2<_{\gamma_{13}}\delta_3$, hence $\delta_1 <_{\gamma_{13}} \delta_3$, hence $\btB(\wt\beta_1) < \btB(\wt\beta_3)$ as desired.
On the other hand, suppose (a) holds and $\gamma_{23} \in (T_b)_{\gamma_{13}} \setminus \gamma_{13}$.
Define $\delta_1, \delta_{23} \in \incom(\gamma_{13})$ by $\wt\beta_1 \in (T_b)_{\delta_1}$, $\gamma_{23} \in (T_b)_{\delta_{23}}$.
The hypothesis $\btB(\wt\beta_1) < \btB(\wt\beta_2)$ implies $\delta_1 <_{\gamma_{13}} \delta_{23}$, hence $\btB(\wt\beta_1) < \btB(\wt\beta_3)$ as desired.
A similar argument applies if (b) holds, so our putative partial order is transitive.
\end{itemize}

Finally, we verify the axioms of a 2-bracketing.
\begin{itemize}
\item[] {\sc(1-bracketing)} Fix $\btB = (B,(2B_i)) \in \stB(2T)$.
Then $\alpha(\btB) \in V_\comp(T_b)$ has the property that $B$ is the set of indices of incoming leaves of $f(\alpha)$, which in turn is a 1-bracket in $\sB(2T)$; therefore $B \in \sB(2T)$.

\medskip

\item[] {\sc(2-bracketing)} Suppose that $\btB, \btB' \in \stB(2T)$ have the property that for some $i_0 \in B \cap B'$, $2B_{i_0} \cap 2B_{i_0}' \neq \emptyset$, and denote $\alpha \coloneqq \alpha(\btB), \alpha' \coloneqq \alpha(\btB')$.
Choose $j \in 2B_{i_0} \cap 2B_{i_0}'$.
By assumption, the path from $\mu_{i_0j}^{T_b}$ to $\alpha_\root^{T_b}$ passes through both $\alpha$ and $\alpha'$, so we must either have $\alpha \in (T_b)_{\alpha'}$ or $\alpha' \in (T_b)_\alpha$.
In the first case we must have $\btB \subset \btB'$, and similarly in the second case.

\medskip

\item[] {\sc(root and marked points)} Since $f(\alpha_\root^{T_b}) = \alpha_\root^{T_s}$, the 2-bracket corresponding to $\alpha_\root^{T_b}$ is $\bigl(\{1,\ldots,r\},(\{1,\ldots,n_1\},\ldots,\{1,\ldots,n_r\})\bigr)$.
On the other hand, the 2-bracket corresponding to $\mu_{ij}^{T_b}$ is $(\{i\},(\{j\}))$.

\medskip

\item[] {\sc(marked seams are unfused)}
\begin{itemize}
\item Fix $\rho \in T_s$, $i \in B(\rho)$, and $j \in \{1,\ldots,n_i\}$; we must produce $\alpha_0 \in V_\comp(T_b) \cap f^{-1}\{\rho\}$ with $2B_i(\alpha_0) \ni j$.
Choose a path from $\mu^{T_b}_{ij}$ to $\alpha_\root^{T_b}$.
By examining the image of this path in $T_s$, we see that some element in the path must lie in $V_\comp(T_b)\cap f^{-1}\{\rho\}$, and we can define $\alpha_0$ to be this element.

\item Fix $B(\rho) \in \sB(2T)$; $\btB(\alpha), \btB(\alpha') \in \stB(2T)_{B(\rho)}$ with $\btB(\alpha') \subsetneq \btB(\alpha)$; $i \in B(\rho)$; and $j \in 2B_i(\alpha) \setminus 2B_i(\alpha')$.
We must produce $\btB(\alpha'') \in \stB_{B(\rho)}$ with $\btB(\alpha'') \subsetneq \btB(\alpha)$ and $2B_i(\alpha'') \ni j$.
The containment $\btB(\alpha') \subsetneq \btB(\alpha)$ and the inclusions $\btB(\alpha),\btB(\alpha') \in \stB(2T)_{B(\rho)}$ imply $\alpha \in V_\comp^1(T_b)$.
Denote the incoming neighbor of $\alpha$ by $\beta$; we may now choose $\alpha'' \in \incom(\beta)$ to have the property that $(T_b)_{\alpha''}$ includes $\mu_{ij}^{T_b}$.
\end{itemize}

\medskip

\item[] {\sc(partial order)} Earlier we endowed every $\stB_{B(\beta)}, \alpha \in T_s$ with a partial order.
\begin{itemize}
\item It is an immediate consequence of our definition of the partial order on $\stB_{B(\beta)}$ that $\btB(\wt\beta_1), \btB(\wt\beta_2)$ are comparable if and only if $2B_i \cap 2B_i' = \emptyset$ for every $i \in B(\beta)$.

\item For any $i \in \{1,\ldots,r\}$ and $j,j'$ with $1 \leq j < j' \leq n_i$, it is clear that $(i,(\{j\})) < (i,(\{j'\}))$ in the partial order on $\stB(2T)_{\{i\}}$.

\item Fix $\btB(\alpha^j) \in \stB(2T)_{B(\rho)}$, $\btB(\wt\alpha^j) \in \stB(2T)_{B(\wt\rho)}$ for $j \in \{1,2\}$ with $\btB(\wt\alpha^j) \subset \btB(\alpha^j)$ and $\btB(\alpha^1) < \btB(\alpha^2)$.
We must show $\btB(\wt\alpha^1) < \btB(\wt\alpha^2)$.
The inclusions $\btB(\wt\alpha^j) \subset \btB(\alpha^j)$ for $j \in \{1,2\}$ are equivalent to the inclusions $\wt\alpha^j \in (T_b)_{\alpha^j}$.
From this it is easy to see that $\btB(\wt\alpha^1) < \btB(\wt\alpha^2)$.
\end{itemize}
\end{itemize}

\medskip

\noindent\emph{Step 3: We define a putative inverse $2\tau\colon W_\bn^\br \to W_\bn^\tree$.}

\medskip

\noindent Fix $(\sB,\stB) \in W_\bn^\br$.
Define $T_s \coloneqq \tau(\sB)$.
Towards the definition of $T_b$, define the following sets:
\begin{gather}
V_\mk \coloneqq \left\{\bigl(\{i\},(\{j\})\bigr) \:\left|\: {{1 \leq i \leq r} \atop {1 \leq j \leq n_i}}\right.\right\} \ni \mu_{ij}^{T_b},
\quad
V_\comp \coloneqq \stB \setminus V_\mk \ni \alpha_{\btB,\comp},
\\
V_\seam \coloneqq V_\seam^1\sqcup V_\seam^{\geq2},
\quad
V_\seam^1 \coloneqq \{\btB \in \stB \:|\: \exists\: \btB' \in \stB_{\pi(\btB)}\colon \btB' \subsetneq \btB\} \ni \alpha_{\btB,\pi(\btB),\seam}, \nonumber
\\
V_\seam^{\geq2} \coloneqq \left\{(\btB,B') \in \stB \times \sB \:\left|\: {{\not\!\exists\: \btB'' \in \stB_{\pi(\btB)}\colon \btB'' \subsetneq \btB}\atop{B' \subsetneq \pi(\btB), \: \not\!\exists B'' \in \sB: B' \subsetneq B'' \subsetneq \pi(\btB)}}\right.\right\} \ni \alpha_{\btB,B',\seam} \nonumber.
\end{gather}
Now define the vertices and incoming neighbors in $T_b$ by
\begin{gather}
\label{eq:Tb_in}
V \coloneqq V_\comp \sqcup V_\seam \sqcup V_\mk, \quad \alpha_\root \coloneqq \alpha_{(\{1,\ldots,r\},(\{1,\ldots,n_1\},\ldots,\{1,\ldots,n_r\})),\comp},
\\
\incom(\mu_{ij}^{T_b}) \coloneqq \emptyset,
\quad
\incom(\alpha_{\btB,\comp}) \coloneqq \begin{cases}
\{\alpha_{\btB,\pi(\btB),\seam}\}, & \exists\: \btB' \in \stB_{\pi(\btB)} \colon \btB' \subsetneq \btB, \\
\left\{\alpha_{\btB,B',\seam} \:\left|\: {{B' \subsetneq \pi(\btB)}\atop{\not\exists B'' \in \sB: B' \subsetneq B'' \subsetneq \pi(\btB)}}\right.\right\}, & \text{otherwise},
\end{cases}
\nonumber
\\
\incom(\alpha_{\btB,B',\seam}) \coloneqq \left\{\alpha_{\btB'',\comp} \:\left|\: {{\pi(\btB'') = B', \btB'' \subsetneq \btB,} \atop {\not\!\exists\: \btB''' \in \stB \colon \btB'' \subsetneq \btB''' \subsetneq \btB}} \right.\right\}
\cup \hspace{1.5in}
\nonumber
\\
\hspace{2.5in} \cup
\left\{\mu_{ij}^{T_b} \:\left|\: {{B'=\{i\}, \bigl(\{i\},(\{j\})\bigr) \subsetneq \btB,} \atop {\not\!\exists\: \btB'' \in \stB \colon \bigl(\{i\},(\{j\})\bigr) \subsetneq \btB'' \subsetneq \btB}} \right.\right\},
\nonumber
\end{gather}
where the incoming edges of $\alpha_{\btB,\comp}$ are solid and the incoming edges of $\alpha_{\btB,B',\seam}$ are dashed.
For $\alpha_{\btB,\comp}$ for which there does not exist $\btB' \in \stB_{\pi(\btB)}$ with $\btB' \subsetneq \btB$, order the incoming neighbors $\left\{\alpha_{\btB,B',\seam} \:\left|\: {{B' \subsetneq \pi(\btB)}\atop{\not\exists B'' \in \sB: B' \subsetneq B'' \subsetneq \pi(\btB)}}\right.\right\}$ according to the order on the incoming neighbors of $\pi(\btB)$ in $T_s = \tau(\sB)$.
For $\alpha_{\btB,B',\seam}$, order the incoming neighbors according to the partial order on $\stB_{\pi(\btB)}$.
Finally, define $f\colon T_b \to T_s$ like so:
\begin{gather}
f(\mu_{ij}^{T_b}) \coloneqq \lambda_i^{T_s},
\quad
f(\alpha_{\btB,\comp}) \coloneqq \alpha(\pi(\btB)),
\quad
f(\alpha_{\btB,B',\seam}) \coloneqq \alpha(B').
\end{gather}

There are a number of things we have to check in order to verify that $T_b \sr{f}{\to} T_s$ is a stable tree-pair.
For \eqref{eq:Tb_in} to define a RRT via Lemma~\ref{lem:RRT_in}, we must check conditions (1--3) in the statement of that lemma.
\begin{itemize}
\item[(1)] It is clear that no $\incom(\alpha)$ can contain $\alpha_\root$.

\item[(2)] Fix a non-root vertex $\alpha$ in $T_b$.
Depending on which type of vertex $\alpha$ is, we check that there exists a unique $\beta$ with $\incom(\beta) \ni \alpha$:
\begin{itemize}
\item[] {\bf $\bullet \: \mathbf{\alpha = \mu_{ij}^{T_b}}$.} The vertices $\beta$ with $\incom(\beta) \ni \mu_{ij}^{T_b}$ are exactly those $\alpha_{\btB,\{i\},\seam}$ with $\btB$ satisfying these properties:
\begin{itemize}
\item[(a)] $(i,(\{j\})) \subsetneq \btB$;

\item[(b)] either $\pi(\btB) = \{i\}$, or $\pi(\btB) \supsetneq \{i\}$ and no $B'' \in \sB$ has $\{i\} \subsetneq B'' \subsetneq \pi(\btB)$;

\item[(c)] no $\btB'' \ni \stB$ has $(i,(\{j\})) \subsetneq \btB'' \subsetneq \btB$.
\end{itemize}
Define $\Sigma$ to consist of those $\btB \in \stB$ that properly contain $(i,(\{j\}))$, and order $\Sigma$ by inclusion.
Since $\alpha$ is not the root, $\Sigma$ contains $\alpha_\root$ and is therefore not empty; by the {\sc(2-bracketing)} property of 2-bracketings, any two elements of $\Sigma$ are comparable under inclusion.
Therefore $\Sigma$ has a unique minimal element $\btB^0$.

I claim that $\btB^0$ is the unique element of $\stB$ satisfying (a--c).
Indeed, it is clear from its definition that $\btB^0$ satisfies (a) and (c).
If $\btB$ does not satisfy (b), then there exists $B'' \in \sB$ with $\{i\} \subsetneq B'' \subsetneq \pi(\btB^0)$.
By the {\sc(marked seams are unfused)} property of 2-bracketings, there exists $\btB'' \in \stB$ with $\pi(\btB'') = B''$ and $\btB''_i \ni j$.
This 2-bracket satisfies $(i,(\{j\})) \subsetneq \btB'' \subsetneq \btB^0$, which contradicts the definition of $\btB$; therefore $\btB$ satisfies (b).
On the other hand, suppose that $\btB$ satisfies (a--c).
(a) implies that $\btB$ lies in $\Sigma$, and (c) implies that $\btB$ is in fact the minimal element of $\Sigma$; therefore $\btB = \btB_0$.
We may conclude that $\beta \coloneqq \alpha_{\btB,\{i\},\seam}$ is the unique vertex in $T_b$ with $\incom(\beta) \ni \alpha$.

\item[] {\bf $\bullet \: \mathbf{\alpha = \alpha_{\btB,\comp}}$.} An argument similar to the one used in the case $\alpha = \mu_{ij}^{T_b}$ shows that there is a unique $\beta$ with $\incom(\beta) \ni \alpha$.

\item[] {\bf $\bullet \: \mathbf{\alpha = \alpha_{\btB,\mathit{B'},\seam}}$.} If there exists $\btB'' \in \stB_{\pi(\btB)}$ with $\btB'' \subsetneq \btB$, then $\alpha \in V_\seam^1$ and $B' = \pi(\btB)$.
Therefore $\beta \coloneqq \alpha_{\btB,\comp}$ is the unique vertex with $\incom(\beta) \ni \alpha$.

On the other hand, suppose that there does not exist such a $\btB''$.
Then $\alpha \in V_\seam^{\geq2}$, and $\beta \coloneqq \alpha_{\btB,\comp}$ is the unique vertex with $\incom(\beta) \ni \alpha$.
\end{itemize}

\item[(3)] Suppose that $\alpha_1, \ldots, \alpha_\ell \in V$ has $\ell \geq 2$ and $\alpha_j \in \incom(\alpha_{j+1})$ for every $j$.
It is clear from our verification of (2) that $\alpha_1$ is not the same as $\alpha_\ell$.
\end{itemize}
Now that we have shown that $T_b$ is well-defined as an RRT, we check the rest of the requirements on $T_b \sr{f}{\to} T_s$.
\begin{itemize}
\item
\begin{itemize}
\item We have defined $V$ as the union $V = V_\comp \sqcup V_\seam \sqcup V_\mk$, and the incoming and outgoing edges of the vertices are clearly the necessary type.
It is not hard to see that $\alpha_{\btB,\comp} \in V_\comp$ has at least one incoming edge: if there exists $\btB' \in \stB_{\pi(\btB)}$ with $\btB' \subsetneq \btB$, then $\#\!\incom(\alpha_{\btB,\comp}) = 1$.
Next, suppose that there is no such $\btB'$.
The incoming vertices of $\alpha_{\btB,\comp}$ are in correspondence with maximal elements of $\Sigma$, where $\Sigma$ consists of those elements $B'$ of $\sB$ with $B' \subsetneq \pi(\btB)$.
It follows from the {\sc (root and leaves)} property of 1-bracketings that $\#\!\incom(\alpha_{\btB,\comp}) \geq 2$.

\item {\sc(stability)} Fix $\alpha = \alpha_{\btB,\comp} \in V_\comp^1$.
Its unique incoming neighbor is $\beta\coloneqq \alpha_{\btB,\pi(\btB),\seam}$.
We must show that $\beta$ has at least 2 incoming neighbors.
The incoming neighbors of $\beta$ are in correspondence with the maximal elements of the set $\Sigma$ of $\btB'' \in \stB_{\pi(\btB)}$ with $\btB'' \subsetneq \btB$.
The fact that $\alpha_{\btB,\comp}$ lies in $V_\comp^1$ implies that $\Sigma$ is nonempty.
Choose $\btB^1$ to be any maximal element of $\Sigma$.
Now choose any $i, j$ with the property that $\btB_i \setminus \wt\btB_i$ contains $j$.
Define $\Sigma'$ to consist of those $\btB'' \in \Sigma$ with $\btB''_i \ni j$.
By {\sc(marked seams are unfused)}, $\Sigma'$ is nonempty; choose $\btB^2$ to be any maximal element of $\Sigma'$.
Then $\btB^1, \btB^2$ are distinct maximal elements of $\Sigma$.
This shows that $\beta$ has at least 2 incoming neighbors.

On the other hand, suppose that $\alpha = \alpha_{\btB,\comp}$ lies in $V_\comp^{\geq2}$.
Write $\btB = (B, (2B_i))$.
The incoming neighbors of $\alpha$ are in correspondence with maximal elements of the set $\Sigma$ of $B' \in \sB$ with $B' \subsetneq B$.
Not every $2B_i$ can be empty, so we may choose $i, j$ with the property that $2B_i$ contains $j$.
Define $\Sigma'$ to be the set of 1-brackets $B' \in \sB$ with $\{i\} \subset B' \subsetneq B$.
$\Sigma'$ contains $\{i\}$, hence is nonempty; define $B^1$ to be a maximal element of $\Sigma'$ (in fact, this determines $B^1$ uniquely).
Now define $\Sigma''$ to be the set of 2-brackets $\btB'' \in \stB_{B^1}$ with $(i,(\{j\})) \subset \btB'' \subsetneq \btB$.
$\Sigma''$ contains $(i,(\{j\}))$, hence is nonempty; define $\btB^2$ to be a maximal element of $\Sigma''$.
Then $\alpha = \alpha_{\btB,\comp}$ has $\beta \coloneqq \alpha_{\btB,B^1,\seam}$ as an incoming neighbor, and $\beta$ has $\alpha_{\btB^2,\comp}$ as an incoming neighbor.
\end{itemize}

\item By Prop.~\ref{prop:Kr_iso}, $T_s$ is an element of $K_r^\tree$.

\item It is clear that $f$ satisfies the necessary properties.
\end{itemize}

\medskip

\noindent\emph{Step 4: We verify that $2\nu$ and $2\tau$ are inverse to one another.}

\medskip

\noindent This can be shown by an argument similar to the one made in the proof of Prop.~\ref{prop:Kr_iso} to show that $\nu$ and $\tau$ are inverse to one another.
\end{proof}


\begin{example}
\label{ex:2nu_example}
In the following figure, we illustrate the definition of $2\nu$ as a map of sets:

\begin{figure}[H]
\centering
\def\svgwidth{0.65\columnwidth}
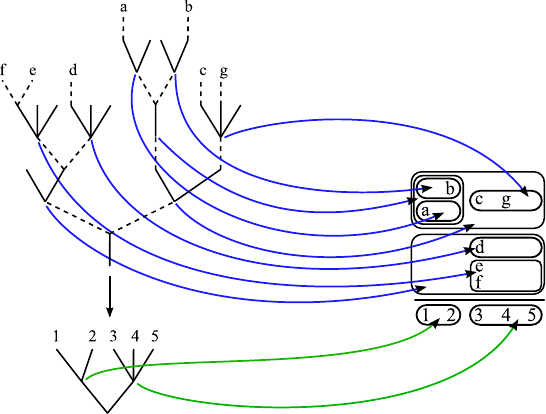
\label{fig:2nu_example}
\end{figure}

\noindent
On the left is the tree-pair in $W_{11410}^\tree$ we discussed in Ex.~\ref{ex:tree-pair_examples}, and on the right is the 2-bracketing in $W_{11410}^\br$ we discussed in Ex.~\ref{ex:2bracketing_examples}.
In fact, these objects are identified by $2\nu$.
Indeed, we see here how the elements of $V_\comp(T_b) \cup V_\mk(T_b)$ are sent to 2-brackets (indicated by blue arrows), and how the elements of $T_s$ are sent to 1-brackets (green arrows).
(We have omitted the blue and green arrows corresponding to $\rho_\root^{T_s}$, $\alpha_\root^{T_b}$, $V_\mk(T_b)$, and $T_s \setminus (T_s)_\inte$.)
The procedure for assigning a 2-bracket to an element $\alpha$ of $V_\comp(T_b) \cup V_\mk(T_b)$ is simple: the 2-bracket includes the elements of $V_\mk(T_b)$ lying above $\alpha$, and the projection of the 2-bracket includes the leaves of $T_s$ above $f(\alpha)$.

In the next figure we indicate, in the case of the same tree-pair $2T$, how the partial order on $2\nu(2T)$ is defined.
Specifically, we indicate why the inequalities $\btB(\gamma_1) < \btB(\gamma_2)$, $\btB(\delta_1) < \btB(\delta_2)$, and $\btB(\eps_1) < \btB(\eps_2)$ hold.
Here is the procedure, in the case of $\delta_1$ and $\delta_2$: draw (blue) paths downward from $\delta_1$ and $\delta_2$, until the paths intersect at a vertex $\alpha$.
At $\alpha$ --- necessarily an element of $V_\seam(T_b)$ --- note which elements of $\incom(\alpha)$ the two paths passed through.
Using the order on $\incom(\alpha)$, we obtain the inequality $\btB(\delta_1) < \btB(\delta_2)$.

\begin{figure}[H]
\centering
\def\svgwidth{0.65\columnwidth}
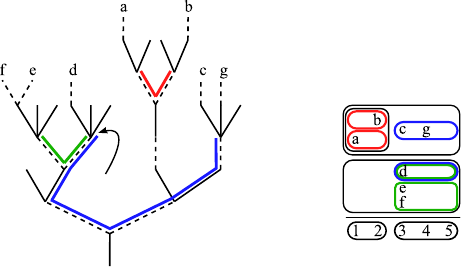
\label{fig:2nu_order_example}
\end{figure}
\null\hfill$\triangle$
\end{example}

In this lemma, we introduce the notion of a move on a 2-bracketing.
We say that $\stB'$ is the result of performing a move on $\stB$ if $(2\nu)^{-1}(\stB')$ is the result of performing a move on $(2\nu)^{-1}(\stB)$.


\begin{lemma}
\label{lem:Wn_models_orders}
The partial order on $W_\bn^\br$ coincides with the one induced by the partial order on $W_\bn^\tree$ and the isomorphism $2\nu\colon W_\bn^\tree \to W_\bn^\br$.
\end{lemma}

\begin{proof}
The nontrivial direction is to show that for $(\sB^1,\stB^1) \subsetneq (\sB^2,\stB^2)$, we can obtain $(\sB^2,\stB^2)$ from $(\sB^1,\stB^1)$ via a finite sequence of moves.
It is enough to show that there is a 2-bracketing $(\sB, \stB) \in W_\bn^\br$ satisfying the containments
\begin{align}
(\sB^1,\stB^1) \subset (\sB, \stB) \subset (\sB^2,\stB^2),
\end{align}
and such that either $(\sB,\stB)$ is the result of performing a single move on $(\sB^1,\stB^1)$, or $(\sB^2,\stB^2)$ is the result of performing a single move on $(\sB,\stB)$.
We produce such a 2-bracketing in the following, exhaustive, cases.
\begin{itemize}
\item[] {\bf Case 1: There exist $B^0 \in \sB^1$ and $\btB \in \stB^2_{B^0}\setminus \stB^1_{B^0}$, $\btB' \in \stB^2_{B^0}$ with $\btB' \subsetneq \btB$.}
We claim that $(\sB^2, \stB^2 \setminus \{\btB\})$ is a valid 2-bracketing.
The only property that does not obviously hold is {\sc(marked seams are unfused)}.
This is a consequence of the fact that $(\sB^2,\stB^2)$ has the {\sc(marked seams are unfused)} property.
$(\sB^2,\stB^2)$ is the result of performing a type-1 move on $(\sB^2,\stB^2\setminus\{\btB\})$, and the necessary containments hold:
\begin{align}
(\sB^1,\stB^1) \subset (\sB^2,\stB^2\setminus\{\btB\}) \subsetneq (\sB^2,\stB^2).
\end{align}


We illustrate this case in the following figure.
On the left are the 2-bracketings $(\sB^1,\stB^1)$, $(\sB,\stB)$, $(\sB^2,\stB^2)$, from left to right; on the right are the tree-pairs corresponding via $2\nu$ to these 2-bracketings.
$(\sB^2,\stB^2)$ is the result of performing a type-1 move on $(\sB,\stB)$, and we highlight in red the portion of the tree-pairs involved in this move.

\begin{figure}[H]
\centering
\def\svgwidth{0.975\columnwidth}
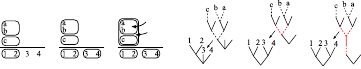
\label{fig:2nu_poset_map_example_1}
\end{figure}


\item[] {\bf Case 2: Case 1 does not hold, and there exist $B^0 \in \sB^1$ and $\btB \in \stB^2_{B^0}$, $\btB' \in \stB^2_{B^0}\setminus \stB^1_{B^0}$ with $\btB' \subsetneq \btB$.}
Fix such 2-brackets $\btB', \btB$.
Without loss of generality, we may assume that $\btB$ is minimal among 2-brackets in $\stB^2_{B^0}$ that properly contain $\btB'$.
The assumption that we are not in Case 1 implies that $\btB$ lies in $\stB^1_{B^0}$; this assumption also implies that $\btB'$ cannot properly contain any 2-bracket in $\stB^2_{B^0}$.
This and the minimality of $\btB$ implies that for any $i \in B^0$ and $j \in 2B'_i$, there is no $\btB'' \in \stB^1_{B^0}$ with $\btB'' \subsetneq \btB$ and $2B''_i \ni j$; it therefore follows from the {\sc(marked seams are unfused)} property of $(\sB^1,\stB^1)$ that there are no 2-brackets in $\stB^1_{B^0}$ that are properly contained in $\btB$.
The {\sc(marked seams are unfused)} property of $(\sB^2,\stB^2)$ implies that for every $i \in B^0$ and $j \in 2B^2_i$, there exists $\wt \btB \in \stB^2_{B^0}$ with $\wt\btB \subsetneq \btB$ and $\wt 2B_i \ni j$.
This, together with {\sc(2-bracketing)}, imply that if $\wt B^1, \ldots, \wt B^k \in \stB^2_{B^0}$ denote the maximal elements (with respect to inclusion) of $\stB^2_{B^0}$ that are properly contained in $\btB$, then these 2-bracketings satisfy $\btB = \bigsqcup_{i=1}^k \wt \btB^k$.
Therefore $(\sB^1,\stB^1\cup\{\wt\btB^1,\ldots,\wt\btB^k\})$ is the result of performing a single type-3 move on $(\sB^1,\stB^1)$, and the necessary containments hold:
\begin{align}
(\sB^1,\stB^1) \subsetneq (\sB^1,\stB^1\cup\{\wt\btB^1,\ldots,\wt\btB^k\}) \subset (\sB^2,\stB^2).
\end{align}


As in Case 1, we illustrate this procedure below.
Here we take $B^0 = (1,2,3,4)$, $\btB = \bigl((1,2,3,4),((c,b,a),(),(),())\bigr)$, and $\btB' = \bigl((1,2,3,4),((b,a),(),(),())\bigr)$.

\begin{figure}[H]
\centering
\def\svgwidth{1.0\columnwidth}
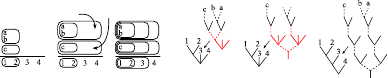
\label{fig:2nu_poset_map_example_2}
\end{figure}


\medskip

\item[] {\bf Case 3: Neither Case 1 nor Case 2 hold.}
The proper containment $(\sB^1,\stB^1) \subsetneq (\sB^2,\stB^2)$ implies that either $\sB^2\setminus \sB^1$ is nonempty or $\stB^2\setminus \stB^1$ is nonempty.
I claim that under the current assumptions, $\sB^2\setminus \sB^1$ must be nonempty.
Indeed, suppose that $\stB^2\setminus\stB^1$ is nonempty, and choose an element $\btB = (B, (2B_i))$.
The assumption that neither Case 1 nor Case 2 hold implies that there is no $\btB' \in \stB^2_B$ with either $\btB' \subsetneq \btB$ or $\btB \subsetneq \btB'$.
This, together with the {\sc (2-bracketing)} and {\sc(marked seams are unfused)} properties of $(\sB^1,\stB^1)$, imply that $B$ lies in $\sB^2\setminus \sB^1$.
We may conclude that $\sB^2\setminus\sB^1$ is nonempty.

Choose $B\in \sB^2\setminus\sB^1$, and note that the assumption that neither Case 1 nor Case 2 hold implies that any two elements of $\stB^2_B$ are disjoint.
Set $\sB \coloneqq \sB^2\setminus\{B\}$ and $\stB \coloneqq \stB^2\setminus \stB^2_B$.
Then the necessary containments hold, and $(\sB^2,\stB^2)$ is the result of performing a single type-2 move on $(\sB,\stB)$.


As above, we illustrate this case in the following figure.

\begin{figure}[H]
\centering
\def\svgwidth{0.95\columnwidth}
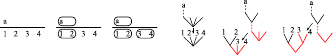
\label{fig:2nu_poset_map_example_3}
\end{figure}

\end{itemize}
\end{proof}


\section{Key properties of \texorpdfstring{$W_\bn$}{Wn}}
\label{sec:Wn_polytope}


In this section we establish several properties of $W_\bn$, collected in this paper's main theorem:

\begin{theorem}[Key properties of $W_\bn$]
\label{thm:main}
For any $r \geq 1$ and $\bn \in \bZ^r_{\geq0}\setminus\{\bzero\}$, the 2-associahedron $W_\bn$ is a poset, the collection of which satisfies the following properties:
\begin{itemize}
\item[] \textsc{(abstract polytope)} For $\bn \neq (1)$, $\wh{W_\bn} \coloneqq W_\bn \cup \{F_{-1}\}$ is an abstract polytope of dimension $|\bn| + r - 3$.

\item[] \textsc{(forgetful)} $W_\bn$ is equipped with forgetful maps $\pi\colon W_\bn \to K_r$, which are surjective maps of posets.

\item[] \textsc{(recursive)} For any stable tree-pair $2T = T_b \sr{f}{\to} T_s \in W_\bn^\tree$, there is an inclusion of posets
\begin{align}
\Gamma_{2T} \colon \prod_{
{\alpha \in V_\comp^1(T_b),}
\atop
{\incom(\alpha)=(\beta)}
} W_{\#\!\incom(\beta)}^\tree
\times
\prod_{\rho \in V_\inte(T_s)} \prod^{K_{\#\!\incom(\rho)}}_{
{\alpha\in V_\comp^{\geq2}(T_b)\cap f^{-1}\{\rho\},}
\atop
{\incom(\alpha)=(\beta_1,\ldots,\beta_{\#\!\incom(\rho)})}
}
\hspace{-0.25in} W^\tree_{\#\!\incom(\beta_1),\ldots,\#\!\incom(\beta_{\#\!\incom(\alpha)})}
\hra W_\bn^\tree,
\end{align}
where the superscript on one of the product symbols indicates that it is a fiber product with respect to the maps described in \textsc{(forgetful)}.
This inclusion is a poset isomorphism onto $\cl(2T) = (F_{-1},2T]$.
\end{itemize}
\end{theorem}

\begin{proof}
We prove the \textsc{(recursive)} and \textsc{(abstract polytope)} properties in Def.-Lem.~\ref{deflem:Gamma2T} and Thm.~\ref{thm:Wn_polytope}, respectively.
$W_\bn$ is a poset by its construction in Def.~\ref{def:Wn}, and the forgetful map from the same definition is evidently a surjection of posets.
\end{proof}


We now turn to the proof of the \textsc{(recursive)} property, which characterizes the closed faces of $W_\bn$ as products and fiber products of lower-dimensional 2-associahedra.
Toward this characterization, we show in the following lemma that for $2T' \in \cl(2T)$, certain vertices in $2T$ have avatars in $2T'$.

\begin{lemma}
\label{lem:avatars}
Fix $2T, 2T' \in W_\bn^\tree$ with $2T' \leq 2T$.
\begin{itemize}
\item For any $\rho \in T_s$, there exists a unique $\rho' \in T_s'$ satisfying $B(\rho') = B(\rho)$.

\item For any $\alpha \in V_\comp(T_b) \cup V_\mk(T_b)$ there exists a unique $\alpha' \in V_\comp(T_b') \cup V_\mk(T_b')$ satisfying $\btB(\alpha') = \btB(\alpha)$.
\end{itemize}
\end{lemma}

\begin{proof}
First, we prove the first statement.
The uniqueness of $\rho'$ is guaranteed by the stability condition, so it suffices to prove existence.
We do so by induction on $d(2T')$, starting with $d(2T') = d(2T)$ and counting down.
If $d(2T')=d(2T)$, then $T_s=T_s'$ and the statement holds trivially.
Next, suppose that for $\rho \in T_s$, we have proved the existence of $\rho' \in T_s'$ with $B(\rho') = B(\rho)$ for every $2T' \leq 2T$ with $d(2T')\geq d(2T)-k \geq 1$; we must show that there is $\rho' \in T_s'$ with $B(\rho') = B(\rho)$ for $2T' < 2T$ with $d(2T') = d(2T)-k-1$.
Choose $2T''$ with $2T' < 2T'' \leq 2T$ and $d(2T'') = d(2T)-k$, and denote by $\rho''$ the vertex in $T_s''$ with $B(\rho'')=B(\rho)$.
$2T'$ can be obtained from $2T''$ via a single move, so either $T_s' = T_s''$ or $T_s'$ can be obtained from $T_s''$ by performing the following modification to some solid corolla in $T_s''$, for $2\leq\ell<k$:

\begin{figure}[H]
\centering
\def\svgwidth{0.35\columnwidth}
\input{2T_move_2s.pdf_tex}
\end{figure}

\noindent In the former case, we can set $\rho'\coloneqq \rho''$.
In the latter case, we can identify $V(T_s') \simeq V(T_s'') \cup \{v_{\text{new}}\}$; if we set $\rho'$ to be the vertex in $T_s'$ corresponding via this identification to $\rho''$, then $B(\rho') = B(\rho'')$.

The second statement of the lemma can be proven similarly.
\end{proof}

Next, we show how this correspondence allows us to extract certain sub-tree-pairs from $2T'$.

\begin{deflem}
\label{deflem:subtreepairs}
Fix $2T, 2T' \in W_\bn^\tree$ with $2T' < 2T$.
\begin{itemize}
\item[(a)] Fix $\alpha \in V_\comp^1(T_b)$ with $\incom(\alpha) \eqqcolon (\beta)$ and $\incom(\beta) \eqqcolon (\gamma_1,\ldots,\gamma_k)$, and denote by $\alpha', \gamma_1', \ldots, \gamma_k'$ the vertices in $T_b'$ that correspond to $\alpha,\gamma_1,\ldots,\gamma_k$ via Lemma~\ref{lem:avatars}.
Define $(T_b')^\alpha$ to be the portion of $T_b'$ bounded by $\alpha'$ and $\gamma_1',\ldots,\gamma_k'$, and define $(T_s')^\alpha$ to be a single vertex.
Then $(2T')^\alpha\coloneqq (T_b')^\alpha \to (T_s')^\alpha$ is a stable tree-pair in $W_k^\tree$.

\item[(b)] For any $\rho \in V_\inte(T_s)$ with $\incom(\rho) \eqqcolon (\sigma_1,\ldots,\sigma_k)$, define $(T_s')^\rho$ to be the portion of $T_s'$ bounded by $\rho'$ and $\sigma_1',\ldots,\sigma_k'$, where we continue to use the notation of Lemma~\ref{lem:avatars}.
For any $\alpha \in V_\comp^{\geq2}(T_b)$ with $f(\alpha) = \rho$, denote $\incom(\alpha) \eqqcolon (\beta_1,\ldots,\beta_k)$ and $\incom(\beta_i) \eqqcolon (\gamma_{i1},\ldots,\gamma_{i\ell_i})$ for $i \in \{1,\ldots,k\}$, and define $(T_b')^\alpha$ to be the portion of $T_b'$ bounded by $\alpha'$ and $\gamma_{11}',\ldots,\gamma_{1\ell_1}',\ldots,\gamma_{k1}',\ldots,\gamma_{k\ell_k}'$.
Then $\Bigl((T_s')^\rho, \bigl((2T')^\alpha\coloneqq (T_b')^\alpha \to (T_s')^\rho\bigr)_\alpha\Bigr)$ is an element of the following fiber product:
\begin{align}
\prod^{K_{\#\!\incom(\rho)}}_{
{\alpha\in V_\comp^{\geq2}(T_b)\cap f^{-1}\{\rho\}}
\atop
{\incom(\alpha)=(\beta_1,\ldots,\beta_k)}
}
\hspace{-0.25in} W^\tree_{\ell_1,\ldots,\ell_k}.
\end{align}
\end{itemize}
\end{deflem}

\begin{proof}
\begin{itemize}
\item[(a)] The statement that $(2T')^\alpha$ is an element of $W_k^\tree$ is almost immediate.
Indeed, $(T_s)^\alpha = \pt$, and $(T_b)^\alpha$ is the following:

\vspace{-0.2in}
\begin{figure}[H]
\centering
\def\svgwidth{0.04\columnwidth}
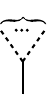
\end{figure}
\vspace{-0.15in}

\noindent Define $2T = \wt{2T}^1,\ldots,\wt{2T}^{a_0} = 2T'$, where $\wt{2T}^{a+1}$ can be obtained from $\wt{2T}^a$ by making a single move.
Then $(\wt T_s^a)^\alpha = \pt$ for every $a$, and $(\wt T_b^{a+1})^\alpha$ is either equal to $(\wt T_b^a)^\alpha$ or can be obtained from $(\wt T_b^a)^\alpha$ by performing the following modification to one of the dashed corollas in $(T_b^a)^\alpha$, for $2 \leq \ell < k$:

\vspace{-0.1in}
\begin{figure}[H]
\centering
\def\svgwidth{0.4\columnwidth}
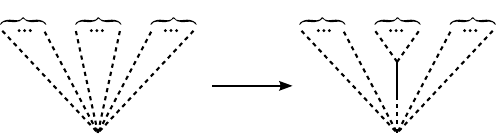
\end{figure}
\vspace{-0.1in}

\noindent $W_k^\tree$ is closed under modifications of this form, so $(2T')^\alpha$ is a stable tree-pair in $W_k^\tree$.

\item[(b)] This statement can be proven via an argument similar to the one made for (a).
\end{itemize}
\end{proof}

We are now ready to establish the \textsc{(recursive)} property.

\begin{deflem}
\label{deflem:Gamma2T}
Fix $2T \in W_\bn^\tree$.
Define a map
\begin{align}
\Gamma_{2T} \colon \prod_{
{\alpha \in V_\comp^1(T_b)}
\atop
{\incom(\alpha)=(\beta)}
} W_{\#\!\incom(\beta)}^\tree
\times
\prod_{\rho \in V_\inte(T_s)} \prod^{K_{\#\!\incom(\rho)}}_{
{\alpha\in V_\comp^{\geq2}(T_b)\cap f^{-1}\{\rho\}}
\atop
{\incom(\alpha)=(\beta_1,\ldots,\beta_{\#\!\incom(\rho)})}
}
\hspace{-0.25in} W^\tree_{\#\!\incom(\beta_1),\ldots,\#\!\incom(\beta_{\#\!\incom(\rho)})}
\hra W_\bn^\tree,
\end{align}
by sending $\Bigl((2T_\alpha)_\alpha,\bigl(T_\rho,(\wt{2T}_\alpha)_\alpha\bigr)_\rho\Bigr)$ to the element of $W_\bn^\tree$ defined like so: for $\alpha \in V^1_\comp(T_b)$ with $\incom(\alpha)\eqqcolon(\beta)$, replace the portion of $T_b$ bounded by $\alpha$ and $\incom(\beta)$ by $2T_\alpha$; for $\rho \in V_\inte(T_s)$, replace the portion of $T_s$ bounded by $\rho$ and $\incom(\rho)$ by $T_\rho$; and for $\rho \in V_\inte(T_s)$ and $\alpha \in V^{\geq2}_\comp(T_b)\cap f^{-1}\{\rho\}$ with $\incom(\alpha)\eqqcolon(\beta_1,\ldots,\beta_{\#\!\incom(\rho)})$, replace the portion of $T_b$ bounded by $\alpha$ and $\incom(\beta_1),\ldots,\incom(\beta_{\#\!\incom(\rho)})$ by $\wt{2T}_\alpha$.
Then $\Gamma_{2T}$ restricts to a poset isomorphism from its domain to $\cl(2T) \subset W_\bn^\tree$.
\end{deflem}

\begin{proof}
\noindent{\it Step 1: $\Gamma_{2T}$ is a map of posets.}

\medskip

\noindent It suffices to show that for any $\Bigl((2T_\alpha^{(2)})_\alpha,\bigl(T_\rho^{(2)},(\wt{2T}_\alpha^{(2)})_\alpha\bigr)_\rho\Bigr) < \Bigl((2T_\alpha^{(1)})_\alpha,\bigl(T_\rho^{(1)},(\wt{2T}_\alpha^{(1)})_\alpha\bigr)_\rho\Bigr)$, there exists $\Bigl((2T_\alpha^{(3)})_\alpha,\bigl(T_\rho^{(3)},(\wt{2T}_\alpha^{(3)})_\alpha\bigr)_\rho\Bigr)$ with
\begin{align}
\Bigl((2T_\alpha^{(2)})_\alpha,\bigl(T_\rho^{(2)},(\wt{2T}_\alpha^{(2)})_\alpha\bigr)_\rho\Bigr) \leq \Bigl((2T_\alpha^{(3)})_\alpha,\bigl(T_\rho^{(3)},(\wt{2T}_\alpha^{(3)})_\alpha\bigr)_\rho\Bigr) < \Bigl((2T_\alpha^{(1)})_\alpha,\bigl(T_\rho^{(1)},(\wt{2T}_\alpha^{(1)})_\alpha\bigr)_\rho\Bigr)
\end{align}
and
\begin{align}
\Gamma_{2T}\Bigl((2T_\alpha^{(3)})_\alpha,\bigl(T_\rho^{(3)},(\wt{2T}_\alpha^{(3)})_\alpha\bigr)_\rho\Bigr) < \Gamma_{2T}\Bigl((2T_\alpha^{(1)})_\alpha,\bigl(T_\rho^{(1)},(\wt{2T}_\alpha^{(1)})_\alpha\bigr)_\rho\Bigr).
\end{align}
To do so, first suppose that there exists $\alpha_0 \in V^1_\comp(T_b)$ with $2T^{(2)}_{\alpha_0} < 2T^{(1)}_{\alpha_0}$.
Define $\Bigl((2T_\alpha^{(3)})_\alpha,\bigl(T_\rho^{(3)},(\wt{2T}_\alpha^{(3)})_\alpha\bigr)_\rho\Bigr)$ to be the result of starting with $\Bigl((2T_\alpha^{(1)})_\alpha,\bigl(T_\rho^{(1)},(\wt{2T}_\alpha^{(1)})_\alpha\bigr)_\rho\Bigr)$, and replacing $2T^{(1)}_{\alpha_0}$ by $2T^{(2)}_{\alpha_0}$.
The assumption on $\alpha_0$ implies that $2T^{(2)}_{\alpha_0}$ can be obtained from $2T^{(1)}_{\alpha_0}$ by performing a sequence of type-1 moves.
Therefore $\Gamma_{2T}\Bigl((2T_\alpha^{(3)})_\alpha,\bigl(T_\rho^{(3)},(\wt{2T}_\alpha^{(3)})_\alpha\bigr)_\rho\Bigr)$ can be obtained from $\Gamma_{2T}\Bigl((2T_\alpha^{(1)})_\alpha,\bigl(T_\rho^{(1)},(\wt{2T}_\alpha^{(1)})_\alpha\bigr)_\rho\Bigr)$ by performing a sequence of type-1 moves.

If there exists no such $\alpha_0$, then we can choose $\rho_0 \in V_\inte(T_s)$ with $\bigl(T_{\rho_0}^{(2)},(\wt{2T}_\alpha^{(2)})_\alpha\bigr) < \bigl(T_{\rho_0}^{(1)},(\wt{2T}_\alpha^{(1)})_\alpha\bigr)$ and make an argument similar to the previous paragraph.

\medskip

\noindent{\it Step 2: $\Gamma_{2T}$ restricts to a poset isomorphism onto $\cl(2T)$.}

\medskip

\noindent The injectivity of $\Gamma_{2T}$ is clear.
It remains to show that the image of $\Gamma_{2T}$ is equal to $\cl(2T)$, and that the inverse is a poset map.
By Step 1, the image of $\Gamma_{2T}$ is contained in $\cl(2T)$.
Now define a putative inverse
\begin{align}
\Gamma_{2T}^{-1}\colon \cl(2T)
&\to
\prod_{
{\alpha \in V_\comp^1(T_b)}
\atop
{\incom(\alpha)=(\beta)}
} W_{\#\!\incom(\beta)}^\tree
\times
\prod_{\rho \in V_\inte(T_s)} \prod^{K_{\#\!\incom(\rho)}}_{
{\alpha\in V_\comp^{\geq2}(T_b)\cap f^{-1}\{\rho\}}
\atop
{\incom(\alpha)=(\beta_1,\ldots,\beta_{\#\!\incom(\rho)})}
}
\hspace{-0.25in} W^\tree_{\#\!\incom(\beta_1),\ldots,\#\!\incom(\beta_{\#\!\incom(\rho)})}, \\
2T' &\mapsto \Bigl((2T'_\alpha)_\alpha, \bigl(T'_\rho,(\wt{2T'}_\alpha)_\alpha\bigr)_\rho\Bigr)
\nonumber
\end{align}
like so:
\begin{itemize}
\item[(a)] For $\alpha \in V_\comp^1(T_b)$ with $\incom(\alpha)\eqqcolon(\beta)$, set $2T'_\alpha \coloneqq (2T')^\alpha \in W^\tree_{\#\!\incom(\beta)}$, where the latter stable tree-pair was defined in Def.-Lem.~\ref{deflem:subtreepairs}(a).

\item[(b)] For $\rho \in V_\inte(T_s)$ and $\alpha \in V_\comp^{\geq2}(T_b)$ with $f(\alpha)=\rho$ and $\incom(\alpha)\eqqcolon(\beta_1,\ldots,\beta_{\#\!\incom(\rho)})$, define
\begin{align}
\bigl(T_\rho',(\wt{2T'}_\alpha)_\alpha\bigr)_\rho
\coloneqq
\bigl((T_s')^\rho,((2T')^\alpha)_\alpha\bigr)_\rho,
\end{align}
where the latter expression was defined in Def.-Lem.~\ref{deflem:subtreepairs}(b).
\end{itemize}
It is simple to verify that $\Gamma_{2T}^{-1}$ is an inverse for the restriction of $\Gamma_{2T}$ to a map to $\cl(2T)$, and to verify that $\Gamma_{2T}^{-1}$ is a poset map.
\end{proof}

Now that we have recursively characterized the closed faces of $W_\bn$, we turn to our proof that $\wh{W_\bn}$ is an abstract polytope.



\begin{theorem}
\label{thm:Wn_polytope}
For any $r \geq 1$ and $\bn \in \bZ^r_{\geq0}\setminus\{\bzero,(1)\}$, $\wh{W_\bn}$ is an abstract polytope of dimension $|\bn| + r - 3$.	
\end{theorem}

\begin{proof}
We defer the proofs of {\sc(diamond)} resp.\ {\sc(strongly connected)} to Props.\ \ref{prop:Wn_diamond} resp.\ \ref{prop:Wn_conn}, so here we only need to establish {\sc(extremal)}, {\sc(flag-length)}, and the dimension formula.

The least face of $\wh{W_\bn}$ is the face $F_{-1}$ we have added to $W_\bn$ to form $\wh{W_\bn}$, while the greatest face (in $\wh{W_\bn^\br}$) is the 2-bracketing $(\sB, \stB)$ with
\begin{gather}
\sB \coloneqq \{\{1,\ldots,r\},\{1\},\ldots,\{r\}\},
\\
\stB \coloneqq \Bigl\{ \bigl(\{1,\ldots,r\},(\{1,\ldots,n_1\},\ldots,\{1,\ldots,n_r\})\bigr)\Bigr\}
\cup
\left\{\bigl(\{i\},(\{j\})\bigr) \:\left|\:
{{1\leq i\leq r}
\atop
{1\leq j\leq n_i}}
\right.\right\}.
\nonumber
\end{gather}
This establishes {\sc(extremal)}.

To prove {\sc(flag-length)} and to show that the dimension of $W_\bn$ is $|\bn|+r-3$, we must show that if $2T^0 < \cdots < 2T^\ell$ is a maximal chain in $W_\bn^\tree$, then $\ell = |\bn| + r - 3$.
By Lemmata~\ref{lem:Wn_dim} and \ref{lem:Wn_moves}, we have $0 \leq d(2T^0) < \cdots < d(2T^\ell) \leq |\bn|+r-3$.
To prove the claim, we must show that every dimension between 0 and $|\bn|+r-3$ is represented.
For any $T^i, T^{i+1}$, we must have $d(T^i) = d(T^{i+1}) - 1$: otherwise, there exists $T' \in W_\bn^\tree$ which can be obtained by performing a single move to $T^{i+1}$, and which satisfies $d(T^i) < d(T') < d(T^{i+1})$; this would contradict the maximality of the chain.
Again by maximality, we must have $2T^\ell = F_\top$.
It remains to show $d(T^0) = 0$.
Suppose for a contradiction that $d(T^0)$ is positive.
Then by Lemma~\ref{lem:Wn_dim}, either (a) there exists $\alpha \in V_\comp^1(T_b)$ with $\#\!\incom(\beta) \geq 3$ for $(\beta) \coloneqq \incom(\alpha)$; (b) there exists $\alpha \in V_\comp^{\geq2}(T_b)$ with $\sum_{\beta \in \incom(\alpha)} \#\!\incom(\beta) \geq 2$; or (c) there exists $\alpha \in (T_s)_\inte$ with $\#\!\incom(\alpha) \geq 3$.
In these three cases we may perform a move of type 1 resp.\ type 3 resp.\ type 2 to $2T^0$, which contradicts the maximality of our chain.
\end{proof}


\begin{proposition}
\label{prop:Wn_diamond}
For any $r \geq 1$ and $\bn \in \bZ^r_{\geq0}\setminus\{\bzero\}$, $\wh{W_\bn}$ satisfies {\textsc{(diamond)}}.
\end{proposition}

\begin{proof}
We must show that for every $F < G$ in $\wh{W_\bn}$ with $d(G) - d(F) = 2$, the open interval $(F,G)$ contains exactly 2 elements.
In the following steps, we prove this in the cases $F \neq F_{-1}$ and $F = F_{-1}$.

\medskip

\noindent{\it Step 1: We show that for $F < G$ in $W_\bn$ with $d(G) - d(F) = 2$, the open interval $(F,G)$ contains exactly 2 elements.}

\medskip

\noindent In this step, we work with $W_\bn^\tree$.
Fix $2T, 2T' \in W_\bn^\tree$ with $d(2T) - d(2T') = 2$.
Then $2T'$ can be obtained from $2T$ by applying two moves; we must prove that there are exactly two elements of the open interval $(2T',2T)$.
There are nine cases to consider, depending on whether each of the two moves are of type 1, 2, or 3.
This proof quickly becomes repetitive, so we only give details in the case of two type-3 moves.

Suppose that $2T'$ is indeed the result of applying two type-3 moves to $2T$.
Denote the modifications to $T_b$ are as in the upper-left and upper-right arrows of the following figure (where an arrow indicates a single move, and the adjacent number indicates the type of move):

\noindent
\begin{figure}[H]
\centering
\def\svgwidth{1.0\columnwidth}
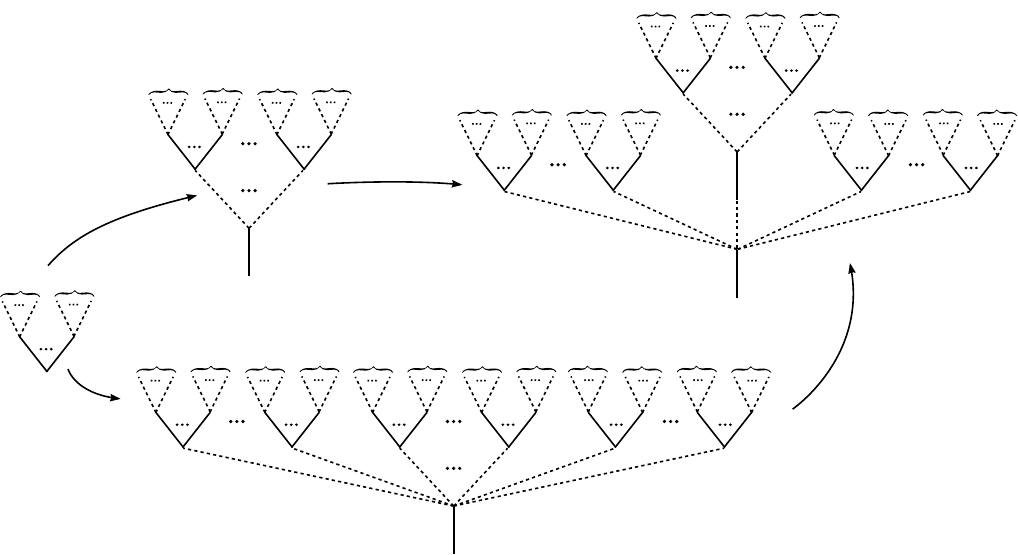
\end{figure}

\noindent (Here we have chosen a partition of $\ba$ as $\ba = \sum_{i=1}^q \bb^i$, then chosen a partition of $\bb^p$ as $\bb^p = \sum_{j=1}^r \bc^j$.)
Then there is exactly one other element of the open interval $(2T',2T)$: the one obtained from $2T$ by replacing the portion of $T_b$ on the left of the figure by the bottom configuration in the figure.
Note that this alternate path from $2T$ to $2T'$ consists not of two type-3 moves, but by a type-3 move followed by a type-1 move.

The other subcases to consider are simpler than this one.
Moreover, the other eight cases are similar to this one; rather than include the details, we show in Table~\ref{tab:diamond_ex} representative examples of the diamond property.

\begin{table}[h]
\begin{center}
\begin{tabular}{|c||c|c|c|}
\hline
&
type-1
&
type-2
&
type-3
\\
\hline
\hline
\raisebox{0.7in}{type-1}
&
$\def\svgwidth{0.2\columnwidth}
\begingroup%
  \makeatletter%
  \providecommand\color[2][]{%
    \errmessage{(Inkscape) Color is used for the text in Inkscape, but the package 'color.sty' is not loaded}%
    \renewcommand\color[2][]{}%
  }%
  \providecommand\transparent[1]{%
    \errmessage{(Inkscape) Transparency is used (non-zero) for the text in Inkscape, but the package 'transparent.sty' is not loaded}%
    \renewcommand\transparent[1]{}%
  }%
  \providecommand\rotatebox[2]{#2}%
  \ifx\svgwidth\undefined%
    \setlength{\unitlength}{63.647232bp}%
    \ifx\svgscale\undefined%
      \relax%
    \else%
      \setlength{\unitlength}{\unitlength * \real{\svgscale}}%
    \fi%
  \else%
    \setlength{\unitlength}{\svgwidth}%
  \fi%
  \global\let\svgwidth\undefined%
  \global\let\svgscale\undefined%
  \makeatother%
  \begin{picture}(1,1.00554255)%
    \put(0,0){\includegraphics[width=\unitlength]{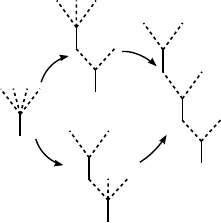}}%
    \put(0.14670349,0.71657884){\color[rgb]{0,0,0}\makebox(0,0)[lb]{\smash{$1$}}}%
    \put(0.59426858,0.79276017){\color[rgb]{0,0,0}\makebox(0,0)[lb]{\smash{$1$}}}%
    \put(0.70854134,0.23092309){\color[rgb]{0,0,0}\makebox(0,0)[lb]{\smash{$1$}}}%
    \put(0.12765816,0.22774854){\color[rgb]{0,0,0}\makebox(0,0)[lb]{\smash{$1$}}}%
  \end{picture}%
\endgroup%
$
&
$\def\svgwidth{0.35\columnwidth}
\begingroup%
  \makeatletter%
  \providecommand\color[2][]{%
    \errmessage{(Inkscape) Color is used for the text in Inkscape, but the package 'color.sty' is not loaded}%
    \renewcommand\color[2][]{}%
  }%
  \providecommand\transparent[1]{%
    \errmessage{(Inkscape) Transparency is used (non-zero) for the text in Inkscape, but the package 'transparent.sty' is not loaded}%
    \renewcommand\transparent[1]{}%
  }%
  \providecommand\rotatebox[2]{#2}%
  \ifx\svgwidth\undefined%
    \setlength{\unitlength}{124.6142832bp}%
    \ifx\svgscale\undefined%
      \relax%
    \else%
      \setlength{\unitlength}{\unitlength * \real{\svgscale}}%
    \fi%
  \else%
    \setlength{\unitlength}{\svgwidth}%
  \fi%
  \global\let\svgwidth\undefined%
  \global\let\svgscale\undefined%
  \makeatother%
  \begin{picture}(1,0.71242697)%
    \put(0,0){\includegraphics[width=\unitlength]{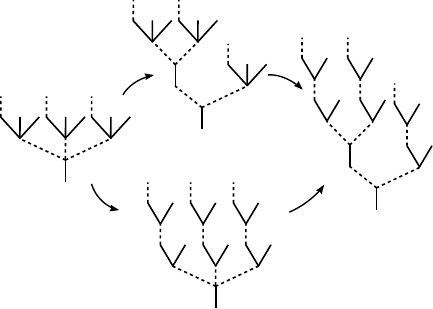}}%
    \put(0.26947433,0.52327421){\color[rgb]{0,0,0}\makebox(0,0)[lb]{\smash{$1$}}}%
    \put(0.63563722,0.53696352){\color[rgb]{0,0,0}\makebox(0,0)[lb]{\smash{$2$}}}%
    \put(0.71807676,0.22134331){\color[rgb]{0,0,0}\makebox(0,0)[lb]{\smash{$1$}}}%
    \put(0.19440235,0.21513627){\color[rgb]{0,0,0}\makebox(0,0)[lb]{\smash{$2$}}}%
  \end{picture}%
\endgroup%
$
&
$\def\svgwidth{0.35\columnwidth}
\begingroup%
  \makeatletter%
  \providecommand\color[2][]{%
    \errmessage{(Inkscape) Color is used for the text in Inkscape, but the package 'color.sty' is not loaded}%
    \renewcommand\color[2][]{}%
  }%
  \providecommand\transparent[1]{%
    \errmessage{(Inkscape) Transparency is used (non-zero) for the text in Inkscape, but the package 'transparent.sty' is not loaded}%
    \renewcommand\transparent[1]{}%
  }%
  \providecommand\rotatebox[2]{#2}%
  \ifx\svgwidth\undefined%
    \setlength{\unitlength}{113.122624bp}%
    \ifx\svgscale\undefined%
      \relax%
    \else%
      \setlength{\unitlength}{\unitlength * \real{\svgscale}}%
    \fi%
  \else%
    \setlength{\unitlength}{\svgwidth}%
  \fi%
  \global\let\svgwidth\undefined%
  \global\let\svgscale\undefined%
  \makeatother%
  \begin{picture}(1,0.7063177)%
    \put(0,0){\includegraphics[width=\unitlength]{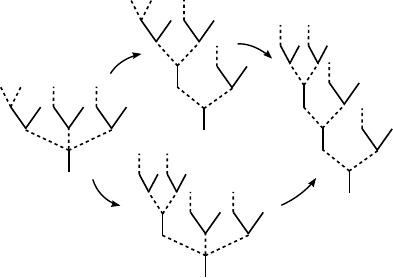}}%
    \put(0.28544713,0.56205987){\color[rgb]{0,0,0}\makebox(0,0)[lb]{\smash{1}}}%
    \put(0.63370591,0.60135067){\color[rgb]{0,0,0}\makebox(0,0)[lb]{\smash{3}}}%
    \put(0.21579529,0.16558055){\color[rgb]{0,0,0}\makebox(0,0)[lb]{\smash{3}}}%
    \put(0.7533642,0.15843693){\color[rgb]{0,0,0}\makebox(0,0)[lb]{\smash{1}}}%
  \end{picture}%
\endgroup%
$
\\
\hline
\raisebox{0.5in}{type-2}
&
$\def\svgwidth{0.2\columnwidth}
\begingroup%
  \makeatletter%
  \providecommand\color[2][]{%
    \errmessage{(Inkscape) Color is used for the text in Inkscape, but the package 'color.sty' is not loaded}%
    \renewcommand\color[2][]{}%
  }%
  \providecommand\transparent[1]{%
    \errmessage{(Inkscape) Transparency is used (non-zero) for the text in Inkscape, but the package 'transparent.sty' is not loaded}%
    \renewcommand\transparent[1]{}%
  }%
  \providecommand\rotatebox[2]{#2}%
  \ifx\svgwidth\undefined%
    \setlength{\unitlength}{59.44533584bp}%
    \ifx\svgscale\undefined%
      \relax%
    \else%
      \setlength{\unitlength}{\unitlength * \real{\svgscale}}%
    \fi%
  \else%
    \setlength{\unitlength}{\svgwidth}%
  \fi%
  \global\let\svgwidth\undefined%
  \global\let\svgscale\undefined%
  \makeatother%
  \begin{picture}(1,0.89524023)%
    \put(0,0){\includegraphics[width=\unitlength]{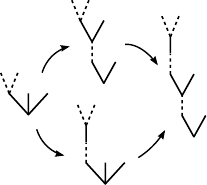}}%
    \put(0.16248998,0.63306905){\color[rgb]{0,0,0}\makebox(0,0)[lb]{\smash{$2$}}}%
    \put(0.64169124,0.71463525){\color[rgb]{0,0,0}\makebox(0,0)[lb]{\smash{$1$}}}%
    \put(0.76404136,0.11308469){\color[rgb]{0,0,0}\makebox(0,0)[lb]{\smash{$2$}}}%
    \put(0.14209843,0.10968576){\color[rgb]{0,0,0}\makebox(0,0)[lb]{\smash{$1$}}}%
  \end{picture}%
\endgroup%
$
&
$\def\svgwidth{0.2\columnwidth}
\begingroup%
  \makeatletter%
  \providecommand\color[2][]{%
    \errmessage{(Inkscape) Color is used for the text in Inkscape, but the package 'color.sty' is not loaded}%
    \renewcommand\color[2][]{}%
  }%
  \providecommand\transparent[1]{%
    \errmessage{(Inkscape) Transparency is used (non-zero) for the text in Inkscape, but the package 'transparent.sty' is not loaded}%
    \renewcommand\transparent[1]{}%
  }%
  \providecommand\rotatebox[2]{#2}%
  \ifx\svgwidth\undefined%
    \setlength{\unitlength}{57.922544bp}%
    \ifx\svgscale\undefined%
      \relax%
    \else%
      \setlength{\unitlength}{\unitlength * \real{\svgscale}}%
    \fi%
  \else%
    \setlength{\unitlength}{\svgwidth}%
  \fi%
  \global\let\svgwidth\undefined%
  \global\let\svgscale\undefined%
  \makeatother%
  \begin{picture}(1,0.85848909)%
    \put(0,0){\includegraphics[width=\unitlength]{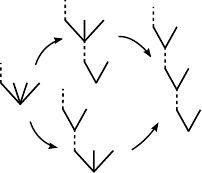}}%
    \put(0.13299599,0.62601967){\color[rgb]{0,0,0}\makebox(0,0)[lb]{\smash{$2$}}}%
    \put(0.62479551,0.70973026){\color[rgb]{0,0,0}\makebox(0,0)[lb]{\smash{$2$}}}%
    \put(0.75036224,0.09236485){\color[rgb]{0,0,0}\makebox(0,0)[lb]{\smash{$2$}}}%
    \put(0.11206835,0.08887656){\color[rgb]{0,0,0}\makebox(0,0)[lb]{\smash{$2$}}}%
  \end{picture}%
\endgroup%
$
&
$\def\svgwidth{0.2\columnwidth}
\begingroup%
  \makeatletter%
  \providecommand\color[2][]{%
    \errmessage{(Inkscape) Color is used for the text in Inkscape, but the package 'color.sty' is not loaded}%
    \renewcommand\color[2][]{}%
  }%
  \providecommand\transparent[1]{%
    \errmessage{(Inkscape) Transparency is used (non-zero) for the text in Inkscape, but the package 'transparent.sty' is not loaded}%
    \renewcommand\transparent[1]{}%
  }%
  \providecommand\rotatebox[2]{#2}%
  \ifx\svgwidth\undefined%
    \setlength{\unitlength}{56.403904bp}%
    \ifx\svgscale\undefined%
      \relax%
    \else%
      \setlength{\unitlength}{\unitlength * \real{\svgscale}}%
    \fi%
  \else%
    \setlength{\unitlength}{\svgwidth}%
  \fi%
  \global\let\svgwidth\undefined%
  \global\let\svgscale\undefined%
  \makeatother%
  \begin{picture}(1,0.9008697)%
    \put(0,0){\includegraphics[width=\unitlength]{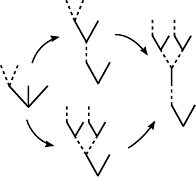}}%
    \put(0.12073703,0.67135632){\color[rgb]{0,0,0}\makebox(0,0)[lb]{\smash{$2$}}}%
    \put(0.62577795,0.75732077){\color[rgb]{0,0,0}\makebox(0,0)[lb]{\smash{$3$}}}%
    \put(0.75472548,0.12333318){\color[rgb]{0,0,0}\makebox(0,0)[lb]{\smash{$1$}}}%
    \put(0.09924592,0.11975097){\color[rgb]{0,0,0}\makebox(0,0)[lb]{\smash{$2$}}}%
  \end{picture}%
\endgroup%
$
\\
\hline
\raisebox{0.5in}{type-3}
&
$\def\svgwidth{0.2\columnwidth}
\begingroup%
  \makeatletter%
  \providecommand\color[2][]{%
    \errmessage{(Inkscape) Color is used for the text in Inkscape, but the package 'color.sty' is not loaded}%
    \renewcommand\color[2][]{}%
  }%
  \providecommand\transparent[1]{%
    \errmessage{(Inkscape) Transparency is used (non-zero) for the text in Inkscape, but the package 'transparent.sty' is not loaded}%
    \renewcommand\transparent[1]{}%
  }%
  \providecommand\rotatebox[2]{#2}%
  \ifx\svgwidth\undefined%
    \setlength{\unitlength}{67.82891176bp}%
    \ifx\svgscale\undefined%
      \relax%
    \else%
      \setlength{\unitlength}{\unitlength * \real{\svgscale}}%
    \fi%
  \else%
    \setlength{\unitlength}{\svgwidth}%
  \fi%
  \global\let\svgwidth\undefined%
  \global\let\svgscale\undefined%
  \makeatother%
  \begin{picture}(1,0.81409603)%
    \put(0,0){\includegraphics[width=\unitlength]{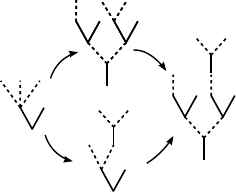}}%
    \put(0.17414208,0.5558869){\color[rgb]{0,0,0}\makebox(0,0)[lb]{\smash{$3$}}}%
    \put(0.5941146,0.62737162){\color[rgb]{0,0,0}\makebox(0,0)[lb]{\smash{$1$}}}%
    \put(0.7013424,0.10017202){\color[rgb]{0,0,0}\makebox(0,0)[lb]{\smash{$3$}}}%
    \put(0.1562709,0.09719319){\color[rgb]{0,0,0}\makebox(0,0)[lb]{\smash{$1$}}}%
  \end{picture}%
\endgroup%
$
&
$\def\svgwidth{0.2\columnwidth}
\begingroup%
  \makeatletter%
  \providecommand\color[2][]{%
    \errmessage{(Inkscape) Color is used for the text in Inkscape, but the package 'color.sty' is not loaded}%
    \renewcommand\color[2][]{}%
  }%
  \providecommand\transparent[1]{%
    \errmessage{(Inkscape) Transparency is used (non-zero) for the text in Inkscape, but the package 'transparent.sty' is not loaded}%
    \renewcommand\transparent[1]{}%
  }%
  \providecommand\rotatebox[2]{#2}%
  \ifx\svgwidth\undefined%
    \setlength{\unitlength}{66.309312bp}%
    \ifx\svgscale\undefined%
      \relax%
    \else%
      \setlength{\unitlength}{\unitlength * \real{\svgscale}}%
    \fi%
  \else%
    \setlength{\unitlength}{\svgwidth}%
  \fi%
  \global\let\svgwidth\undefined%
  \global\let\svgscale\undefined%
  \makeatother%
  \begin{picture}(1,0.81285717)%
    \put(0,0){\includegraphics[width=\unitlength]{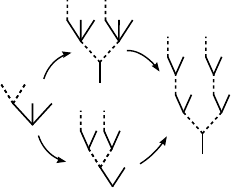}}%
    \put(0.14805066,0.54762778){\color[rgb]{0,0,0}\makebox(0,0)[lb]{\smash{$3$}}}%
    \put(0.57764763,0.6207507){\color[rgb]{0,0,0}\makebox(0,0)[lb]{\smash{$2$}}}%
    \put(0.68733274,0.08146935){\color[rgb]{0,0,0}\makebox(0,0)[lb]{\smash{$3$}}}%
    \put(0.12976993,0.07842226){\color[rgb]{0,0,0}\makebox(0,0)[lb]{\smash{$2$}}}%
  \end{picture}%
\endgroup%
$
&
$\def\svgwidth{0.2\columnwidth}
\begingroup%
  \makeatletter%
  \providecommand\color[2][]{%
    \errmessage{(Inkscape) Color is used for the text in Inkscape, but the package 'color.sty' is not loaded}%
    \renewcommand\color[2][]{}%
  }%
  \providecommand\transparent[1]{%
    \errmessage{(Inkscape) Transparency is used (non-zero) for the text in Inkscape, but the package 'transparent.sty' is not loaded}%
    \renewcommand\transparent[1]{}%
  }%
  \providecommand\rotatebox[2]{#2}%
  \ifx\svgwidth\undefined%
    \setlength{\unitlength}{70.310768bp}%
    \ifx\svgscale\undefined%
      \relax%
    \else%
      \setlength{\unitlength}{\unitlength * \real{\svgscale}}%
    \fi%
  \else%
    \setlength{\unitlength}{\svgwidth}%
  \fi%
  \global\let\svgwidth\undefined%
  \global\let\svgscale\undefined%
  \makeatother%
  \begin{picture}(1,0.80764687)%
    \put(0,0){\includegraphics[width=\unitlength]{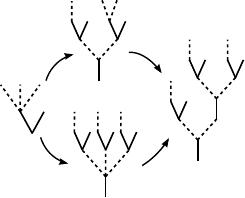}}%
    \put(0.14631721,0.55205575){\color[rgb]{0,0,0}\makebox(0,0)[lb]{\smash{$3$}}}%
    \put(0.55146538,0.62101717){\color[rgb]{0,0,0}\makebox(0,0)[lb]{\smash{$3$}}}%
    \put(0.65490821,0.11242686){\color[rgb]{0,0,0}\makebox(0,0)[lb]{\smash{$1$}}}%
    \put(0.12907685,0.10955318){\color[rgb]{0,0,0}\makebox(0,0)[lb]{\smash{$3$}}}%
  \end{picture}%
\endgroup%
$
\\
\hline
\end{tabular}
\end{center}
\caption{In this step we show that for any $2T, 2T' \in W^\tree_\bn$ with $d(2T)-d(2T')=2$, the open interval $(2T',2T)$ contains two elements.
Here we illustrate this fact in nine cases: $2T'$ can be obtained from $2T$ by applying two moves, and these moves can be of type 1, 2, or 3.
In each case, the four configurations are the bubble trees $T_b$ of a stable tree-pair; the seam trees can be inferred from the bubble tree.
On the left is $2T$, on the right is $2T'$, and the remaining two stable tree-pairs are the two elements of $(2T',2T)$.
The arrows indicate moves, and their labels are the types.}
\label{tab:diamond_ex}
\end{table}

\medskip

\noindent{\it Step 2: We show that for $G \in W_\bn$ with $d(G) = 1$, $(F_{-1},G)$ contains exactly 2 elements.}

\medskip

\noindent In this step, we again work with $W_\bn^\tree$.
Fix $2T \in W_\bn^\tree$ with $d(2T) = 1$.
It follows from Lemma~\ref{lem:Wn_dim} that the vertices in $2T$ satisfy exactly one of three valency conditions, which we treat in cases below.
\begin{itemize}
\item[] {\it Every $\rho \in (T_s)_\inte$ has $\#\!\incom(\rho) = 2$.
There is a single $\alpha \in V_\comp^1(T_b)$ with $\#\!\incom(\beta) = 3$, where $\beta$ is the incoming neighbor of $\alpha$; every $\gamma \in V_\comp^1(T_b) \setminus \{\alpha\}$ with $\incom(\gamma) \eqqcolon (\delta)$ has $\#\!\incom(\delta) = 2$; and every $\gamma \in V_\comp^{\geq2}(T_b)$ with $\incom(\gamma) \eqqcolon (\delta_1,\delta_2)$ has $\#\!\incom(\delta_1)+\#\!\incom(\delta_2) = 1$.} In this case, the only moves that can be performed on $2T$ are type-1 moves based at $\alpha$.
If we denote $\incom(\beta) \eqqcolon (\gamma_1,\gamma_2,\gamma_3)$, then the type-1 moves at $\alpha$ correspond to proper consecutive subsets of $(\gamma_1,\gamma_2,\gamma_3)$ of length at least 2; $(\gamma_1,\gamma_2)$ and $(\gamma_2,\gamma_3)$ are the only such subsets.

\medskip

\item[] {\it Every $\rho \in (T_s)_\inte$ has $\#\!\incom(\rho) = 2$.
There is a single $\alpha \in V_\comp^{\geq2}(T_b)$ with $\#\!\incom(\beta_1)+\#\!\incom(\beta_2) = 2$, where $\beta_1,\beta_2$ are the incoming neighbors of $\alpha$; every $\gamma \in V_\comp^1(T_b)$ with $\incom(\gamma) \eqqcolon (\delta)$ has $\#\!\incom(\delta) = 2$; and every $\gamma \in V_\comp^{\geq2}(T_b) \setminus \{\alpha\}$ with $\incom(\gamma) \eqqcolon (\delta_1,\delta_2)$ has $\#\!\incom(\delta_1)+\#\!\incom(\delta_2) = 1$.} There are two subcases: either (a) $(\#\!\incom(\beta_1),\#\!\incom(\beta_2))=(1,1)$ or (b) $(\#\!\incom(\beta_1),\#\!\incom(\beta_2)) \in \{(2,0),(0,2)\}$.
If (a) holds, the only moves that can be performed on $2T$ are type-3 moves based at $\alpha$.
In the notation of the definition of type-3 moves, $\ba=(1,1)$, and the type-3 moves at $\alpha$ correspond to choices of $\bb^1, \ldots, \bb^q \in \bZ_{\geq0}^2\setminus\{\bzero\}$ with $\sum_j \bb^j = \ba$.
There are two such choices: $\bb^1 = (1,0)$ and $\bb^2 = (0,1)$, or $\bb^1 = (0,1)$ and $\bb^2 = (1,0)$.
On the other hand, if (b) holds, we can either perform a type-1 move or a type-3 move at $\alpha$, and there is only possible move of each type.

\medskip

\item[] {\it There is a single $\rho \in (T_s)_\inte$ with $\#\!\incom(\rho)=3$.
Every $\alpha \in V_\comp^1(T_b)$ with $\incom(\alpha) \eqqcolon (\beta)$ has $\#\!\incom(\beta) = 2$; every $\alpha \in V_\comp^{\geq2}(T_b)$ with $\incom(\alpha) \eqqcolon (\beta_1,\beta_2)$ has $\#\!\incom(\beta_1)+\#\!\incom(\beta_2) = 1$; and every $\sigma \in (T_s)_\inte \setminus \{\rho\}$ has $\#\!\incom(\sigma) = 2$.} In this case, the only moves that can be performed on $2T$ are type-2 moves based at $\rho$.
Denote $\incom(\rho) \eqqcolon (\sigma_1,\sigma_2,\sigma_3)$.
The type-2 moves based at $\rho$ correspond to (1) a choice of a proper consecutive subset of $\incom(\rho)$ of length at least 2 (of which there are two), and (2) for every $\alpha \in V_\comp^{\geq2}(T_b)$ with $f(\alpha)=\rho$, a choice of $q\geq0$ and $\bb^1,\ldots,\bb^q \in \bZ_{\geq0}^\ell\setminus\{\bzero\}$ with $\sum_j \bb^j = \ba$, where $\ba \in \bZ_{\geq0}^\ell\setminus\{\bzero\}$ is defined by setting $\incom(\alpha) \eqqcolon (\beta_1,\beta_2,\beta_3)$ and $a_i \coloneqq \#\!\incom(\beta_{p+i})$.
By assumption, for every such $\alpha$, we have $|\ba|=1$.
Therefore there is exactly one choice of the decomposition $\ba = \sum_j \bb^j$, so there are two type-2 moves based at $\rho$.
\end{itemize}
\end{proof}

\begin{proposition}
\label{prop:Wn_conn}
For any $r \geq 1$ and $\bn \in \bZ^r_{\geq0}\setminus\{\bzero\}$, $\wh{W_\bn}$ is strongly connected.
\end{proposition}

\begin{proof}
We must show that for every $F < G$ in $\wh{W_\bn}$ with $d(G) - d(F) \geq 3$, $[F,G]$ is connected, i.e.\ any two elements in the open interval $(F,G)$ can be connected by a path contained in $(F,G)$.
In the following steps, we prove this in the cases $F \neq F_{-1}$ and $F = F_{-1}$.

\medskip

\noindent{\it Step 1: For $s\geq 1$ and $\bm^1,\ldots,\bm^\ell \in \bZ^s_{\geq0}\setminus\{\bzero\}$, define a dimension function on the completed fiber product $P\coloneqq \{F_{-1}\} \cup \prod_{1\leq i\leq \ell}^{K_s^\tree} W_{\bm^i}^\tree$ like so:
\begin{align}
\label{eq:d_on_fiber_product}
d\bigl(T,(2T^{(i)})_i\bigr) \coloneqq d(T)+\sum_{1\leq i\leq\ell} \bigl(d(2T^{(i)})-d(T)\bigr), \qquad d(F_{-1}) \coloneqq -1.
\end{align}
Denote the maximal element of $P$ by $F_\top$.
For any $G \in P$ with $d(F_\top)-d(G) \geq 3$, the interval $[G,F_\top]$ is connected.
}

\medskip

\noindent We divide this step into two cases, depending on whether or not $G$ is the minimal element $F_{-1}$.

First, suppose $G = F_{-1}$.
The condition $d(F_\top) - d(G) \geq 3$ translates into the condition $s-2+\sum_{1\leq i\leq \ell} (|\bm^i|-1) \geq 2$, which implies that at least one of these conditions holds:
\begin{itemize}
\item[(a)] $s \geq 4$.

\item[(b)] $s \geq 3$ and there exists $i$ with $|\bm^i| \geq 2$.

\item[(c)] There exist $i \neq j$ with $|\bm^i| \geq 2$ and $|\bm^j| \geq 2$.

\item[(d)] There exists $i$ with $|\bm^i| \geq 3$.
\end{itemize}
In case (a), it suffices to show that for any face $H \in (F_{-1},F_\top)$ with $d(H) = d(F^\top)-1$, there is a path in $(F_{-1},F_\top)$ from $H$ to the following element of $P$:

\noindent
\begin{figure}[H]
\centering
\def\svgwidth{0.65\columnwidth}
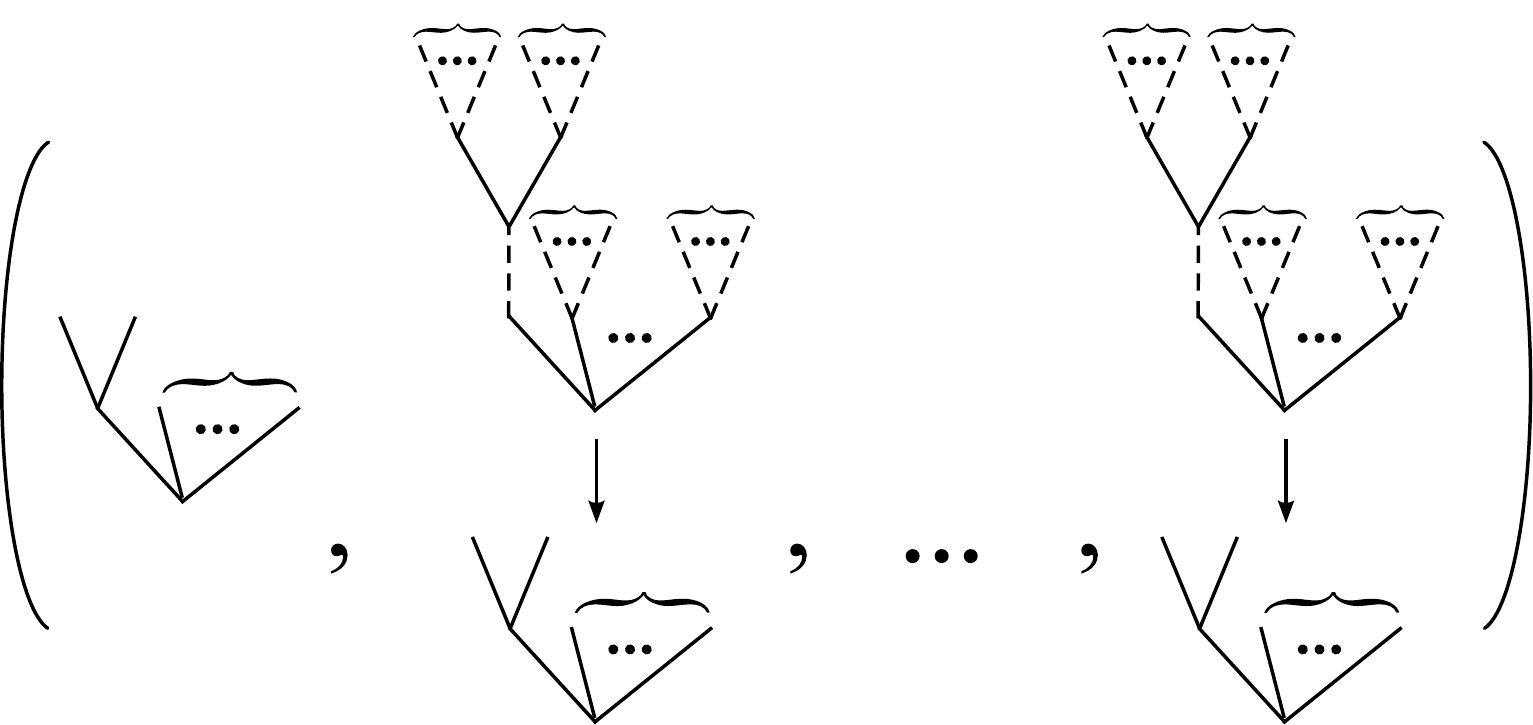
\end{figure}

\noindent This can be shown by an argument similar to the one made in Prop.~\ref{prop:Kr_polytope} to prove the {\sc(strongly connected)} property for the associahedra.
The same is true for cases (b--d).

Second, we must show that $[G,F_\top]$ is connected for $G \neq F_{-1}$.
We do so by using the 2-bracketing model for 2-associahedra.
Write $G = \bigl(\wh\sB,(\wh{\stB}_i)_{i=1}^\ell\bigr)$.
Throughout this proof we often abbreviate 2-bracketings by their collection of 2-bracketings and omit the underlying 1-bracketing, as this 1-bracketing will be evident.
Fix distinct $F^{(j)}\coloneqq\bigl(\sB^{(j)},(\stB^{(j)}_i)_{i=1}^\ell\bigr) \in (G,F_\top) \subset P$ for $j \in \{1,2\}$; we must show that there is a path between these elements within $(G,F_\top)$.
Without loss of generality, we may assume $d(F^{(j)})=d(F_\top)-1$ for $j \in \{1,2\}$.
It follows that each $F^{(j)}$ can be obtained in exactly one of the following ways, where we write $F_\top \eqqcolon \bigl(\sB,(\stB_i)_i\bigr)$:
\begin{itemize}
\item[(1)] Perform a type-1 move on a single $\stB_i$.

\item[(2)] Perform a single move $\sB \to \sB'$, then perform a type-2 move $\stB_i \to \stB'_i$ for every $1 \leq i \leq \ell$, such that $\pi(\stB'_i) = \sB'$.

\item[(3)] Perform a type-3 move on a single $\stB_i$.
\end{itemize}
Define $\wt F\coloneqq \bigl(\sB^{(1)}\cup\sB^{(2)},(\stB^{(1)}_i\cup\stB^{(2)}_i)_i\bigr)$.
Then $\wt F$ is again an element of the fiber product $P$, where for $B_0 \in \sB^{(1)} \cup \sB^{(2)}$ and $1 \leq i \leq \ell$ the partial order on $(\stB^{(1)}_i\cup\stB^{(2)}_i)_{B_0}$ is induced by the partial order on the larger collection $(\wh{\stB}_i)_{B_0}$.
If (1) holds for $F^{(j)}$ for both $j=1$ and $j=2$, then $d(\wt F) = d(F_\top)-2$; we can then connect $F^{(1)}$ and $F^{(2)}$ by the path $F^{(1)}\to \wt F\to F^{(2)}$.
The same is true in all of the other cases, except in the cases that (2) holds for $j \in \{1,2\}$, or that (3) holds for $j\in\{1,2\}$: here, $d(\wt F)$ could be less than $d(F_\top)$.
The constructions in these two situations are similar, so we assume that (3) holds for $j \in \{1,2\}$.
Denote by $i_1$ resp.\ $i_2$ the indices with the property that $F^{(j)}$ is obtained from $F_\top$ by performing a type-3 move on $\stB_{i_j}$.
If $i_1\neq i_2$, then $d(\wt F) = d(F_\top)-2$ and we can again use the path $F^{(1)} \to \wt F \to F^{(2)}$.
Suppose, on the other hand, that $i_1=i_2\eqqcolon i_0$.
If $d(\wh{\stB}_{i_0})=d(\stB_{i_0})-2$, we can again use the construction described above.
Otherwise, it suffices to show that there is a path from $\stB^{(1)}_{i_0}$ to $\stB^{(2)}_{i_0}$ within
$(\wh{\stB}_{i_0},F_\top^{W_{\bm^{i_0}}})$.
Towards this, we express $\stB^{(1)}_{i_0}$ and $\stB^{(2)}_{i_0}$ like so:
\begin{align}
\stB^{(j)}_{i_0} &= \Bigl\{\bigl(\{1,\ldots,s\},(\{1,\ldots,m^{i_0}_1\},\ldots,\{1,\ldots,m^{i_0}_s\})\bigr)\Bigr\}
\cup
\left\{\bigl(\{k\},(\{\ell\})\bigr) \:\left|\: {{1\leq k\leq s}\atop{1\leq \ell\leq m^{i_0}_k}}\right.\right\}
\\
&\hspace{2.5in} \cup \left\{\bigl(\{1,\ldots,s\},(A^{(j)}_{t,1},\ldots,A^{(j)}_{t,s})\bigr) \:\Big|\: 1 \leq t \leq q^{(j)}\right\}
\nonumber\\
&\eqqcolon \stB_\top \cup \left\{\bigl(\{1,\ldots,s\},(A^{(j)}_{t,1},\ldots,A^{(j)}_{t,s})\bigr) \:\Big|\: 1 \leq t \leq q^{(j)}\right\}.
\nonumber
\end{align}
To define our path from $\stB^{(1)}_{i_0}$ to $\stB^{(2)}_{i_0}$, we begin by examining $(A^{(1)}_{1,1},\ldots,A^{(1)}_{1,s})$ and $(A^{(2)}_{1,1},\ldots,A^{(2)}_{1,s})$.
If these sequences of sets are equal, we do nothing.
If they are not equal, then because $\stB^{(1)}_{i_0}$, $\stB^{(2)}_{i_0}$ are bounded from below by $\wh\stB_{i_0}$, there must exist $q' \geq 2$ such that one of the following equalities holds:
\begin{align}
\label{eq:conn_containments}
(A^{(1)}_{1,1},\ldots,A^{(1)}_{1,s}) &= (A^{(2)}_{1,1}\cup\cdots\cup A^{(2)}_{q',1},\ldots,A^{(2)}_{1,s}\cup\cdots\cup A^{(2)}_{q',s}), \\
(A^{(2)}_{1,1},\ldots,A^{(2)}_{1,s}) &= (A^{(1)}_{1,1}\cup\cdots\cup A^{(1)}_{q',1},\ldots,A^{(1)}_{1,s}\cup\cdots\cup A^{(1)}_{q',s}).
\nonumber
\end{align}
Suppose that the first equality holds.
Then we define the first two steps in our path like so:
\begin{align}
\Bigl(&\stB^{(1)}_{i_0} = \stB_\top \cup \left\{\bigl(\{1,\ldots,s\},(A^{(1)}_{t,1},\ldots,A^{(1)}_{t,s})\bigr) \:\Big|\: 1 \leq t \leq q^{(1)}\right\}, \\
& \stB_\top \cup \left\{\bigl(\{1,\ldots,s\},(A^{(1)}_{t,1},\ldots,A^{(1)}_{t,s})\bigr) \:\Big|\: 1 \leq t \leq q^{(1)}\right\} \cup \left\{\bigl(\{1,\ldots,s\},(A^{(2)}_{t,1},\ldots,A^{(2)}_{t,s})\bigr) \:\Big|\: 1 \leq t \leq q'\right\},
\nonumber\\
& \stB_\top \cup \left\{\bigl(\{1,\ldots,s\},(A^{(1)}_{t,1},\ldots,A^{(1)}_{t,s})\bigr) \:\Big|\: 2 \leq t \leq q^{(1)}\right\} \cup \left\{\bigl(\{1,\ldots,s\},(A^{(2)}_{t,1},\ldots,A^{(2)}_{t,s})\bigr) \:\Big|\: 1 \leq t \leq q'\right\}\Bigr).
\nonumber
\end{align}
If, on the other hand, the second equation in \eqref{eq:conn_containments} holds, we define the first two steps in our path like so:
\begin{align}
\Bigl(&\stB^{(1)}_{i_0} = \stB_\top \cup \left\{\bigl(\{1,\ldots,s\},(A^{(1)}_{t,1},\ldots,A^{(1)}_{t,s})\bigr) \:\Big|\: 1 \leq t \leq q^{(1)}\right\}, \\
& \stB_\top \cup \left\{\bigl(\{1,\ldots,s\},(A^{(1)}_{t,1},\ldots,A^{(1)}_{t,s})\bigr) \:\Big|\: 1 \leq t \leq q^{(1)}\right\} \cup \left\{\bigl(\{1,\ldots,s\},(A^{(2)}_{1,1},\ldots,A^{(2)}_{1,s})\bigr)\right\},
\nonumber\\
& \stB_\top \cup \left\{\bigl(\{1,\ldots,s\},(A^{(1)}_{t,1},\ldots,A^{(1)}_{t,s})\bigr) \:\Big|\: q'+1 \leq t \leq q^{(1)}\right\} \cup \left\{\bigl(\{1,\ldots,s\},(A^{(2)}_{1,1},\ldots,A^{(2)}_{1,s})\bigr)\right\}\Bigr).
\nonumber
\end{align}
By proceeding in this fashion, we can construct a path from $\stB^{(1)}$ to $\stB^{(2)}$ whose elements are either of codimension 1 or 2 in $W_{\bm^{i_0}}$
and which are bounded below by $G$.

\medskip

\noindent{\it Step 2: For $G \in W_\bn$ with $d(G) \geq 2$, $[F_{-1},G]$ is connected.}

\medskip

\noindent Fix $2T \in W_\bn^\tree$ with $d(2T) \geq 2$.
By Def.-Lem.~\ref{deflem:Gamma2T}, we have the following formula for $[F_{-1},2T]$:
\begin{align}
\label{eq:Wn_conn_decomp}
[F_{-1},2T]
\simeq
\{F_{-1}\}
\cup
\prod_{
{\alpha \in V_\comp^1(T_b)}
\atop
{\incom(\alpha)=(\beta)}
} W_{\#\!\incom(\beta)}^\tree
\times
\prod_{\rho \in V_\inte(T_s)} \prod^{K^\tree_{\#\!\incom(\rho)}}_{
{\alpha\in V_\comp^{\geq2}(T_b)\cap f^{-1}\{\rho\}}
\atop
{\incom(\alpha)=(\beta_1,\ldots,\beta_{\#\!\incom(\rho)})}
}
\hspace{-0.25in} W^\tree_{\#\!\incom(\beta_1),\ldots,\#\!\incom(\beta_{\#\!\incom(\rho)})}.
\end{align}
A calculation using Lemma~\ref{lem:Wn_dim} yields the following equality:
\begin{align}
\label{eq:fiber_decomp_of_d}
d(G)
=
\sum_{{\alpha \in V_\comp^1(T_b)}
\atop
{\incom(\alpha) = (\beta)}} \dim(W_{\#\!\incom(\beta)}^\tree)
+
\sum_{\rho \in V_\inte(T_s)}
\dim\left(\prod^{K^\tree_{\#\!\incom(\rho)}}_{
{\alpha\in V_\comp^{\geq2}(T_b)\cap f^{-1}\{\rho\}}
\atop
{\incom(\alpha)=(\beta_1,\ldots,\beta_{\#\!\incom(\rho)})}
}
\hspace{-0.25in} W^\tree_{\#\!\incom(\beta_1),\ldots,\#\!\incom(\beta_{\#\!\incom(\rho)})}\right)
\end{align}
(We have not shown that the fiber products are abstract polytopes.
The dimension of this fiber product should be interpreted as the dimension of the top face, using the dimension function $d$ defined in \eqref{eq:d_on_fiber_product}.)
The inequality $d(G)\geq 2$ implies that either (a) one of the posets in \eqref{eq:fiber_decomp_of_d} has dimension at least 2 or (b) at least two of the posets in \eqref{eq:fiber_decomp_of_d} have positive dimension.
If (b) holds, $[F_{-1},G]$ is clearly connected.
Next, suppose (a) holds.
If $\dim(W_{\#\!\incom(\beta)}^\tree) \geq 2$ for some $\alpha \in V_\comp^1(T_b), \incom(\alpha) = (\beta)$, then the connectedness of $[F_{-1},G]$ follows from the isomorphism $W_{\#\!\incom(\beta)}^\tree \simeq K_{\#\!\incom(\beta)}^\tree$ proven in Lemma~\ref{lem:WnKn} and the strong connectedness of the associahedra proven in Prop.~\ref{prop:Kr_polytope}.
If one of the fiber products in \eqref{eq:fiber_decomp_of_d} has dimension at least 2, then the connectedness of $[F_{-1},G]$ follows from Step 1.

\medskip

\noindent{\it Step 3: For $F < G$ in $W_\bn$ with $d(G) - d(F) \geq 3$, $[F,G]$ is connected.}

\medskip

\noindent The argument in Step 2 applies equally well to this case.
\end{proof}



\appendix

\section{2- and 3-dimensional 2-associahedra}
\label{app:ex}

In this appendix, we work out all\footnote{We make use of the identity $W_{n_1,n_2\ldots,n_r} \simeq W_{n_r,n_{r-1},\ldots,n_1}$.
Also, we do not include 2-associahedra of the form $W_{n_1}$ or
$W_{0,\ldots,0,1,0,\ldots,0}$, since these can be identified with associahedra via Lemma~\ref{lem:WnKn} and the following remark.} the 2- and 3-dimensional 2-associahedra $W_\bn$.
First, we record the face vectors of these examples and note whether or not $\wh{W_\bn}$ is simple.
(An abstract $d$-polytope is \emph{simple} if every 0-face is adjacent to exactly $d$ 1-faces.
The author believes that for every $r \geq 2$ and $\bn \in \bZ^r_{\geq0}$ with $|\bn| \geq 4$, $\wh{W_\bn}$ is not simple.)
Second, we give ``net'' representations of these examples, where each vertex corresponds to a 0-face of $W_\bn$, each edge to a 1-face, and each polygonal face to a 2-face.
In the case of the 3-dimensional examples, we do not represent the 3-dimensional face.
Each numbered edge should be identified with the correspondingly numbered edge.
Some faces are labelled by the 2-bracketings to which they correspond.

\bigskip

\begin{center}
\begin{tabular}{|c|c|c||c|c|c|}
\hline
\multicolumn{3}{|c||}{2-dimensional 2-associahedra} & \multicolumn{3}{|c|}{3-dimensional 2-associahedra} \\
\hline
$\bn$ & face vector & simple? & $\bn$ & face vector & simple? \\
\hline
30 & (6,6,1) & yes & 40 & (21,32,13,1) & no \\
\hline
21 & (8,8,1) & yes & 31 & (36,56,22,1) & no \\
\hline
200 & (5,5,1) & yes & 22 & (44,69,27,1) & no \\
\hline
110 & (6,6,1) & yes & 300 & (18,27,11,1) & yes \\
\hline
101 & (4,4,1) & yes & 210 & (30,45,17,1) & yes \\
\hline
020 & (6,6,1) & yes & 201 & (18,27,11,1) & yes \\
\hline
\multicolumn{3}{|c||}{} & 120 & (32,48,18,1) & yes \\
\hline
\multicolumn{3}{|c||}{} & 030 & (24,36,14,1) & yes \\
\hline
\multicolumn{3}{|c||}{} & 2000 & (14,21,9,1) & yes \\
\hline
\multicolumn{3}{|c||}{} & 1100 & (18,27,11,1) & yes \\
\hline
\multicolumn{3}{|c||}{} & 1010 & (14,21,9,1) & yes \\
\hline
\multicolumn{3}{|c||}{} & 1001 & (10,15,7,1) & yes \\
\hline
\multicolumn{3}{|c||}{} & 0200 & (18,27,11,1) & yes \\
\hline
\multicolumn{3}{|c||}{} & 0110 & (22,33,13,1) & yes \\
\hline
\end{tabular}
\end{center}

\newpage

\begin{figure}[H]
\centering
\includegraphics[width=0.99\columnwidth]{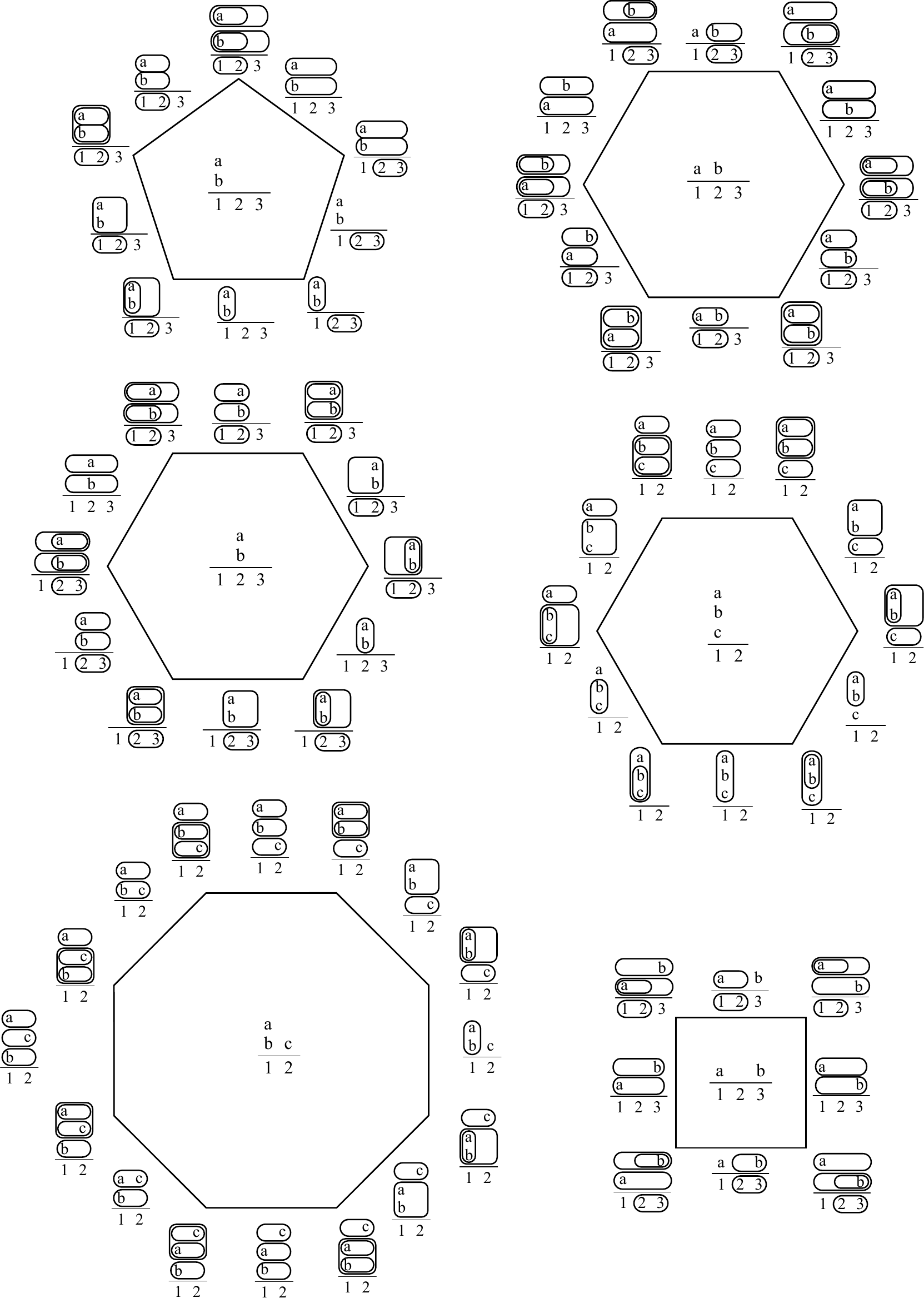}
\end{figure}

\newpage

\begin{figure}[H]
\centering
\includegraphics[width=\columnwidth]{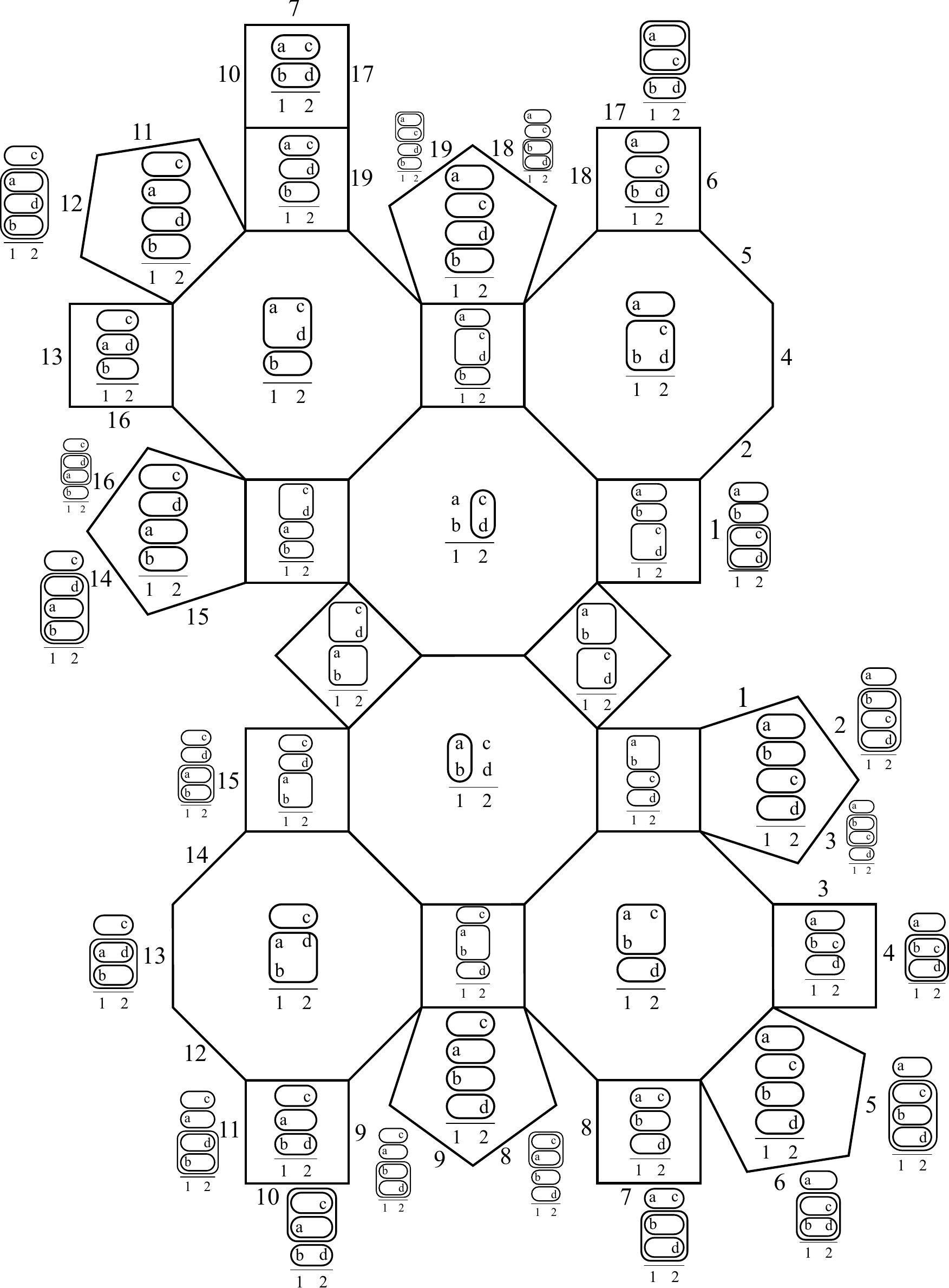}
\end{figure}

\newpage

\begin{figure}[H]
\centering
\includegraphics[width=\columnwidth]{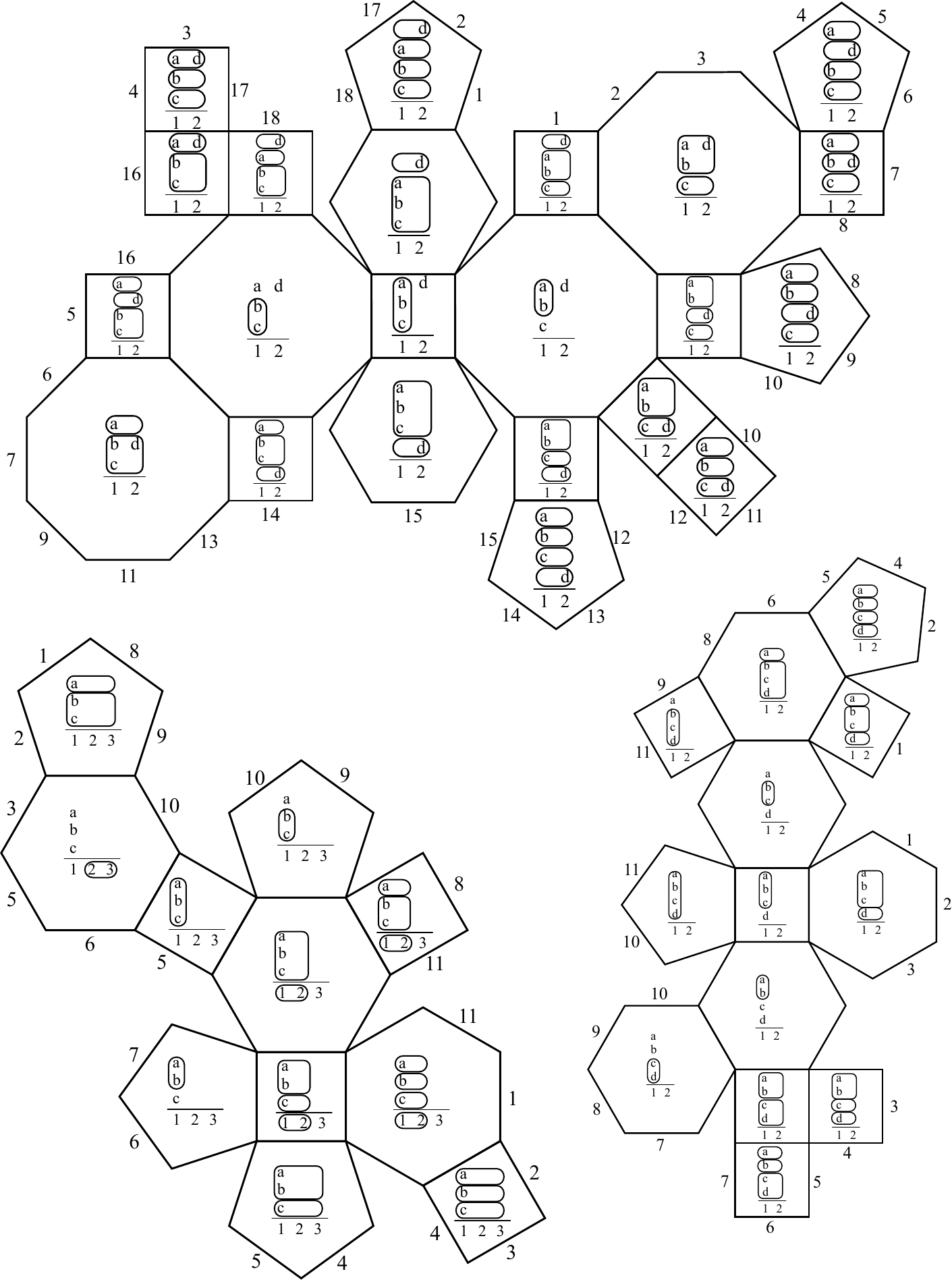}
\end{figure}

\newpage

\begin{figure}[H]
\centering
\includegraphics[width=1.0\columnwidth]{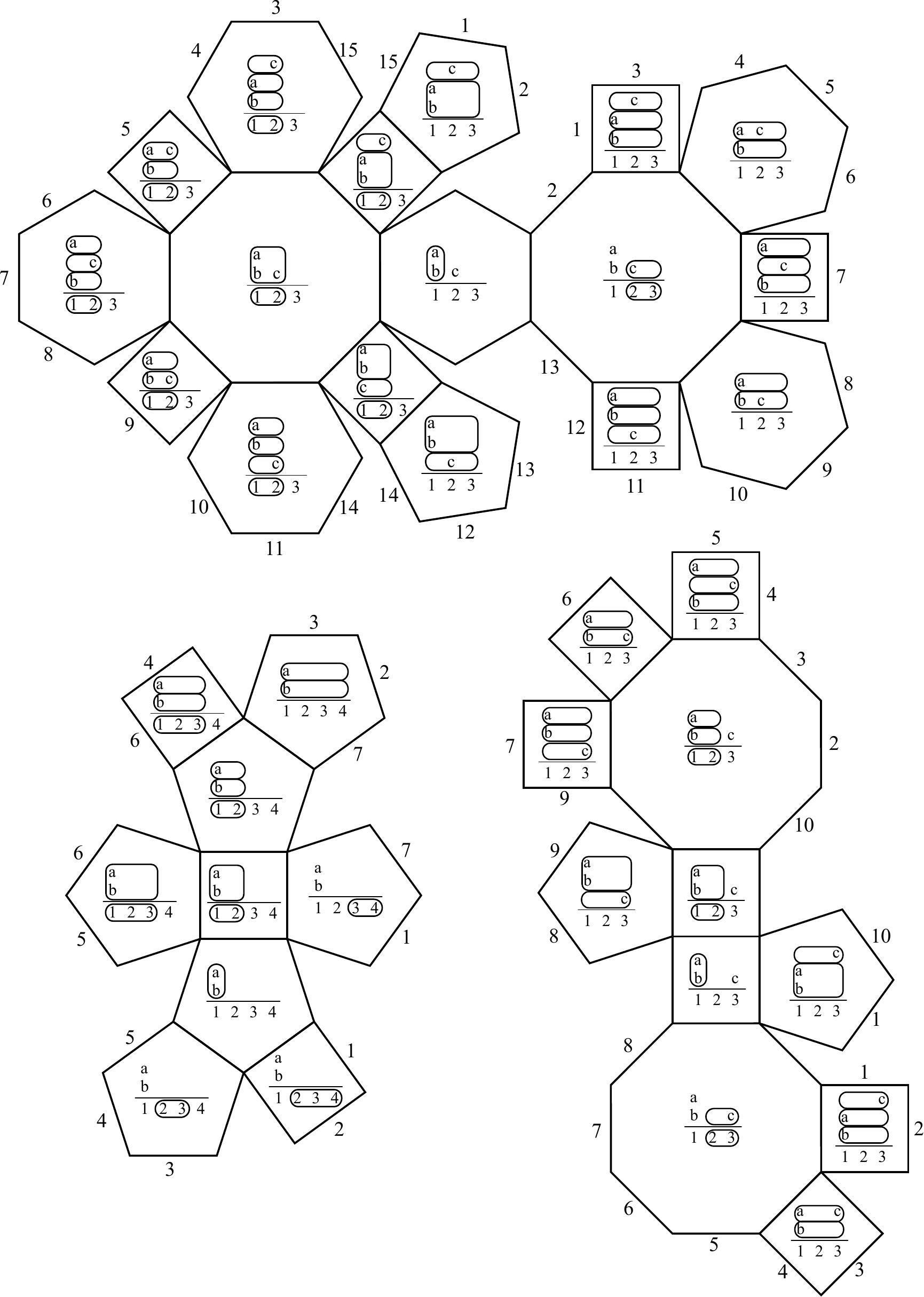}
\end{figure}

\newpage

\begin{figure}[H]
\centering
\includegraphics[width=0.98\columnwidth]{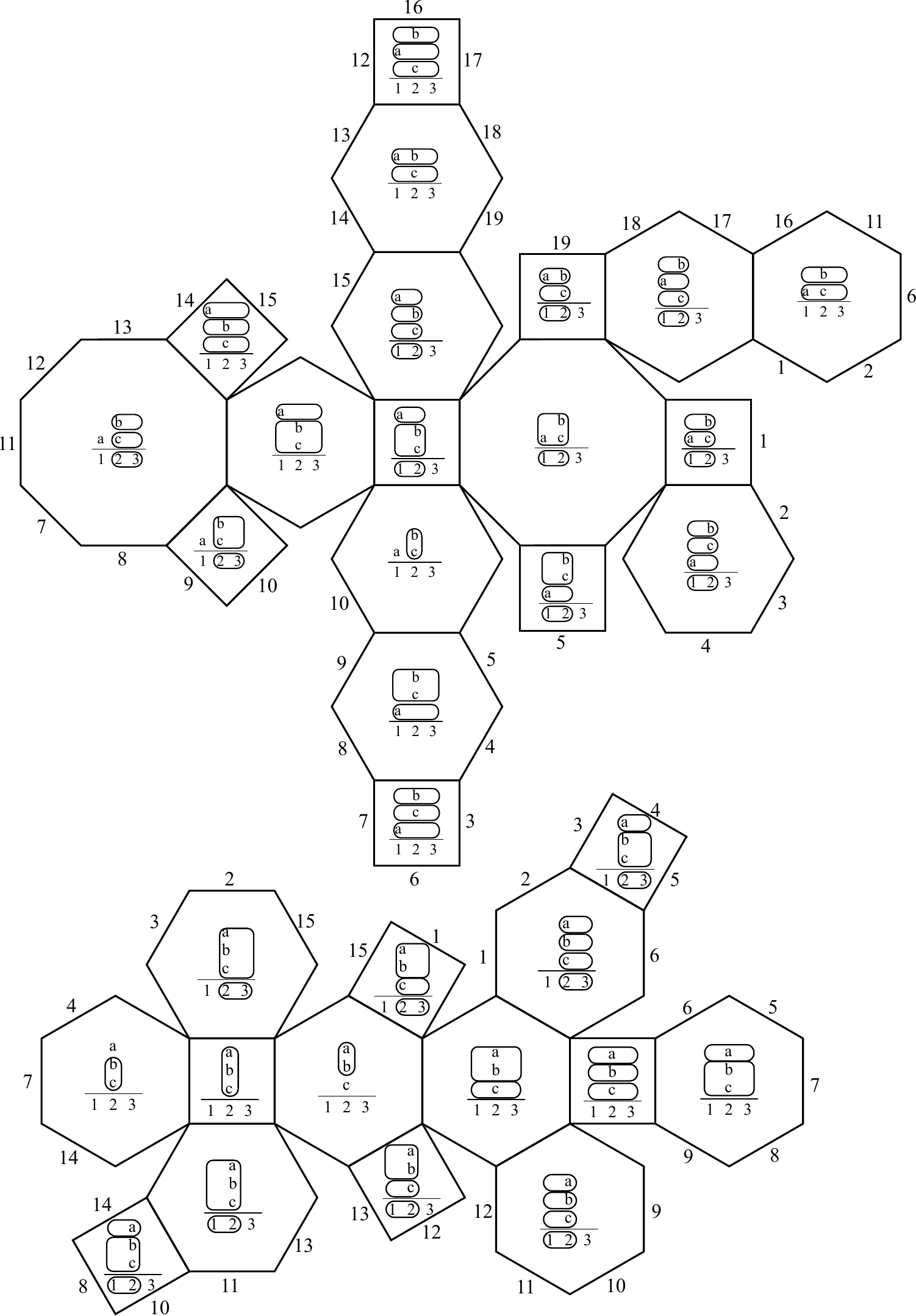}
\end{figure}

\newpage

\begin{figure}[H]
\centering
\includegraphics[width=0.9\columnwidth]{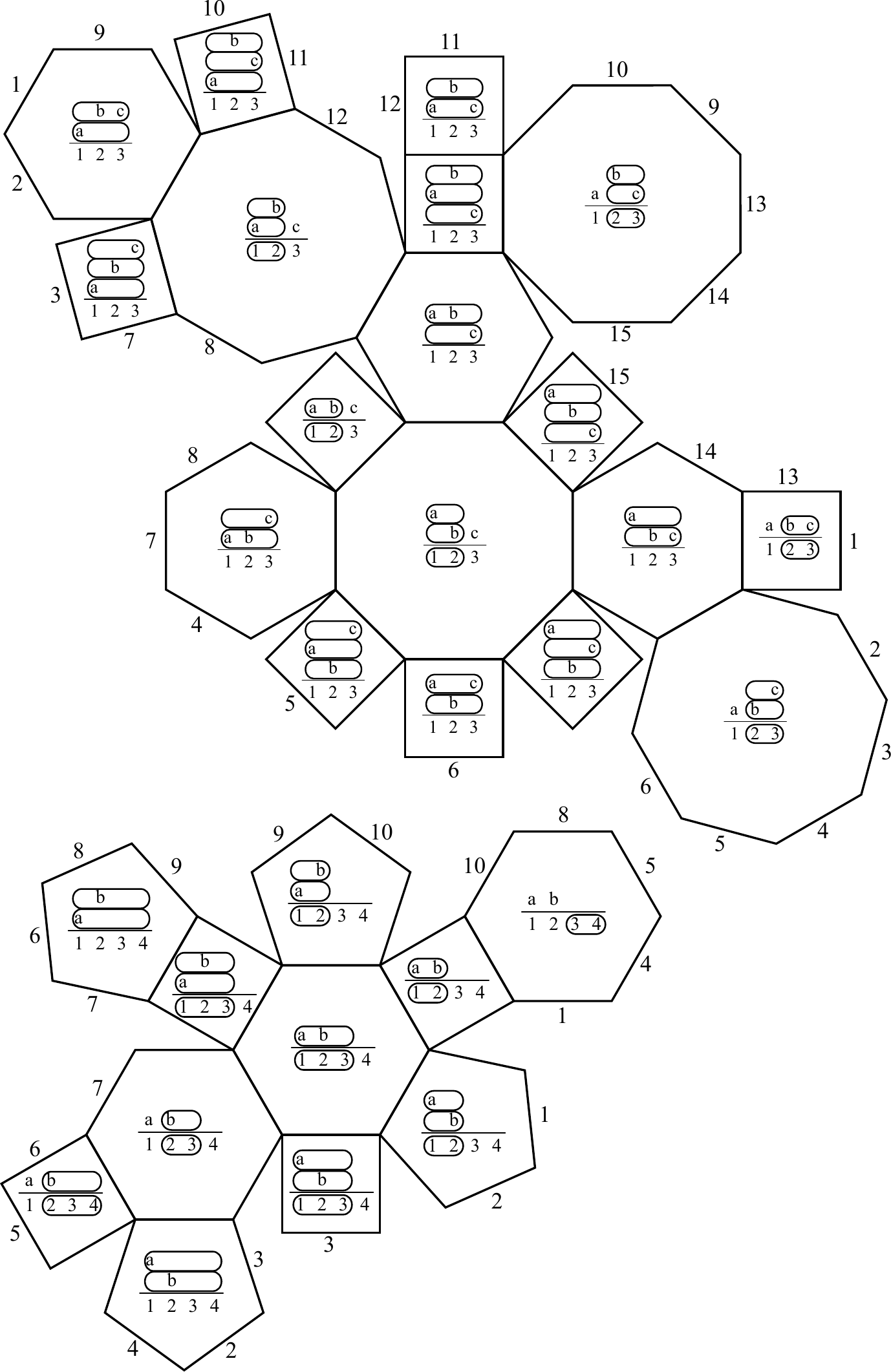}
\end{figure}

\newpage

\begin{figure}[H]
\centering
\includegraphics[width=0.78\columnwidth]{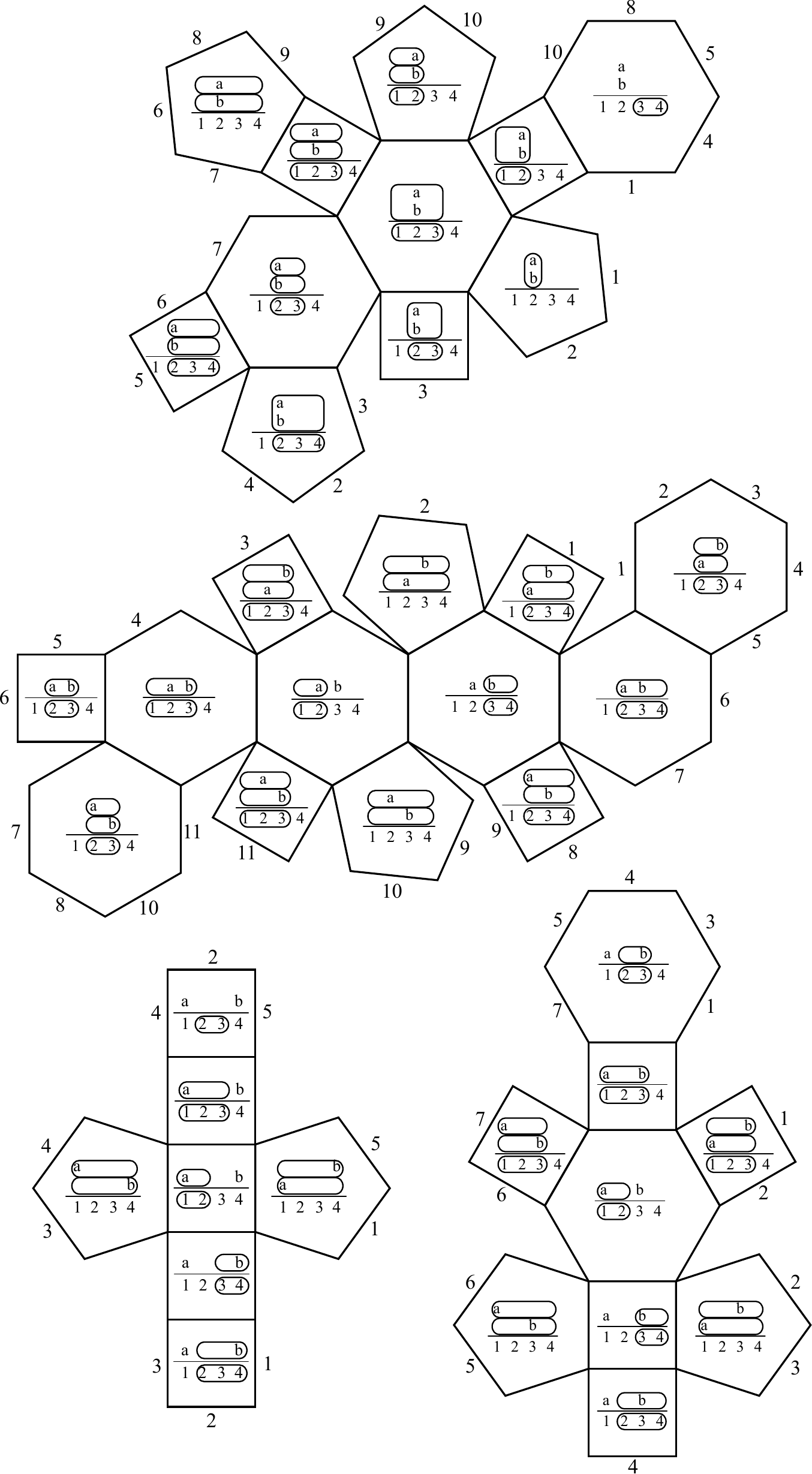}
\end{figure}

\end{document}